\documentclass[reqno]{amsart}

\usepackage{amsmath,amssymb,amsthm,amsfonts}
\usepackage{hyperref}
\usepackage{appendix}
\usepackage{graphicx}
\usepackage{setspace}

\usepackage{color}

\newtheorem{lemma}{Lemma}[section]
\newtheorem{theorem}{Theorem}[section]

\newtheorem{proposition}{Proposition}[section]
\newtheorem{remark}{Remark}[section]
\newtheorem{corollary}{Corollary}[section]

\numberwithin{equation}{section}

\begin{document}
\title[Boundary layer ansatz for the steady MHD equations]{Verification of Prandtl boundary layer ansatz for the steady electrically conducting fluids with a moving physical boundary}
\thanks{$^*$Corresponding author}
\thanks{{\it Keywords}: Prandtl boundary layer expansion; steady viscous incompressible magnetohydrodynamics; High Reynolds number limit; Sobolev spaces.}
\thanks{{\it AMS Subject Classification}: 76N10, 35Q30, 35R35}%
\author[Shijin Ding]{Shijin Ding}
\address[S. Ding]{South China Research Center for Applied Mathematics and Interdisciplinary Studies, South China Normal University,
Guangzhou, 510631, China}\address{School of Mathematical Sciences, South China Normal University,
Guangzhou, 510631, China}
\email{dingsj@scnu.edu.cn}
\author[Zhilin Lin]{Zhilin Lin}
\address[Z. Lin]{School of Mathematical Sciences, South China Normal University,
Guangzhou, 510631, China}
\email{zllin@m.scnu.edu.cn}
\author[Feng Xie]{Feng Xie$^*$}
\address[Corresponding author: F. Xie]{School of Mathematical Sciences and LSC-MOE, Shanghai Jiao Tong University,
Shanghai, 200240, China}
\email{tzxief@sjtu.edu.cn}


\begin{abstract}
In this paper, we are concerned with the validity of Prandtl boundary layer expansion for the solutions to two dimensional (2D) steady viscous incompressible magnetohydrodynamics (MHD) equations in a domain $\{(X, Y)\in[0, L]\times\mathbb{R}_+\}$ with a moving flat boundary $\{Y=0\}$. As a direct consequence, even though there exist strong boundary layers, the inviscid type limit is still established for the solutions of 2D steady viscous incompressible MHD equations in Sobolev spaces provided that the following three assumptions hold: the hydrodynamics and magnetic Reynolds numbers take the same order in term of the reciprocal of a small parameter $\epsilon$, the tangential component of the magnetic field does not degenerate near the boundary and the ratio of the  strength of tangential component of  magnetic field and  tangential component of velocity is suitably small. And the error terms are estimated in $L^\infty$ sense.
\end{abstract}

\maketitle

\vspace{-5mm}

\section{Introduction}
In this paper, we consider the Prandtl boundary layer expansion for 2D incompressible viscous electrically conducting fluid with high Reynolds numbers in a domain with a moving flat boundary, which is described by the following 2D steady incompressible viscous MHD equations in the domain $(X,Y)\in\Omega:=[0,L]\times[0,+\infty)$:
\begin{equation}\label{1.1}
\left \{
\begin{array}{lll}
(U \partial_X +V\partial_Y)U-(H \partial_X +G\partial_Y)H+\partial_X P=\nu\epsilon (\partial_{XX}+\partial_{YY})U,\\
(U \partial_X +V\partial_Y)V-(H \partial_X +G\partial_Y)G+\partial_Y P=\nu\epsilon (\partial_{XX}+\partial_{YY})V,\\
(U \partial_X +V\partial_Y)H-(H \partial_X +G\partial_Y)U=\kappa\epsilon (\partial_{XX}+\partial_{YY})H,\\
(U \partial_X +V\partial_Y)G-(H \partial_X +G\partial_Y)V=\kappa\epsilon (\partial_{XX}+\partial_{YY})G,\\
\partial_XU+\partial_YV=0,\qquad \partial_XH+\partial_YG=0,
\end{array}
\right.
\end{equation}
where $(U, V)$ and $(H, G)$ are velocity and magnetic field respectively. The viscosity and resistivity coefficients take the same order of a small parameter $\epsilon$. $\nu$ and $\kappa$ are two positive constants. The boundary condition of velocity is imposed as follows:
\begin{align}
\label{BCV}
(U, V)|_{Y=0}=(u_b,0),
\end{align}
with $u_b>0$. It shows that the physical boundary is non-penetration and moves with a speed of $u_b$. The magnetic field satisfies the perfect conducting boundary condition on the boundary:
\begin{align}
\label{BCB}
(\partial_YH, G)|_{Y=0}=(0,0).
\end{align}
We are interested in the inviscid type limit problem of (\ref{1.1}). Precisely, when the small parameter $\epsilon$ in (\ref{1.1}) goes to zero, what is the exact limit state of solutions to (\ref{1.1})-(\ref{BCB}). Such a kind of limit process is one of central problems in magnetic hydrodynamics. Moreover, the mathematical analysis of this singular limit process is extremely challenging due to the appearance of strong boundary layers.

Formally, let $\epsilon=0$, the system (\ref{1.1}) is reduced into the ideal MHD equations:
\begin{equation}\label{IMHD}
\left \{
\begin{array}{lll}
(U_0 \partial_X +V_0\partial_Y)U_0-(H_0 \partial_X +G_0\partial_Y)H_0+\partial_X P_0=0,\\
(U_0 \partial_X +V_0\partial_Y)V_0-(H_0 \partial_X +G_0\partial_Y)G_0+\partial_Y P_0=0,\\
(U_0 \partial_X +V_0\partial_Y)H_0-(H_0 \partial_X +G_0\partial_Y)U_0=0,\\
(U_0 \partial_X +V_0\partial_Y)G_0-(H_0 \partial_X +G_0\partial_Y)V_0=0,\\
\partial_XU_0+\partial_YV_0=0,\qquad \partial_XH_0+\partial_YG_0=0.
\end{array}
\right.
\end{equation}
For the well-posedness of solutions to this ideal MHD equations (\ref{IMHD}), it suffices to impose non-penetration conditions for both velocity and magnetic field on the boundary $\{Y=0\}$. That is,
\begin{align*}
V_0|_{Y=0}=0,\qquad G_0|_{Y=0}=0.
\end{align*}
In general, there exists an obvious mismatch between the tangential components $(U, H)$ and $(U_0, H_0)$ in the study of inviscid limit process from (\ref{1.1}) to (\ref{IMHD}). According to the Prandtl boundary layer ansatz \cite{Prandtl}, it is necessary to introduce the boundary layer corrector functions.  These boundary layer functions change from the boundary conditions of tangential components $(U, H)$ in (\ref{BCV}) to the trace of $(U_0, H_0)$ on the boundary $\{Y=0\}$ in a thin layer near the physical boundary with width of $\sqrt{\epsilon}$. Moreover, the boundary layer functions decay to zero rapidly outside this thin layer. It is obvious that the vorticity of these boundary layer functions becomes uncontrollable when $\epsilon$ goes to zero. This is indeed one of the essential difficulties to achieve the vanishing viscosity limit problem with strong boundary layers. The main goal of this paper is to study the inviscid type limit problem of (\ref{1.1})-(\ref{BCB}). To this end, we introduce the Prandtl fast variables:
\begin{align*}
x=X,\qquad y=\frac{Y}{\sqrt{\epsilon}},
\end{align*}
and the new unknowns:
\begin{equation}\label{1.5}
\left \{
\begin{array}{lll}
U^\epsilon(x,y)=U(X,Y),\ \ V^\epsilon(x,y)=\frac{1}{\sqrt{\epsilon}}V(X,Y),\\
H^\epsilon(x,y)=H(X,Y),\ \ G^\epsilon(x,y)=\frac{1}{\sqrt{\epsilon}}G(X,Y),\\
P^\epsilon(x,y)=P(X,Y).
\end{array}
\right.
\end{equation}
Rewrite the viscous MHD equations (\ref{1.1}) in the Prandtl fast variables as follows:
\begin{equation}\label{1.6}
\left \{
\begin{aligned}
&(U^\epsilon\partial_x +V^\epsilon\partial_y)U^\epsilon-(H^\epsilon\partial_x+G^\epsilon\partial_y)H^\epsilon+\partial_x P^\epsilon=\nu\epsilon \partial_{xx}U^\epsilon+\nu\partial_{yy}U^\epsilon,\\
&(U^\epsilon\partial_x +V^\epsilon\partial_y)V^\epsilon-(H^\epsilon\partial_x+G^\epsilon\partial_y)G^\epsilon+\frac{\partial_y P^\epsilon}{\epsilon}=\nu\epsilon \partial_{xx}V^\epsilon+\nu\partial_{yy}V^\epsilon,\\
&(U^\epsilon\partial_x +V^\epsilon\partial_y)H^\epsilon-(H^\epsilon\partial_x+G^\epsilon\partial_y)U^\epsilon=\kappa\epsilon \partial_{xx}H^\epsilon+\kappa\partial_{yy}H^\epsilon,\\
&(U^\epsilon\partial_x +V^\epsilon\partial_y)G^\epsilon-(H^\epsilon\partial_x+G^\epsilon\partial_y)V^\epsilon=\kappa\epsilon \partial_{xx}G^\epsilon+\kappa\partial_{yy}G^\epsilon,\\
&\partial_xU^\epsilon+\partial_yV^\epsilon=0,\qquad \partial_xH^\epsilon+\partial_yG^\epsilon=0.
\end{aligned}
\right.
\end{equation}
In this paper, we only focus on the case that the inner ideal MHD flow is a special but important shear flow. Precisely, the solutions to (\ref{IMHD}) take the following form:
\begin{equation}\label{1.2}
(U_0(Y),0,H_0(Y),0),\qquad (\partial_X P_0, \partial_Y P_0)=(0, 0).
\end{equation}
According to the classical Prandtl boundary layer ansatz, the solutions to (\ref{1.6}) should be decomposed into the following three parts:
\begin{align*}
(U^\epsilon, V^\epsilon, H^\epsilon, G^\epsilon)=&(U^0(\sqrt{\epsilon}y),0,H^0(\sqrt{\epsilon}y),0)\\
&+(u(x,y), v(x,y), h(x,y), g(x,y))+o(1),
\end{align*}
where $(u,v,h,g)(x,y)$ are the leading order boundary layer corrector functions. And the error terms $o(1)$ tend to zero, as $\epsilon$ goes to zero.

To justify the validity of above Prandtl boundary layer ansatz, it is usual to construct a higher order approximation solution to (\ref{1.6}). Precisely, we search for the solutions $(U^\epsilon, V^\epsilon, H^\epsilon, G^\epsilon,P^\epsilon)$ to the scaled viscous MHD equations (\ref{1.6}) in the following form:
\begin{equation}\label{1.8}
\begin{array}{lll}
(U^\epsilon,V^\epsilon,H^\epsilon,G^\epsilon,P^\epsilon)=
(u_{app},v_{app},h_{app},g_{app},p_{app})+\epsilon^{\frac{1}{2}+\gamma}(u^\epsilon,v^\epsilon,h^\epsilon,g^\epsilon,p^\epsilon)
\end{array}
\end{equation}
for some $\gamma>0$, where
\begin{equation}\label{1.9}
\left \{
\begin{array}{lll}
u_{app}=u^0_e(\sqrt{\epsilon}y)+u^0_p(x,y)+\sqrt{\epsilon}\left(u^1_e(x,\sqrt{\epsilon}y)+u^1_p(x,y)\right),\\
v_{app}=v^0_p(x,y)+v^1_e(x,\sqrt{\epsilon}y)+\sqrt{\epsilon}v^1_p(x,y),\\
h_{app}=h^0_e(\sqrt{\epsilon}y)+h^0_p(x,y)+\sqrt{\epsilon}\left(h^1_e(x,\sqrt{\epsilon}y)+h^1_p(x,y)\right),\\
g_{app}=g^0_p(x,y)+g^1_e(x,\sqrt{\epsilon}y)+\sqrt{\epsilon}g^1_p(x,y),\\
p_{app}=\sqrt{\epsilon}\left(p^1_e(x,\sqrt{\epsilon}y)+p^1_p(x,y)\right)+\epsilon p^2_p(x,y),
\end{array}
\right.
\end{equation}
in which $(u^j_e,v^j_e,h^j_e,g^j_e,p^j_e)$ and $(u^j_p,v^j_p,h^j_p,g^j_p,p^j_p)$ with $j=0,1$ denote the inner flows and boundary layer profiles respectively. It should be emphasized that the leading order inner flow $(u^0_e,v^0_e,h^0_e,g^0_e,p^0_e)=(U_0(Y),0,H_0(Y),0,0)$ and the inner flows are always evaluated at $(x,Y)=(x,\sqrt{\epsilon}y)$, while the boundary layer profiles are at $(x,y)$. And $(u^\epsilon,v^\epsilon,h^\epsilon,g^\epsilon,p^\epsilon)$ stand for the error terms.

\subsection{Boundary conditions}\label{subsec1.1}

To construct the approximate solutions in (\ref{1.9}), it is necessary to determine the boundary conditions for the profiles $(u^j_e,v^j_e,h^j_e,g^j_e,p^j_e)$ and $(u^j_p,v^j_p,h^j_p,g^j_p,p^j_p)\ (j=0,1)$ in term of the order of $\sqrt{\epsilon}$.

(I) Since the zeroth-order inner ideal MHD flows $(u^0_e,0,h^0_e,0)$ are given, the boundary values for the zeroth-order inner flow are $u_e=u^0_e(0),\ h_e=h^0_e(0)$, which are different from the ``slip" boundary condition of velocity $U^\epsilon(x,0)=u_b$ in (\ref{BCV}) and the perfectly conducting boundary condition $\partial_Y H^\epsilon(x,0)=0$ in (\ref{BCB}). By the first equality in (\ref{1.9}), we impose the boundary conditions for the zeroth-order boundary layer profile $u^0_p(x,y)$ as follows:
$$u_e+u^0_p(x,0)=u_b,  \ \ \ \ \lim_{y\to \infty} u^0_p=0.$$
And the boundary conditions of the zeroth-order boundary layer profile $h^0_p(x,y)$ are given by
$$\partial_y h^0_p(x,0)=0, \ \ \ \ \lim_{y\to \infty} h^0_p=0.$$
For the vertical components of velocity and magnetic fields, the boundary conditions satisfy that
$$v^1_e(x,0)+v^0_p(x,0)=g^1_e(x,0)+g^0_p(x,0)=0.$$
Since $u_b>0$, the $x$ variable direction can be regarded as a ``time" variable. So, the zeroth-order boundary layer profiles satisfy a parabolic-type system of equations. To solve this system, the ``initial data" (in fact, the value on the boundary $\{x=0\}$) are also needed:
$$(u^0_p,h^0_p)(0,y)=(\overline{u}_0(y),\overline{h}_0(y)).$$

(II) For the first-order inner ideal MHD flows, the profiles obey an elliptic system, see Subsection \ref{subsec4.1}. For this, we impose the following boundary conditions:
\begin{equation}\nonumber
\left \{
\begin{aligned}
&(u^1_e,h^1_e)(0,Y)=(u^1_b,h^1_b)(Y),\\
&(v^1_e,g^1_e)(x,0)=-(v^0_p,g^0_p)(x,0),\\
&(v^1_e,g^1_e)(0,Y)=(V_{b0},G_{b0})(Y),\\
&(v^1_e,g^1_e)(L,Y)=(V_{bL},G_{bL})(Y),\\
&(v^1_e,g^1_e) \to (0,0) \ \ \mathrm{as} \ \ Y\to \infty,
\end{aligned}
\right.
\end{equation}
with compatibility conditions at corners
\begin{equation}\nonumber
\left \{
\begin{aligned}
&(V_{b0},G_{b0})(0)=-(v^0_p,g^0_p)(0,0),\\
&(V_{bL},G_{bL})(0)=-(v^0_p,g^0_p)(L,0).
\end{aligned}
\right.
\end{equation}

(III) The first-order boundary profiles satisfy a linearized MHD boundary layer system of equations, which is a also parabolic type system. The boundary conditions on $\{y=0\}$ and the ``initial data" on $\{x=0\}$ we impose are listed as follows:
\begin{equation}\nonumber
\left \{
\begin{aligned}
&(u^1_p,h^1_p)(0,y)=(\overline{u}_1,\overline{h}_1)(y),\\
&(u^1_p,\partial_y h^1_p)(x,0)=-(u^1_e,\partial_Y h^0_e)(x,0),\\
&(v^1_p,g^1_p)(x,0)=(0,0),\\
&(u^1_p,h^1_p) \to (0,0) \ \ \ \mathrm{as} \ \ y \to \infty.
\end{aligned}
\right.
\end{equation}

(IV) As a consequence of the construction above and the Prandtl boundary layer expansion (\ref{1.8})-(\ref{1.9}), the boundary conditions for the remainder $(u^\epsilon,v^\epsilon,h^\epsilon,g^\epsilon)$ are thus given as
$$(u^\epsilon,v^\epsilon,\partial_y h^\epsilon,g^\epsilon)|_{y=0}=(0,0,0,0),\ (u^\epsilon,v^\epsilon, h^\epsilon,g^\epsilon)|_{x=0}=(0,0,0,0),$$
$$p^\epsilon-2\nu\epsilon \partial_x u^\epsilon=0, \ \partial_y u^\epsilon+\nu\epsilon \partial_x v^\epsilon=h^\epsilon=\partial_x g^\epsilon=0\ \ \mathrm{on} \ \ \{x=L\}.$$
It should be emphasized that the boundary condition of the magnetic field is still the perfectly conducting boundary condition due to the construction of $h_e^1$ in Subsection \ref{IMP}.
\subsection{Main result and comments}
\

The main results in this paper are stated as follows.
\begin{theorem}\label{mainresult}
Let $ u_b>0$ be a constant tangential velocity of the viscous MHD flow on the boundary $\{Y=0\}$ and $u^0_e(Y),h^0_e(Y)$ be the given smooth positive inner ideal MHD flows such that $\partial_Y u^0_e,\partial_Yh^0_e$ and their derivatives decay exponentially fast to zero at infinity. Assume that the data $u^1_b,h^1_b,V_{b0},V_{bL},G_{b0},G_{bL}$ and $\overline{u}_0,\overline{h}_0,\overline{u}_1,\overline{h}_1$ are smooth and decay exponentially fast at infinity in their arguments. Moreover, suppose that $|(V_{bL}-V_{b0},G_{bL}-G_{b0})(Y)|\lesssim L$ for some small $L>0$, $\Vert   \langle Y  \rangle \partial_Y (u^0_e, h^0_e)\Vert_{L^\infty}<\delta_0$ for some suitably small $\delta_0>0$, and $u^0_e(Y)>>h^0_e(Y)$ uniformly in $Y$, and
\begin{equation}\label{main0}
u_e+\overline{u}_0(y)>h_e+\overline{h}_0(y)\geq \vartheta_0>0
\end{equation}
uniform in $y$ for some $\vartheta_0>0$, also, if
\begin{equation}\label{main1}
\begin{aligned}
&|u_e^0(\sqrt{\epsilon}y)+\overline{u}_0(y)|>>|h_e^0(\sqrt{\epsilon}y)+\overline{h}_0(y)|,\\
&| \langle y\rangle^{l+1}\partial_y (u_e+\overline{u}_0,h_e+\overline{h}_0)(y)| \leq \frac{1}{2}\sigma_0,\\
&| \langle y\rangle^{l+1}\partial_y^2 (u_e+\overline{u}_0,h_e+\overline{h}_0)(y)| \leq \frac{1}{2}\vartheta_0^{-1},
\end{aligned}
\end{equation}
uniform in $y$, here $\sigma_0>0$ is another suitably small constant. Then, there exists a constant $L_0>0$ that depends only on the given data such that the boundary layer expansion (\ref{1.8}) holds on $[0,L]\times [0,+\infty)$ for $\gamma \in (0,\frac{1}{4})$ and $0< L \leq L_0$. Furthermore, the following estimate holds:
\begin{equation}\label{main2}
\begin{aligned}
&\Vert \nabla_\epsilon u^\epsilon\Vert_{L^2}+\Vert \nabla_\epsilon v^\epsilon\Vert_{L^2}+\Vert \nabla_\epsilon h^\epsilon\Vert_{L^2}+\Vert \nabla_\epsilon g^\epsilon\Vert_{L^2}\\
&+\epsilon^{\frac{\gamma}{2}}\Vert u^\epsilon\Vert_{L^\infty}+\epsilon^{\frac{\gamma}{2}+\frac{1}{2}}\Vert v^\epsilon\Vert_{L^\infty}+\epsilon^{\frac{\gamma}{2}}\Vert h^\epsilon\Vert_{L^\infty}+\epsilon^{\frac{\gamma}{2}+\frac{1}{2}}\Vert g^\epsilon\Vert_{L^\infty}\leq C_0,
\end{aligned}
\end{equation}
where the constant $C_0>0$ depends only on the given data. Here, $\nabla_\epsilon=(\sqrt{\epsilon} \partial_x,\partial_y)$ and $\Vert \cdot \Vert_{L^p}$ denotes the usual $L^p$ norm over $[0,L]\times [0,\infty)$.
\end{theorem}

With the above theorem in hand, we have the following corollary.
\begin{corollary}
Under the assumptions in Theorem \ref{mainresult}, there exists a solution $(U,V,H,G)$ to the original viscous MHD equations (\ref{1.1}) in $\Omega=[0,L]\times [0, \infty)$ with $L>0$ being as what is described in Theorem \ref{mainresult}, so that
\begin{equation}\label{main3}
\begin{aligned}
\sup_{(X,Y)\in \Omega}&\left|U(X,Y)-u^0_e(Y)-u^0_p\left(X,\frac{Y}{\sqrt{\epsilon}}\right)\right|\lesssim \sqrt{\epsilon},\\
\sup_{(X,Y)\in \Omega}&\left|V(X,Y)-\sqrt{\epsilon}v^0_p\left(X,\frac{Y}{\sqrt{\epsilon}}\right)-\sqrt{\epsilon}v^1_e(X,Y)\right|\lesssim \epsilon^{\frac{\gamma}{2}+\frac{1}{2}},\\
\sup_{(X,Y)\in \Omega}&\left|H(X,Y)-h^0_e(Y)-h^0_p\left(X,\frac{Y}{\sqrt{\epsilon}}\right)\right|\lesssim \sqrt{\epsilon},\\
\sup_{(X,Y)\in \Omega}&\left|G(X,Y)-\sqrt{\epsilon}g^0_p\left(X,\frac{Y}{\sqrt{\epsilon}}\right)-\sqrt{\epsilon}g^1_e(X,Y)\right|\lesssim \epsilon^{\frac{\gamma}{2}+\frac{1}{2}},
\end{aligned}
\end{equation}
for $\epsilon \in [0, \epsilon_0)$ with $\epsilon_0$ being suitably small, where $(u^0_e,0,h^0_e,0)$ is the given inner ideal MHD flows, $(v^1_e, g^1_e)$  and $(u^0_p,\sqrt{\epsilon}v^0_p,h^0_p,\sqrt{\epsilon}g^0_p)$ are the the ideal MHD flows and boundary layer profiles constructed in approximate solution (\ref{1.9}).
\end{corollary}

\begin{remark}\label{rk1}
The condition that $u^0_e(Y)>>h^0_e(Y)$ uniform in $Y$ implies that
$$\sup_Y \left|\frac{h^0_e}{u^0_e}\right|<<1.$$
This is an important condition in the analysis. First, it ensures the elliptic system of the first-order ideal MHD correctors is non-degenerate, see Subsection \ref{subsec4.1} for details. Second, it play a key role to achieve the uniform estimates of the first-order ideal MHD correctors $(u^1_e,v^1_e,h^1_e,g^1_e,p^1_e)$, see Subsection \ref{subsec4.1} for details.
\end{remark}

\begin{remark}\label{rk2}
The condition (\ref{main0}), $u_e+\overline{u}_0(y)>h_e+\overline{h}_0(y)$ uniform in $y$,
is needed to guarantee that the systems, such as the MHD boundary layer equations, are forward ``parabolic" system. Without magnetic diffusion, Wang and Ma \cite{JWang} showed the steady MHD boundary layer equations have no global-in-$x$ variable solution when the strength of magnetic field is larger than the the strength of velocity. At this moment, we have no idea that this condition is needed only for technical reason, or it is an essential one.

In addition, the condition that $|u_e^0(\sqrt{\epsilon}y)+\overline{u}_0(y)|>>|h_e^0(\sqrt{\epsilon}y)+\overline{h}_0(y)|$ uniform in $y$ will imply the following conclusion:
\begin{align}\label{RM1}\left\Vert \frac{h_s}{u_s}\right\Vert_{L^\infty}<<1,
\end{align}
provided that $0<L<<1$. Here $u_s,\ h_s$ are approximate solutions defined in Section \ref{sec3}. And the condition (\ref{RM1}) is needed in the proof of the validity of the Prandtl layer ansatz. It is noted that the first condition in (\ref{main1}) is satisfied by two different cases: (1) the strength of tangential magnetic field is small enough; (2) the strength of tangential velocity is much larger than the strength of tangential magnetic field.
\end{remark}
Before proceeding, we first review some related works on the Prandtl boundary layer theories and high Reynolds number limit problem in fluid dynamics. In fact, without the magnetic field $(H, G)$ in \eqref{1.1},  the system is reduced into the two dimensional steady incompressible Navier-Stokes equations. As is well-known that the rigorous justification of the Prandtl boundary layer expansion for unsteady Navier-Stokes equations has a long history. The key issue is to verify the
solution to the Navier-Stokes equations can be written as the superposition of
solutions to Euler and Prandtl systems with small error term in the study of vanishing viscosity limit. To our knowledge, the rigorous mathematical results were
achieved only in some infinite regularity function spaces, for example, in the analytic framework \cite{Sammartino2, Wang, Maekawa} and in Gevrey class \cite{GMM} recently.
However, for the steady incompressible Navier-Stokes equations, there are some satisfactory and interesting results about the validity of Prandtl boundary layer ansatz in Sobolv spaces. This kind of result was initiated by Guo and  Nguyen in \cite{YGuo2}. Under the assumption that the physical boundary is moving with a speed of $u_b>0$, then error terms can be estimated in Sobolev spaces by finding some positive estimates for the linearized Navier-Stokes operator. And this result was extended to the rotating disk case in \cite{Iyer1}. Very recently, the moving boundary condition was removed by Guo and Iyer in \cite{YGuo1} by using the classical Blasius profile as the leading order boundary layer corrector function. At the same time, Ger\'ard and Maekawa also established the convergence result for steady Navier-Stokes equations with some additional force term for no-slip boundary boundary case in \cite{Gerard2}. Li and the first author in \cite{Li} studied the Prandtl boundary layer expansions of steady Navier-Stokes equations on bounded domain. In addition, Iyer established a $x$-global steady Prandtl expansion with a moving boundary under the assumption that the Euler flow is $(1,0)$ in a series of works \cite{Iyer2,Iyer3,Iyer4}.

Motivated by the fifteenth open problem in Oleinik-Samokhin's classical book \cite{Oleinik1} (page 500-503),
{\it ``15. For the equations of the magnetohydrodynamic boundary layer, all problems of the above type are still open,"}
efforts have been made to study the well-posedness of solutions to the unsteady MHD boundary layer equations and to justify the MHD boundary layer expansion for the two dimensional unsteady MHD equations in \cite{CLiu2, CLiu3}. Precisely, when the hydrodynamic  and magnetic Reynolds numbers have the same order, the well-posedness of solutions to the unsteady MHD boundary layer equations and the validity of the Prandtl ansatz were established without any monotone condition imposed on the velocity in \cite{CLiu2, CLiu3}. Also refer to \cite{G-P} for the derivation of MHD boundary layer equations in different physical regime.  The long-time existence of solutions to the unsteady MHD boundary layer equations in analytic settings was also studied in \cite{X-Y2}.

The inviscid type limit problem of the steady MHD equations (\ref{1.1})-(\ref{BCB}) is considered in this paper, it is necessary to compare the assumptions and main results in this paper with the case of steady incompressible Navier-Stokes equations in \cite{YGuo2} and the case of unsteady MHD equations in \cite{CLiu2, CLiu3}, which include the following four main aspects:

(a) For the well-posedness of solutions to the zeroth-order boundary layer profiles $(u^0_p,v^0_p,h^0_p,g^0_p)$, which satisfy a nonlinear parabolic system with nonlocal terms. The main difficulty in analysis in Sobolev spaces lies in the terms of loss of $x$-derivatives. This feature is similar as that of the unsteady Prandtl equations. Unlike the idea in \cite{YGuo2}, whose authors used
the classical von Mises transformation for the steady Prandtl equations, here, we apply a modified energy method to deal with steady MHD boundary layer equations directly, instead of searching for a nonlinear coordinate transformation to reduce the MHD boundary layer equations. This method is inspired by the work due to Liu, Yang and the third author in \cite{CLiu2}. We can also find a function transformation to cancel the terms including the loss of $x$-derivatives by introducing the equation of stream function for the magnetic fields. And the condition that the tangential component of magnetic field has a lower positive bound is essentially used to ensure the cancellation functions are well-defined. See Subsection \ref{0th} and Appendix \ref{ap1} for more details.

(b) As we will see in the Subsection \ref{subsec4.1}, the first-order inner ideal MHD flows $(u^1_e,v^1_e,h^1_e,g^1_e)$ obey a coupled elliptic system and the condition in Theorem \ref{mainresult} that $u^0_e(Y)>>h^0_e(Y)$ uniform in $Y$ guarantees that the system is non-degenerate. It is remarked that the positive type estimate in \cite{YGuo2} can not be obtained directly here by applying the vorticity formulations for both velocity and magnetic fields as in \cite{YGuo2}. The reason is due to the coupling effects of velocity and magnetic fields. To overcome the difficulties, one key observation is that the third equation in (\ref{4.09}) for the tangential magnetic field can be rewritten as
$$\partial_Y(u^0_eg^1_e-h^0_ev^1_e)=0,$$
which gives that
\begin{equation}\nonumber
g_e^1=\frac{h^0_e}{u_e^0}v^1_e+\frac{1}{u^0_e}(h_e\overline{v^0_p}-u_e\overline{g^0_p}),
\end{equation}
where the notation of $\overline{f}:=f(x,0)$ denotes the trace of a function $f(x,y)$ on the boundary $\{y=0\}$.

This yields a direct algebraic relation between $v^1_e$ and $g^1_e$. Consequently, the following intrinsic structure can be used:
\begin{equation}\nonumber
\left \{
\begin{aligned}
&-u_e^0\Delta v^1_e+\partial_Y^2 u^0_e\cdot v^1_e+\left(h^0_e\Delta g^1_e-\partial_Y^2h^0_e \cdot g^1_e\right)=\cdots,\\
&g_e^1=\frac{h^0_e}{u_e^0}v^1_e+\frac{1}{u^0_e}(h_e\overline{v^0_p}-u_e\overline{g^0_p}).
\end{aligned}
\right.
\end{equation}
In this way, the key part of the above formulation still has following positive operator part:
$$-u_e^0\Delta +\partial_Y^2 u^0_e \ \ , $$
provided that $\left|\frac{h^0_e}{u^0_e}\right|$ is sufficiently small. This is exact the reason why we can obtain the positivity estimates for MHD case. Moreover, for the higher order weighted estimates, one can modify the arguments slightly to establish the desired estimates. It should be pointed out that the key structure and positive part is preserved in the formulation for higher order weighted estimates. Refer to Subsection \ref{subsec4.1} for details.

(c) For the equations of first-order boundary layer profiles $(u^1_p,v^1_p,h^1_p,g^1_p)$, it is also difficult to apply the formulation proposed in \cite{YGuo2} due to the fact that the equation of magnetic fields will destroy the structure (the fourth order equation) introduced in \cite{YGuo2}. Therefore, with the help of the stream function of the magnetic field similar as in (a) above, we avoid using the fourth PDE formulation of \cite{YGuo2}. See Subsection \ref{1stmhdbdprofile} for more details.

(d) The key estimates for equations of the remainder terms rest with the positive estimates for $\nabla_\epsilon v^\epsilon,\nabla_\epsilon g^\epsilon$. Under the small assumptions imposed on the ratio of magnetic field and velocity, the positive estimates as in \cite{YGuo2} also can be achieved for the MHD equations. See Section \ref{sec3} for more details (also refer to Remark \ref{rk2}).

The rest of this paper is organized as follows: In Section \ref{sec2}, we will construct a
suitable approximation solution and derive some key estimates. In Section \ref{sec3}, the error terms are estimated in $L^\infty$-norm
for the proof of Theorem \ref{mainresult}. Finally, to make the paper self-contained, we will provide several key proofs and computations in the Appendix \ref{ap1}-- \ref{ap3}.

\section{Construction of approximate solution}\label{sec2}
The strategy to prove Theorem \ref{mainresult} usually includes the following two steps: (1) Construct a suitable approximate solution in the form of (\ref{1.9}), which satisfies the equations (\ref{1.6}) with high order error terms of $\epsilon$; (2) Suppose the viscous solutions to (\ref{1.6}) can be written as a superposition of approximate solution and remainder terms, it suffices to estimate the remainder terms in $L^\infty$-norm.

In this section, the approximate solution will be constructed step by step: (a) Plug the form of (\ref{1.9}) into the system of equations (\ref{1.6}); (b) Compare the order of $\sqrt{\epsilon}$, we obtain the systems of ideal inner flows and the systems of boundary layer functions for each order of $\sqrt{\epsilon}$; (c) Prove the well-posedness of solution to each system of ideal inner flows and system of boundary layer functions, and derive the key estimates of solutions.  Based on the estimates achieved in (c), the error terms can be estimated in $L^\infty$ sense.

Plugging the ansatz (\ref{1.8}) with (\ref{1.9}) into the scaled viscous MHD equations (\ref{1.6}), and match the order of $\epsilon$ to determine the equations of the corresponding profiles. With the definition given in (\ref{1.9}), one can calculate the error caused by the approximation solutions as follows.
\begin{equation}\label{2.5}
\left \{
\begin{aligned}
R^{1}_{app}=&[u_{app}\partial_x+v_{app}\partial_y]u_{app}+\partial_x p_{app}\\
&-[h_{app}\partial_x+g_{app}\partial_y]h_{app}-\nu\Delta_\epsilon u_{app},\\
R^{2}_{app}=&[u_{app}\partial_x+v_{app}\partial_y]v_{app}+\frac{1}{\epsilon}\partial_y p_{app}\\
&-[h_{app}\partial_x+g_{app}\partial_y]g_{app}-\nu\Delta_\epsilon v_{app},\\
R^3_{app}=&[u_{app}\partial_x+v_{app}\partial_y]h_{app}-[h_{app}\partial_x+g_{app}\partial_y]u_{app}-\kappa\Delta_\epsilon h_{app},\\
R^{4}_{app}=&[u_{app}\partial_x+v_{app}\partial_y]g_{app}-[h_{app}\partial_x+g_{app}\partial_y]v_{app}-\kappa\Delta_\epsilon g_{app},
\end{aligned}
\right.
\end{equation}
in which $\Delta_\epsilon=\epsilon\partial_x^2+\partial_y^2.$
\subsection{Zeroth-order boundary layer}\label{0th}
In this subsection, we will construct the zeroth-order boundary layer profiles $(u^0_p,v^0_p,h^0_p,g^0_p,0)$. Keep in mind that the trace of a function $f$ on $\{y=0\}$ is denoted as $\overline{f}:=f(x,0).$
The leading order terms in the first, third and fourth equations in (\ref{2.5}) can be written as follows:
\begin{equation}\nonumber
\left \{
\begin{aligned}
R^{1,0}=&\left[(u^0_e+u^0_p)\partial_x +(v^0_p+v^1_e)\partial_y\right](u^0_e+u^0_p)\\
&-\left[(h^0_e+h^0_p)\partial_x +(g^0_p+g^1_e)\partial_y\right](h^0_e+h^0_p)-\nu \partial_y^2(u^0_e+u^0_p),\\
R^{3,0}=&\left[(u^0_e+u^0_p)\partial_x +(v^0_p+v^1_e)\partial_y\right](h^0_e+h^0_p)\\
&-\left[(h^0_e+h^0_p)\partial_x +(g^0_p+g^1_e)\partial_y\right](u^0_e+u^0_p)-\kappa \partial_y^2(h^0_e+h^0_p),\\
R^{4,0}=&\left[(u^0_e+u^0_p)\partial_x +(v^0_p+v^1_e)\partial_y\right](g^0_p+g^1_e)\\
&-\left[(h^0_e+h^0_p)\partial_x +(g^0_p+g^1_e)\partial_y\right](v^0_p+v^1_e)-\kappa \partial_y^2(g^0_p+g^1_e).
\end{aligned}
\right.
\end{equation}
Recall that $\partial_x u^0_e=\partial_x h^0_e=0$ and the ideal MHD flows $(u^0_e, h^0_e)$ are evaluated at $(0,Y)=(0,\sqrt{\epsilon}y)$, direct calculations lead to
\begin{equation}\nonumber
\begin{array}{lll}
\left(v^0_p+v^1_e\right)\partial_y u^0_e=\sqrt{\epsilon}\left(v^0_p+v^1_e\right)\partial_Y u^0_{e},\ \  \left(g^0_p+g^1_e\right)\partial_y h^0_e=\sqrt{\epsilon}\left(g^0_p+g^1_e\right)\partial_Y h^0_{e},\\
\left(v^0_p+v^1_e\right)\partial_y h^0_e=\sqrt{\epsilon}\left(v^0_p+v^1_e\right)\partial_Y h^0_{e},\ \  \left(g^0_p+g^1_e\right)\partial_y u^0_e=\sqrt{\epsilon}\left(g^0_p+g^1_e\right)\partial_Y u^0_{e},\\
\left(v^0_p+v^1_e\right)\partial_y g^1_e=\sqrt{\epsilon}\left(v^0_p+v^1_e\right)\partial_Y g^1_{e},\ \  \left(g^0_p+g^1_e\right)\partial_y v^1_e=\sqrt{\epsilon}\left(g^0_p+g^1_e\right)\partial_Y v^1_{e}.
\end{array}
\end{equation}
Consequently,
\begin{equation}\nonumber
\begin{array}{lll}
&u^0_e\partial_x u^0_p+v^1_e\partial_y u^0_p-h^0_e\partial_x h^0_p-g^1_e\partial_y h^0_p\\
=&u_e\partial_x u^0_p+\overline{v^1_e}\partial_y u^0_p+\sqrt{\epsilon}y\left( u^0_{eY}(\sqrt{\epsilon}y)\partial_x u^0_p+v^1_{eY}(x,\sqrt{\epsilon}y)\partial_y u^0_p\right)\\
&-h_e\partial_x h^0_p-\overline{g^1_e}\partial_y h^0_p-\sqrt{\epsilon}y\left( h^0_{eY}(\sqrt{\epsilon}y)\partial_x h^0_p+g^1_{eY}(x,\sqrt{\epsilon}y)\partial_y h^0_p\right)+E_1,\\
\end{array}
\end{equation}
and
\begin{equation}\nonumber
\begin{array}{lll}
&u^0_e\partial_x h^0_p+v^1_e\partial_y h^0_p-h^0_e\partial_x u^0_p-g^1_e\partial_y u^0_p\\
=&u_e\partial_x h^0_p+\overline{v^1_e}\partial_y h^0_p+\sqrt{\epsilon}y\left( u^0_{eY}(\sqrt{\epsilon}y)\partial_x h^0_p+v^1_{eY}(x,\sqrt{\epsilon}y)\partial_y h^0_p\right)\\
&-h_e\partial_x u^0_p-\overline{g^1_e}\partial_y u^0_p-\sqrt{\epsilon}y\left( h^0_{eY}(\sqrt{\epsilon}y)\partial_x u^0_p+g^1_{eY}(x,\sqrt{\epsilon}y)\partial_y u^0_p\right)+E_3,\\
\end{array}
\end{equation}
and
\begin{equation}\nonumber
\begin{array}{lll}
&u^0_e\partial_x g^0_p+v^1_e\partial_y g^0_p+u^0_p\partial_x g^1_e-h^0_e\partial_x v^0_p-g^1_e\partial_y v^0_p-h^0_p\partial_x v^1_e\\
=&u_e\partial_x g^0_p+\overline{v^1_e}\partial_y g^0_p+\sqrt{\epsilon}y\left( u^0_{eY}(\sqrt{\epsilon}y)\partial_x g^0_p+v^1_{eY}(x,\sqrt{\epsilon}y)\partial_y g^0_p\right)\\
&-h_e\partial_x v^0_p-\overline{g^1_e}\partial_y v^0_p-\sqrt{\epsilon}y\left( h^0_{eY}(\sqrt{\epsilon}y)\partial_x v^0_p+g^1_{eY}(x,\sqrt{\epsilon}y)\partial_y v^0_p\right)\\
&+u^0_p\overline{\partial_x g^1_e}-h^0_p\overline{\partial_x v^1_e}+E_4,
\end{array}
\end{equation}
where $(u_e,h_e):=(u^0_e,h^0_e)(0)$, and
\begin{equation}\label{2.6a}
\left \{
\begin{aligned}
E_1=&\epsilon\partial_x u^0_p \int_0^y \int_y^\theta \partial_Y^2u^0_e(\sqrt{\epsilon}\tau)\mathrm{d}\tau\mathrm{d}\theta+\epsilon\partial_y u^0_p\int_0^y\int_y^\theta \partial_Y^2v^1_e(x,\sqrt{\epsilon}\tau)\mathrm{d}\tau\mathrm{d}\theta\\
&-\epsilon\partial_x h^0_p \int_0^y \int_y^\theta \partial_Y^2h^0_e(\sqrt{\epsilon}\tau)\mathrm{d}\tau\mathrm{d}\theta-\epsilon\partial_y h^0_p\int_0^y\int_y^\theta \partial_Y^2g^1_e(x,\sqrt{\epsilon}\tau)\mathrm{d}\tau\mathrm{d}\theta,\\
E_3=&\epsilon\partial_x h^0_p \int_0^y \int_y^\theta \partial_Y^2u^0_e(\sqrt{\epsilon}\tau)\mathrm{d}\tau\mathrm{d}\theta+\epsilon\partial_y h^0_p\int_0^y\int_y^\theta \partial_Y^2v^1_e(x,\sqrt{\epsilon}\tau)\mathrm{d}\tau\mathrm{d}\theta\\
&-\epsilon\partial_x u^0_p \int_0^y \int_y^\theta \partial_Y^2h^0_e(\sqrt{\epsilon}\tau)\mathrm{d}\tau\mathrm{d}\theta-\epsilon\partial_y u^0_p\int_0^y\int_y^\theta \partial_Y^2g^1_e(x,\sqrt{\epsilon}\tau)\mathrm{d}\tau\mathrm{d}\theta,\\
E_4=&\epsilon\partial_x g^0_p \int_0^y \int_y^\theta \partial_Y^2u^0_e(\sqrt{\epsilon}\tau)\mathrm{d}\tau\mathrm{d}\theta+\epsilon\partial_y g^0_p\int_0^y\int_y^\theta \partial_Y^2v^1_e(x,\sqrt{\epsilon}\tau)\mathrm{d}\tau\mathrm{d}\theta\\
&-\epsilon\partial_x v^0_p \int_0^y \int_y^\theta \partial_Y^2h^0_e(\sqrt{\epsilon}\tau)\mathrm{d}\tau\mathrm{d}\theta-\epsilon\partial_y v^0_p\int_0^y\int_y^\theta \partial_Y^2g^1_e(x,\sqrt{\epsilon}\tau)\mathrm{d}\tau\mathrm{d}\theta\\
&+\sqrt{\epsilon}u^0_p\int_0^y \partial_{Yx}g^1_e(x,\sqrt{\epsilon}\tau)\mathrm{d}\tau-\sqrt{\epsilon}h^0_p\int_0^y \partial_{Yx}v^1_e(x,\sqrt{\epsilon}\tau)\mathrm{d}\tau.
\end{aligned}
\right.
\end{equation}
Collecting the leading order terms, we obtain the nonlinear MHD boundary layer equations as follows.
\begin{equation}\label{2.6i}
\left \{
\begin{aligned}
&(u_e+u^0_p)\partial_x u^0_p+(v^0_p+v^1_e(x,0))\partial_y u^0_p -(h_e+h^0_p)\partial_x h^0_p\\
&\quad -(g^0_p+g^1_e(x,0))\partial_y h^0_p=\nu\partial_y^2u^0_p,\\
&(u_e+u^0_p)\partial_x h^0_p+(v^0_p+v^1_e(x,0))\partial_y h^0_p-(h_e+h^0_p)\partial_x u^0_p\\
&\quad -(g^0_p+g^1_e(x,0))\partial_y u^0_p=\kappa\partial_y^2h^0_p,\\
&(u_e+u^0_p)\partial_x (g^0_p+\overline{g^1_e})+(v^0_p+v^1_e(x,0))\partial_y (g^0_p+\overline{g^1_e})\\
&\quad -(h_e+h^0_p)\partial_x (v^0_p+\overline{v^1_e})-(g^0_p+g^1_e(x,0))\partial_y (v^0_p+\overline{v^1_e})=\kappa\partial_y^2g^0_p,\\
&(v^0_p,g^0_p)(x,y)=\int_y^\infty \partial_x (u^0_p,h^0_p)(x,z)\mathrm{d}z,\\
&(v^1_e,g^1_e)(x,0)= -\int_0^\infty\partial_x (u^0_p,h^0_p)(x,z)\mathrm{d}z,\\
&(u^0_p,\partial_y h^0_p)(x,0)=(u_b-u_e,0), \\
&(v^0_p,g^0_p)(x,0)=-(v^1_e,g^1_e)(x,0),\  (u^0_p,h^0_p)(0,y)=(\overline{u}_0(y),\overline{h}_0(y)).
\end{aligned}
\right.
\end{equation}
Since the third equation in (\ref{2.6i}) is a direct consequence of the second equation in (\ref{2.6i}), the divergence-free conditions and the boundary conditions, in this way, it suffices to consider the following initial-boundary value problem:
\begin{equation}\label{2.6}
\left \{
\begin{aligned}
&(u_e+u^0_p)\partial_x u^0_p+(v^0_p+v^1_e(x,0))\partial_y u^0_p -(h_e+h^0_p)\partial_x h^0_p\\
&\quad -(g^0_p+g^1_e(x,0))\partial_y h^0_p=\nu\partial_y^2u^0_p,\\
&(u_e+u^0_p)\partial_x h^0_p+(v^0_p+v^1_e(x,0))\partial_y h^0_p-(h_e+h^0_p)\partial_x u^0_p\\
&\quad -(g^0_p+g^1_e(x,0))\partial_y u^0_p=\kappa\partial_y^2h^0_p,\\
&(v^0_p,g^0_p)(x,y)=\int_y^\infty \partial_x (u^0_p,h^0_p)(x,z)\mathrm{d}z,\\
&(v^1_e,g^1_e)(x,0)= -\int_0^\infty\partial_x (u^0_p,h^0_p)(x,z)\mathrm{d}z,\\
&(u^0_p,\partial_y h^0_p)(x,0)=(u_b-u_e,0), \\
&(v^0_p,g^0_p)(x,0)=-(v^1_e,g^1_e)(x,0),\  (u^0_p,h^0_p)(0,y)=(\overline{u}_0(y),\overline{h}_0(y)).
\end{aligned}
\right.
\end{equation}
After extracting the equations (\ref{2.6}) satisfied by the leading order boundary layer profiles, the accuracy of the error terms of $R^{1,0},R^{3,0}$ and $R^{4,0}$ can be improved as the following forms:
\begin{equation}\label{2.6b}
\left \{
\begin{aligned}
R^{1,0}=&\sqrt{\epsilon}\left(v^0_p+v^1_e\right)\partial_Yu^0_e+\sqrt{\epsilon}y\left(\partial_Yu^0_e(\sqrt{\epsilon}y)\partial_x u^0_p+\partial_Yv^1_e\partial_yu^0_p\right)-\nu\epsilon \partial_Y^2 u^0_e\\
&-\sqrt{\epsilon}\left(g^0_p+g^1_e\right)\partial_Yh^0_e-\sqrt{\epsilon}y\left(\partial_Yh^0_e(\sqrt{\epsilon}y)\partial_x h^0_p+\partial_Yg^1_e\partial_yh^0_p\right) +E_1,\\
R^{3,0}=&\sqrt{\epsilon}\left(v^0_p+v^1_e\right)\partial_Yh^0_e+\sqrt{\epsilon}y\left(\partial_Yu^0_e(\sqrt{\epsilon}y)\partial_x h^0_p+\partial_Yv^1_e\partial_yh^0_p\right)-\kappa\epsilon \partial_Y^2 h^0_e\\
&-\sqrt{\epsilon}\left(g^0_p+g^1_e\right)\partial_Yu^0_e-\sqrt{\epsilon}y\left(\partial_Yh^0_e(\sqrt{\epsilon}y)\partial_x u^0_p+\partial_Yg^1_e\partial_yu^0_p\right) +E_3,\\
R^{4,0}=&E_4.
\end{aligned}
\right.
\end{equation}
Further improvement of accuracy of $R^{1,0},R^{3,0}$ in term of $\epsilon$ will be left to the next subsection by solving the equations of next order inner flows and boundary layer profiles in the first and third equalities in (\ref{2.5}) with the source terms coming from $R^{1,0}$ and $R^{3,0}$ in (\ref{2.6b}).

To solve the problem (\ref{2.6}), it is convenient to homogenize the boundary conditions at $\{y=0\}$ and at infinity. To this end,  let us introduce the cut-off function
\begin{equation}\label{cutoff}
\phi(y)=
\left \{
\begin{array}{lll}
1,& y \geq 2R_0,\\
0,& 0\leq y \leq R_0
\end{array}
\right.
\end{equation}
for some constant $R_0>0$. It is reasonable to assume that the derivatives of $\phi(y)$ are bounded. Define the new unknowns as follows.
\begin{equation}\label{newunknown}
\left \{
\begin{array}{lll}
u=u^0_p+u_e-u_e\phi(y)-u_b(1-\phi(y)),\\
v=v^1_e(x,0)+v^0_p,\\
h=h^0_p+h_e-h_e\phi(y),\\
g=g^1_e(x,0)+g^0_p.
\end{array}
\right.
\end{equation}
Consequently, we have
\begin{equation}\label{bc}
\left \{
\begin{array}{lll}
(u,v,\partial_y h,g)|_{y=0}=(0,0,0,0),\\
(u,h) \to (0,0) \ \ \mathrm{as} \ y \to \infty,\\
\partial_x u +\partial_y v =\partial_x h+\partial_y g=0,
\end{array}
\right.
\end{equation}
and the new unknowns satisfy the following initial-boundary value problem:
\begin{equation}\label{newequation}
\left \{
\begin{array}{lll}
\big[\big(u+u_e\phi(y)+u_b(1-\phi(y))\big)\partial_x +v\partial_y\big]u-\big[(h+h_e\phi(y))\partial_x+g\partial_y\big]h\\
\quad \quad \quad -\nu \partial_y^2u
-gh_e\phi'(y)+v(u_e-u_b)\phi'(y)=r_1,\\
\big[\big(u+u_e\phi(y)+u_b(1-\phi(y))\big)\partial_x +v\partial_y\big]h-\big[(h+h_e\phi(y))\partial_x+g\partial_y\big]u\\
\quad \quad \quad -\kappa \partial_y^2 h
-g(u_e-u_b)\phi'(y)+vh_e\phi'(y)=r_2,\\
\partial_x u+\partial_y v=\partial_x h+\partial_y g=0,\\
(u,v,\partial_y h,g)|_{y=0}=(0,0,0,0),\\
(u,h) \to (0,0) \ \ \mathrm{as} \ \ y\to \infty,\\
(u,h)|_{x=0}=(\overline{u}_0(y)+(u_e-u_b)(1-\phi(y)),\overline{h}_0(y)+h_e(1-\phi(y))
\triangleq (u_0,h_0)(y),
\end{array}
\right.
\end{equation}
with
$$r_1=\nu (u_b-u_e)\phi''(y), \qquad\qquad  r_2=\kappa h_e\phi''(y).$$

Let us define the weighted Sobolev spaces used in this subsection. For $l \in \mathbb{R}$, denote
$$L^2_l:=\left\{f(x,y):[0,L)\times [0,\infty)\to \mathbb{R}, \ \Vert f\Vert_{L^2_l}^2=\int_0^\infty \langle y \rangle^{2l}|f(y)|^2\mathrm{d}y<\infty\right\},$$
where $\langle y\rangle =\sqrt{1+y^2}$.
For $\alpha=(\beta,k)\in \mathbb{N}^2$, $D^\alpha=\partial_x^\beta\partial_y^k$, we define
$$H^m_l:=\left\{f(x,y):[0,L]\times [0,\infty)\to \mathbb{R}, \ \Vert f\Vert_{H^m_l}^2<\infty\right\}$$
with the norm
$$\Vert f\Vert_{H^m_l}^2=\sum_{|\alpha| \leq m}\Vert  \langle y\rangle^{l+k}D^\alpha f \Vert_{L^2_y(0,\infty)}^2.$$
First, one deduces from the definition of $\phi(y)$ and (\ref{newunknown}) that
\begin{equation}\label{newunknown1}
\Vert (u,h)\Vert_{H^m_l} -C(u_e,u_b,h_e)\leq \Vert (u^0_p,h^0_p)\Vert_{H^m_l} \leq \Vert (u,h)\Vert_{H^m_l} +C(u_e,u_b,h_e).
\end{equation}
Moreover, for any $x\in [0,L], \lambda \geq 0$ and $|\alpha|\leq m$, it holds that
\begin{equation}\label{r}
\Vert \langle y\rangle^\lambda D^\alpha (r_1,r_2)\Vert_{L^2} \leq C(u_e,h_e,u_b).
\end{equation}
Similarly,
\begin{equation}\label{initialdata}
\begin{aligned}
\Vert (u_0,h_0)\Vert_{H^m_l} -C(u_e,u_b,h_e)\leq &\Vert (u^0_p,h^0_p)(0,y)\Vert_{H^m_l} \\
 \leq & \Vert (u_0,h_0)\Vert_{H^m_l} +C(u_e,u_b,h_e).
\end{aligned}
\end{equation}

The well-posedness of solutions to the system of equations for the leading order boundary layers (\ref{newequation}) is stated as follows.
\begin{proposition}\label{wellposedness}
Let $m\geq 5$ and $l\geq 0$. Suppose that $u_e^0,h_e^0$ are the smooth ideal MHD flows so that $\partial_Y u_e^0,\partial_Yh_e^0$ and their derivatives decay exponentially fast to zero at infinity. Moreover, suppose that
\begin{equation}\label{c1}
 u_e+u^0_p (0,y)>h_e+h^0_p (0,y)\geq \vartheta_0>0
\end{equation}
uniform in $y$ for some $\vartheta_0>0$, also if
\begin{equation}\label{c2}
\begin{aligned}
&|u_e^0(\sqrt{\epsilon}y)+u^0_p(0,y)|>>|h_e^0(\sqrt{\epsilon}y)+h^0_p(0,y)|,\\
&| \langle y\rangle^{l+1}\partial_y (u_e+u^0_p,h_e+h^0_p)(0,y)| \leq \frac{1}{2}\sigma_0, \\
&| \langle y\rangle^{l+1}\partial_y^2 (u_e+u^0_p,h_e+h^0_p)(0,y)| \leq \frac{1}{2}\vartheta_0^{-1},
\end{aligned}
\end{equation}
uniform in $y$, here $\sigma_0>0$ is also another suitably small constant. Then there exists $L_1>0$, such that the problem (\ref{newequation}) admits local-in-$x$ smooth solutions $u,v,h,g$ in $[0,L_1]\times [0,\infty)$ with
\begin{equation}\label{c3}
\begin{aligned}
\sup\limits_{0\leq x\leq L_1}&\Vert (u,h)\Vert_{H^m_l(0,\infty)}\\
&+\Vert  \partial_y(u,h)\Vert_{L^2(0,L;H^m_l(0,\infty))}\leq    C(l,m,u_0,h_0),
\end{aligned}
\end{equation}
and for any $x \in[0,L_1]$. Moreover, for $(x,y)\in [0,L_1]\times [0,\infty)$,
\begin{equation}\label{c3a}
\begin{aligned}
&h_e+h^0_p (x,y)\geq \frac{\vartheta_0}{2}>0,\\
&|u_e^0(\sqrt{\epsilon}y)+u^0_p(x,y)|>>|h_e^0(\sqrt{\epsilon}y)+h^0_p(x,y)|,\\
&| \langle y\rangle^{l+1}\partial_y  (u_e+u^0_p,h_e+h^0_p)(x,y)| \leq \sigma_0,\\
&| \langle y\rangle^{l+1}\partial_y^2  (u_e+u^0_p,h_e+h^0_p)(x,y)| \leq \vartheta_0^{-1}.
\end{aligned}
\end{equation}
\end{proposition}
Based on Proposition \ref{wellposedness}, we have the following corollary.
\begin{corollary}\label{coro1}
Under the assumptions in Proposition \ref{wellposedness}, then $u^0_p,v^0_p,h^0_p,g^0_p$ enjoy that
\begin{equation}\label{c4}
\begin{aligned}
&\sup\limits_{0\leq x\leq L_1}\Vert \langle y\rangle^l D^\alpha (u_p^0,v_p^0,h_p^0,g_p^0)\Vert_{L^2(0,\infty)}\leq    C,\\
&| \langle y\rangle^{l+1}\partial_y  (u^0_p,h^0_p)(x,y)| \leq \sigma_0 , \  (x,y)\in [0,L_1]\times [0,\infty),\\
&|u_e^0(\sqrt{\epsilon}y)+u^0_p(x,y)|>>|h_e^0(\sqrt{\epsilon}y)+h^0_p(x,y)|,\  (x,y)\in [0,L_1]\times [0,\infty),\\
\end{aligned}
\end{equation}
for any $\alpha,m,l$ with $|\alpha|\leq m$.
\end{corollary}

The proof of the Proposition \ref{wellposedness} is relatively tricky and long, and it is left to be handled in the Appendix \ref{ap1}.

\subsection{First-order correctors}\label{sec1st}

Next, by collecting all terms with a factor of $\sqrt{\epsilon}$ from $R^{1}_{app}$ and $R^{3}_{app}$, together with the $\sqrt{\epsilon}$-order terms from $R^{1,0}$ and $R^{3,0}$, we have
\begin{equation}\label{4.001}
\left \{
\begin{aligned}
R^{u,1}=&\big[(u^1_e+u^1_p)\partial_x +v^1_p\partial_y\big](u^0_e+u^0_p)+\big[(u^0_e+u^0_p)\partial_x+(v^0_p+v^1_e)\partial_y\big](u^1_e+u^1_p)\\
&+\partial_x( p^1_e+p^1_p)-\nu\partial_y^2 (u^1_e+u^1_p)+(y\partial_x u^0_p+v^0_p+v^1_e)\partial_Yu^0_e+y\partial_Y v^1_e\partial_y u^0_p\\
&-\big[(h^1_e+h^1_p)\partial_x +g^1_p\partial_y\big](h^0_e+h^0_p)-\big[(h^0_e+h^0_p)\partial_x+(g^0_p+g^1_e)\partial_y\big](h^1_e+h^1_p)\\
&-(y\partial_x h^0_p+g^0_p+g^1_e)\partial_Yh^0_e-y\partial_Y g^1_e\partial_y h^0_p,\\
R^{h,1}=&\big[(u^1_e+u^1_p)\partial_x +v^1_p\partial_y\big](h^0_e+h^0_p)+\big[(u^0_e+u^0_p)\partial_x+(v^0_p+v^1_e)\partial_y\big](h^1_e+h^1_p)\\
&-\kappa\partial_y^2 (h^1_e+h^1_p)+(v^0_p+v^1_e)\partial_Y h^0_e+y\partial_Yu^0_e\partial_x h^0_p+y\partial_Y v^1_e\partial_y h^0_p\\
&-\big[(h^1_e+h^1_p)\partial_x +g^1_p\partial_y\big](u^0_e+u^0_p)-\big[(h^0_e+h^0_p)\partial_x+(g^0_p+g^1_e)\partial_y\big](u^1_e+u^1_p)\\
&-(g^0_p+g^1_e)\partial_Yu^0_e-y\partial_Yh^0_e\partial_x u^0_p-y\partial_Y g^1_e\partial_y u^0_p.
\end{aligned}
\right.
\end{equation}
Note that when the operator $\partial_y$ acts on the inner flow terms which are evaluated at $Y=\sqrt{\epsilon}y$, there will produce a factor of $\sqrt{\epsilon}$. In this way, these terms should be moved into the next order of $\sqrt{\epsilon}$. For example,
$$[v^0_p+v^1_e]\partial_y u^1_e=\sqrt{\epsilon}[v^0_p+v^1_e]\partial_Y u^1_e, \qquad \partial_y^2 u^1_e=\epsilon \partial_Y^2 u^1_e.$$
Consequently, the terms of related $\epsilon^{\frac12}$-order MHD inner flow satisfy
\begin{equation}\label{4.002}
\left \{
\begin{aligned}
&u^0_e\partial_x u^1_e+v^1_e\partial_Y u^0_e-h^0_e\partial_x h^1_e-g^1_e\partial_Y h^0_e+\partial_x p^1_e=0,\\
&u^0_e\partial_x h^1_e+v^1_e\partial_Y h^0_e-h^0_e\partial_x u^1_e-g^1_e\partial_Y u^0_e=0,
\end{aligned}
\right.
\end{equation}
and the $\epsilon^{\frac12}$-order boundary layer terms are governed by the following system of equations:
\begin{equation}\label{4.01}
\left \{
\begin{aligned}
(&u^1_e+u^1_p)\partial_x u^0_p+u^0_p\partial_x u^1_e+(u^0_e+u^0_p)\partial_x u^1_p+v^1_p(\partial_y u^0_p+\sqrt{\epsilon}\partial_Y u^0_e)\\
&+(v^0_p+v^1_e)\partial_y u^1_p+\partial_x p^1_p-\nu \partial_{y}^2 u^1_p+(y\partial_x u^0_p+v^0_p)\partial_Y u^0_e+y\partial_Y v^1_e\partial_y u^0_p\\
&-(h^1_e+h^1_p)\partial_x h^0_p-h^0_p\partial_x h^1_e-(h^0_e+h^0_p)\partial_x h^1_p-g^1_p(\partial_y h^0_p+\sqrt{\epsilon}\partial_Y h^0_e)\\
&-(g^0_p+g^1_e)\partial_y h^1_p-(y\partial_x h^0_p+g^0_p)\partial_Yh^0_e-y\partial_Y g^1_e \partial_y h^0_p=0,\\
(&u^1_e+u^1_p)\partial_x h^0_p+u^0_p\partial_x h^1_e+(u^0_e+u^0_p)\partial_x h^1_p+v^1_p(\partial_y h^0_p+\sqrt{\epsilon}\partial_Yh^0_e)\\
&+(v^0_p+v^1_e)\partial_y h^1_p-\kappa \partial_y^2 h^1_p+v^0_p\partial_Y h^0_e+y\partial_Y u^0_e\partial_x h^0_p+y\partial_Y v^1_e\partial_y h^0_p\\
&-(h^1_e+h^1_p)\partial_x u^0_p-h^0_p\partial_x u^1_e-(h^0_e+h^0_p)\partial_x u^1_p-g^1_p(\partial_y u^0_p+\sqrt{\epsilon}\partial_Y u^0_e)\\
&-(g^0_p+g^1_e)\partial_y u^1_p-g^0_p\partial_Y u^0_e-y\partial_Y h^0_e\partial_x u^0_p-y\partial_Y g^1_e\partial_y u^0_p=0.
\end{aligned}
\right.
\end{equation}
After constructing the above profiles, the errors of $R^{u,1}$ and $R^{h,1}$ in (\ref{4.001}) are further reduced to
\begin{equation}\label{4.02}
\sqrt{\epsilon}\left[(v^0_p+v^1_e)\partial_Y u^1_e-(g^0_p+g^1_e)\partial_Y h^1_e\right]-\nu \epsilon \partial_Y^2u^1_e
\end{equation}
and
\begin{equation}\label{4.03}
\sqrt{\epsilon}\left[(v^0_p+v^1_e)\partial_Y h^1_e-(g^0_p+g^1_e)\partial_Y u^1_e\right]-\kappa \epsilon \partial_Y^2h^1_e.
\end{equation}

Next, we consider the approximation of normal components, which correspond to the second and fourth equalities in (\ref{2.5}). Clearly, the leading order term in the second equality in (\ref{2.5}) yields  $\frac{1}{\sqrt{\epsilon}}\partial_y p^1_p=0$.  Therefore, it holds that
\begin{equation}\label{4.04}
p^1_p(x,y)=p^1_p(x).
\end{equation}
Furthermore, the next order terms in the second and fourth equalities in (\ref{2.5}) consists of
\begin{equation}\label{4.05}
\left \{
\begin{aligned}
R^{v,0}=&\big[(u^0_e+u^0_p)\partial_x +(v^0_p+v^1_e)\partial_y\big](v^0_p+v^1_e)+\partial_Y p^1_e+\partial_y p^2_p\\
&-\nu \partial_y^2(v^0_p+v^1_e)-\big[(h^0_e+h^0_p)\partial_x +(g^0_p+g^1_e)\partial_y\big](g^0_p+g^1_e),\\
R^{g,0}=&\big[(u^0_e+u^0_p)\partial_x +(v^0_p+v^1_e)\partial_y\big](g^0_p+g^1_e)-\kappa \partial_y^2(g^0_p+g^1_e)\\
&-\big[(h^0_e+h^0_p)\partial_x +(g^0_p+g^1_e)\partial_y\big](v^0_p+v^1_e).\\
\end{aligned}
\right.
\end{equation}
As above, we enforce $R^{v,0}=R^{g,0}=0$. It should be noted that some terms in $R^{g,0}$ have been determined in the leading order boundary layer equations in Subsection \ref{0th}, see (\ref{2.6i}) for details. The rest terms in $R^{g,0}$ read as
\begin{equation}\label{4.05}
\begin{aligned}
R^{g,0}=&u^0_e\partial_x g^1_e-h^0_e\partial_x v^1_e+\sqrt{\epsilon}\big[(v^0_p+v^1_e)\partial_Y g^1_e-(g^0_p+g^1_e)\partial_Y v^1_e\big]-\kappa\epsilon \partial_Y^2 g^1_e.\\
\end{aligned}
\end{equation}
And $(v^1_p,g^1_p)$ is determined by the divergence-free conditions and $(u^1_p,h^1_p)$. Consequently, the first order inner flow profiles $(u^1_e,v^1_e,h^1_e,g^1_e,p^1_e)$ enjoy
\begin{equation}\label{4.06}
\left \{
\begin{aligned}
&u^0_e\partial_x v^1_e-h^0_e\partial_x g^1_e+\partial_Y p^1_e=0,\\
&u^0_e\partial_x g^1_e-h^0_e\partial_x v^1_e=0.
\end{aligned}
\right.
\end{equation}
Whereas, the next order boundary layer of pressure $p^2_p$ is taken as
\begin{equation}\label{4.07}
\begin{aligned}
p^2_p(x,y)=&\int_y^\infty \bigg[(u^0_e+u^0_p)\partial_x v^0_p+u^0_p\partial_x v^1_e+(v^0_p+v^1_e)\partial_y v^0_p\\
&-(h^0_e+h^0_p)\partial_x g^0_p-h^0_p\partial_x g^1_e-(g^0_p+g^1_e)\partial_y g^0_p-\nu \partial_y^2 v^0_p\bigg](x,\theta)\mathrm{d}\theta.
\end{aligned}
\end{equation}
As a consequence, the error terms $R^{v,0},R^{g,0}$ can be reduced into the following forms:
\begin{equation}\label{4.08}
\left \{
\begin{aligned}
R^{v,0}=&\sqrt{\epsilon}\big[(v^0_p+v^1_e)\partial_Y v^1_e-(g^0_p+g^1_e)\partial_Y g^1_e\big]-\nu\epsilon \partial_Y^2 v^1_e,\\
R^{g,0}=&\sqrt{\epsilon}\big[(v^0_p+v^1_e)\partial_Y g^1_e-(g^0_p+g^1_e)\partial_Y v^1_e\big]-\kappa\epsilon \partial_Y^2 g^1_e.
\end{aligned}
\right.
\end{equation}

\subsection{The ideal MHD correctors}\label{subsec4.1}
Based on the above deduction, we solve the following system of equations with some suitable boundary conditions to obtain the $\sqrt{\epsilon}$ order ideal inner MHD correctors $(u^1_e,v^1_e,h^1_e,g^1_e,p^1_e)$:
\begin{equation}\label{4.09}
\left \{
\begin{aligned}
&u^0_e\partial_x u^1_e+v^1_e\partial_Yu^0_e-h^0_e\partial_x h^1_e-g^1_e\partial_Y h^0_e+\partial_x p^1_e=0,\\
&u^0_e\partial_x v^1_e-h^0_e\partial_x g^1_e+\partial_Y p^1_e=0,\\
&u^0_e\partial_x h^1_e+v^1_e\partial_Y h^0_e-h^0_e\partial_x u^1_e-g^1_e\partial_Y u^0_e=0,\\
&u^0_e\partial_x g^1_e-h^0_e\partial_x v^1_e=0,\\
&\partial_x u^1_e+\partial_Y v^1_e=\partial_x h^1_e+\partial_Y g^1_e=0,
\end{aligned}
\right.
\end{equation}
with the boundary conditions
\begin{equation}\label{4.09a}
\left \{
\begin{aligned}
&(v^1_e,g^1_e)(x,0)=-(v^0_p,g^0_p)(x,0),\\
&(v^1_e,g^1_e)(0,Y)=(V_{b0},G_{b0})(Y),\\
&(v^1_e,g^1_e)(L,Y)=(V_{bL},G_{bL})(Y),\\
&(v^1_e,g^1_e) \to (0,0) \ \ \mathrm{as} \ \ Y\to \infty.
\end{aligned}
\right.
\end{equation}
And the compatibility conditions hold at corners:
\begin{equation}\label{4.09b}
\left \{
\begin{aligned}
&(V_{b0},G_{b0})(0)=-(v^0_p,g^0_p)(0,0),\\
&(V_{bL},G_{bL})(0)=-(v^0_p,g^0_p)(L,0).
\end{aligned}
\right.
\end{equation}

To solve the problem (\ref{4.09})-(\ref{4.09b}), by the divergence-free conditions for velocity and magnetic fields, the third equation in (\ref{4.09}) can be rewritten as
$$\partial_Y(u^0_eg^1_e-h^0_ev^1_e)=0,$$
which gives that
\begin{equation}\label{4.09c}
g_e^1=\frac{h^0_e}{u_e^0}v^1_e+\frac{1}{u^0_e}(h_e\overline{v^0_p}-u_e\overline{g^0_p}):=\frac{h^0_e}{u_e^0}v^1_e+\frac{b(x)}{u^0_e}
\end{equation}
with
\begin{align*}
b(x)=(h_e\overline{v^0_p}-u_e\overline{g^0_p}).
\end{align*}
Moreover, we deduce from the first, second and fifth equations in (\ref{4.09}) that
\begin{equation}\label{4.010}
\begin{aligned}
-u_e^0\Delta v^1_e&+\partial_Y^2 u^0_e\cdot v^1_e+\left(h^0_e\Delta g^1_e-\partial_Y^2h^0_e \cdot g^1_e\right)=0.
\end{aligned}
\end{equation}

To avoid the singularity resulted from the presence of corners, we consider the following modified elliptic problem instead of (\ref{4.010}):
\begin{equation}\label{4.011}
\begin{aligned}
-u_e^0\Delta v^1_e&+\partial_Y^2 u^0_e\cdot v^1_e+\left(h^0_e\Delta g^1_e-\partial_Y^2h^0_e\cdot g^1_e\right)=E_b,
\end{aligned}
\end{equation}
with the boundary conditions (\ref{4.09a}). Here $E_b$ will be defined later.

To define $E_b$, we introduce
\begin{equation}\label{4.013}
\left \{
\begin{aligned}
B_v(x,Y):=&\left(1-\frac{x}{L}\right)\frac{V_{b0}(Y)}{v^0_p(0,0)}v^0_p(x,0)+\frac{x}{L}\frac{V_{bL}(Y)}{v^0_p(L,0)}v^0_p(x,0),\\
B_g(x,Y):=&\left(1-\frac{x}{L}\right)\frac{G_{b0}(Y)}{g^0_p(0,0)}g^0_p(x,0)+\frac{x}{L}\frac{G_{bL}(Y)}{g^0_p(L,0)}g^0_p(x,0)
\end{aligned}
\right.
\end{equation}
for the case that all of $v^0_p(0,0), v^0_p(L,0),g^0_p(0,0),g^0_p(L,0)$ are nonzero. If this is not the case, for example, if $v^0_p(0,0)=0$, we would replace $\frac{V_{b0}(Y)}{v^0_p(0,0)}v^0_p(x,0)$ by $V_{b0}(Y)-v^0_p(x,0)$. Then the definition of $B_v$ is changed into the following form:
\begin{equation}\nonumber
B_v(x,Y)=\left(1-\frac{x}{L}\right)\left[V_{b0}(Y)-v^0_p(x,0)\right]+\frac{x}{L}\frac{V_{bL}(Y)}{v^0_p(L,0)}v^0_p(x,0).
\end{equation}
Similar modifications can be done for the case that some of $v^0_p(L,0),g^0_p(0,0)$ and $g^0_p(L,0)$ are zero.

According to the compatibility assumptions (\ref{4.09b}) at corners, it is easy to check that the $B_v,B_g$ satisfy the boundary conditions (\ref{4.09a}).
Additionally, it follows that $B_v,B_g \in W^{k,p}$ for arbitrary $k,p$, provided that $|\partial_Y^k(V_{bL}-V_{b0},G_{bL}-G_{b0})(Y)|\leq CL$.

Then, we introduce the function $F_e$ with the definition being given as follows:
\begin{equation}\label{4.014}
\begin{aligned}
-u_e^0\Delta B_v+\partial_Y^2 u^0_e\cdot B_v+\left(h^0_e\Delta B_g-\partial_Y^2h^0_e\cdot B_g\right)=F_e,
\end{aligned}
\end{equation}
and $F_e$ is arbitrarily smooth and it holds that
\begin{equation}\label{4.015a}
\Vert \langle Y\rangle^n F_e\Vert_{W^{k,q}} \leq C
\end{equation}
for any $n,k \geq 0$ with $q\in [1,\infty]$, where the constant $C>0$ is independent of small $L$.
Set
$$v_e^1=B_v+w_1,\ \ g_e^1=B_g+w_2,$$
then it follows that
\begin{equation}\label{4.016}
\left \{
\begin{aligned}
&-u_e^0\Delta w_1+\partial_Y^2 u^0_e\cdot w_1+\left(h^0_e\Delta w_2-\partial_Y^2 h^0_e\cdot w_2\right)=E_b-F_e,\\
&w_2=\frac{h^0_e}{u^0_e}(w_1+B_v)+\frac{b}{u^0_e}-B_g,\\
&w_i|_{\partial\Omega}=0, \ \ i=1,2.
\end{aligned}
\right.
\end{equation}
Therefore, it suffices to establish the estimates for $w_1$.
It is convenient to introduce the boundary layer corrector $E_{b}$ as follows to obtain the high regularity
\begin{equation}\label{4.0170a}
E_b:=\chi\left(\frac{Y}{\epsilon}\right)F_e(x,0),
\end{equation}
where $\chi(\cdot)$ is a smooth cut-off function with support in $[0,1]$ with $\chi(0)=1$. Then we have
\begin{equation}\label{4.017a}
\Vert \langle Y\rangle^n\partial_Y^k E_b\Vert_{L^q} \lesssim \epsilon^{-k+\frac{1}{q}}, \ \ q\geq 1, \ \ n,k\geq0.
\end{equation}
\begin{lemma}\label{1stMHD}
Suppose that $V_{b0},V_{bL},G_{b0}$ and $G_{bL}$ are sufficiently smooth, decay rapidly as $Y\to \infty$, and satisfy $|\partial_Y^k(V_{bL}-V_{b0},G_{bL}-G_{b0})(Y)|\leq CL$ for any $k\geq 0$, uniformly in $Y$. Then, there exist unique smooth solutions $v^1_e,g^1_e$ to the elliptic problem (\ref{4.09a}), (\ref{4.09c}), (\ref{4.011}) and (\ref{4.0170a}), and it holds that
\begin{equation}\nonumber
\begin{aligned}
&\Vert (v^1_e,g^1_e)\Vert_{L^\infty}+\Vert \langle Y\rangle^n(v^1_e,g^1_e)\Vert_{H^2}\leq C_0,\\
&\Vert \langle Y\rangle^n(v^1_e,g^1_e)\Vert_{H^3}\leq C_0\epsilon^{-1/2},\\
&\Vert \langle Y\rangle^n(v^1_e,g^1_e)\Vert_{H^4}\leq C_0\epsilon^{-3/2},
\end{aligned}
\end{equation}
for any $n\geq 0$ and some constant $C_0>0$ independent of small $L>0$. Moreover,
\begin{equation}\nonumber
\begin{aligned}
&\Vert \langle Y \rangle^n(v^1_e,g^1_e)\Vert_{W^{2,q}}\leq C(L),\\
&\Vert \langle Y\rangle^n(v^1_e,g^1_e)\Vert_{W^{3,q}}\leq C(L)\epsilon^{-1+1/q},\\
&\Vert \langle Y\rangle^n(v^1_e,g^1_e)\Vert_{W^{4,q}}\leq C(L)\epsilon^{-2+1/q},
\end{aligned}
\end{equation}
for $n\geq 0$ and $q\in(1,\infty)$, and for some $C(L)>0$, which depends only on $L>0$.
\end{lemma}
\begin{proof}
The condition that $u^0_e(Y)>>h^0_e(Y)$ uniform in $Y$ yields that the system is non-degenerate. The existence of $w_1$ and $w_2$ can be obtained by the classical elliptic theory and we only focus on the derivation of estimates of solutions that will be used later.

Multiplying the first equation in (\ref{4.016}) by $\frac{w_1}{u^0_e}$ and integrating by parts, we get that
\begin{equation}\nonumber
\begin{aligned}
\iint &|\nabla w_1|^2+\iint \frac{\partial_Y^2u^0_e}{u^0_e}|w_1|^2\\
=&\iint\frac{h^0_e}{u^0_e}\nabla w_2\cdot\nabla w_1+\iint\partial_Y\left(\frac{h^0_e}{u^0_e}\right)\partial_Yw_2\cdot w_1\\
&+\iint\frac{\partial_Y^2h^0_e}{u^0_e}w_1\cdot w_2+\iint\frac{(E_b-F_e)}{u^0_e}w_1=:\sum_{i=1}^4 s_i.
\end{aligned}
\end{equation}
Each of $s_i\ (i=1,2,3,4)$ can be estimated as follows:
\begin{equation}\nonumber
\begin{aligned}
 s_1=\iint\frac{h^0_e}{u^0_e}\nabla w_2\cdot\nabla w_1\lesssim |\frac{h^0_e}{u^0_e}|\Vert\nabla w_1\Vert_{L^2}\Vert\nabla w_2\Vert_{L^2},
\end{aligned}
\end{equation}

\begin{equation}\nonumber
\begin{aligned}
s_2=&\iint\partial_Y\left(\frac{h^0_e}{u^0_e}\right)\partial_Yw_2\cdot w_1\\
 \lesssim& (|\frac{h^0_e}{u^0_e}||\frac{\partial_Y u^0_e}{u^0_e}|+|\frac{\partial_Y h^0_e}{u^0_e}|)\Vert w_1\Vert_{L^2}\Vert\partial_Y w_2\Vert_{L^2}\\
  \lesssim& L(|\frac{h^0_e}{u^0_e}||\frac{\partial_Y u^0_e}{u^0_e}|+|\frac{\partial_Y h^0_e}{u^0_e}|)\Vert \partial_x w_1\Vert_{L^2}\Vert\partial_Y w_2\Vert_{L^2},
\end{aligned}
\end{equation}
and
\begin{equation}\nonumber
\begin{aligned}
s_3=&\iint\frac{\partial_Y^2h^0_e}{u^0_e} w_1\cdot w_2\\
 \lesssim& \Vert w_1\Vert_{L^2}\Vert w_2\Vert_{L^2}\\
  \lesssim& L^2\Vert \partial_x w_1\Vert_{L^2}\Vert\partial_x w_2\Vert_{L^2}.
\end{aligned}
\end{equation}

In addition,
$$s_4=\iint \frac{E_b-F_e}{u^0_e}w_1 \lesssim L\Vert \partial_x w_1\Vert_{L^2}\Vert E_b-F_e\Vert_{L^2}.$$

Applying the crucial positivity estimate (see \cite{YGuo2} for details), it follows that
$$\iint \left(|\partial_Y w_1|^2+\frac{\partial_Y^2 u^0_e}{u^0_e}|w_1|^2\right)=\iint |u^0_e|^2\left|\partial_Y\left(\frac{w_1}{u^0_e}\right)\right|^2\geq \theta_0 \iint |\partial_Y w_1|^2$$
for some constant $\theta_0$ independent of $\epsilon,L$.

On the other hand, according to the second equation in (\ref{4.016}), one has
\begin{equation}\nonumber
\begin{aligned}
\nabla  w_2=\begin{pmatrix} \frac{h^0_e}{u^0_e}\partial_x w_1\\ \frac{h^0_e}{u^0_e}\partial_Y w_1+\partial_Y\left(\frac{h^0_e}{u^0_e}\right)w_1\end{pmatrix}+\nabla \left(\frac{h^0_e}{u^0_e}B_v+\frac{b}{u^0_e}-B_g\right).
\end{aligned}
\end{equation}
Therefore, we have
\begin{equation}\nonumber
\begin{aligned}
\Vert \nabla  w_2\Vert_{L^2}\lesssim 1+ \sup_Y\left|\frac{h^0_e}{u^0_e}\right|\Vert \nabla w_1\Vert_{L^2}+L\left(\sup_Y \left|\frac{h^0_e}{u^0_e}\right|+\Vert \partial_Y h^0_e\Vert_{L^\infty}\right)\Vert \partial_x w_1\Vert_{L^2}.
\end{aligned}
\end{equation}

Combining the above estimates and using the smallness of  $\sup\limits_Y\left|\frac{h^0_e}{u^0_e}\right|$ and $|\partial_Y h^0_e|$, one gets that
\begin{equation}\label{4.019}
\Vert w_1\Vert_{H^1} \leq C,
\end{equation}
in which we have used the Poincare's inequality and Young's inequality.
Moreover, we have
\begin{equation}\label{4.018}
\Vert w_2\Vert_{H^1} \leq C.
\end{equation}

Now we establish the higher order energy estimates. To this end, we rewrite the equation as follows.
\begin{equation}\label{4.020}
-\Delta w_1+\frac{h^0_e}{u^0_e}\Delta w_2=G_e,
\end{equation}
where
$$G_e:=\frac{1}{u^0_e}\left(E_b-F_e-\partial_Y^2u^0_e\cdot w_1+\partial_Y^2h^0_e\cdot w_2\right).$$
It is direct to calculate that $\Vert G_e\Vert_{L^2}\leq C$.

Since $E_b(x,0)-F_e(x,0)=0$, therefore $G_e=0$ on $\{Y=0\}$, and it follows from the first equation in (\ref{4.016}) that
\begin{equation}\label{4.021}
-\partial_Y^2 w_1+\frac{h_e}{u_e}\partial_Y^2 w_2=0, \  \  \mathrm{on} \ \ \{Y=0 \}.
\end{equation}

Applying the operator $\partial_Y$ on (\ref{4.020}) to get that
\begin{equation}\label{4.022}
-\Delta \partial_Y w_1+\frac{h^0_e}{u^0_e}\Delta \partial_Y w_2+\partial_Y \left(\frac{h^0_e}{u^0_e}\right)\Delta  w_2=\partial_Y G_e.
\end{equation}
Multiplying (\ref{4.022}) by $\partial_Y w_1$, one has
\begin{equation}\nonumber
\begin{aligned}
\iint &|\nabla \partial_Y w_1|^2+\int_0^L \left(-\partial_Y^2w_1+\frac{h_e}{u_e}\partial_Y^2 w_2\right)\partial_Y w_1(x,0)\\
&-\int_0^\infty \partial_{xY}w_1\cdot \partial_Yw_1\big|_{x=0}^{x=L}+\int_0^\infty \frac{h^0_e}{u^0_e}\partial_{xY}w_2\cdot \partial_Y w_1\big|_{x=0}^{x=L}\\
=&\iint \frac{h^0_e}{u^0_e}\nabla \partial_Y w_1\cdot\nabla \partial_Yw_2+
\iint\partial_Y\left(\frac{h^0_e}{u^0_e}\right)\partial_{YY}w_2\cdot\partial_Y w_1\\
&-\int_0^L\partial_Y \left(\frac{h^0_e}{u^0_e}\right)\partial_Y w_1\cdot \partial_Yw_2\big|_{Y=0} +\iint \partial_Y\left(\frac{h^0_e}{u^0_e}\right)\nabla w_2\cdot \nabla \partial_Yw_1\\
&+\iint \partial_Y^2\left(\frac{h^0_e}{u^0_e}\right)\partial_Y w_1\cdot \partial_Yw_2-\iint G_e\partial_{YY} w_1.
\end{aligned}
\end{equation}
All of boundary integrals on the left-hand side of above equality disappear due to the zero boundary conditions in (\ref{4.016}) and (\ref{4.021}). Hence, it follows that
\begin{equation}\nonumber
\begin{aligned}
\iint& |\nabla \partial_Y w_1|^2
=\iint \frac{h^0_e}{u^0_e}\nabla \partial_Y w_1\cdot\nabla \partial_Yw_2+
\iint\partial_Y\left(\frac{h^0_e}{u^0_e}\right)\partial_{YY}w_2\cdot\partial_Y w_1\\
&-\int_0^L\partial_Y \left(\frac{h^0_e}{u^0_e}\right)\partial_Y w_1\cdot \partial_Yw_2\big|_{Y=0} +\iint \partial_Y\left(\frac{h^0_e}{u^0_e}\right)\nabla w_2\cdot \nabla \partial_Yw_1\\
&+\iint \partial_Y^2\left(\frac{h^0_e}{u^0_e}\right)\partial_Y w_1\cdot \partial_Yw_2-\iint G_e\partial_{YY} w_1\triangleq \sum_{i=1}^6 I_i.
\end{aligned}
\end{equation}
Each term in the right-hand side is estimated as follows. First
$$I_1=\iint \frac{h^0_e}{u^0_e}\nabla \partial_Y w_1\cdot\nabla \partial_Yw_2 \lesssim \sup_Y \left|\frac{h^0_e}{u^0_e}\right|\Vert \nabla \partial_Yw_1\Vert_{L^2}\Vert \nabla \partial_Yw_2\Vert_{L^2},$$
and
\begin{equation}\nonumber
\begin{aligned}
I_2=\iint\partial_Y\left(\frac{h^0_e}{u^0_e}\right)\partial_{YY}w_2\cdot\partial_Y w_1\lesssim &\left\Vert \partial_Y\left(\frac{h^0_e}{u^0_e}\right)\right\Vert_{L^\infty}\Vert \nabla \partial_Yw_2\Vert_{L^2}\Vert  \partial_Y w_1\Vert_{L^2}\\
\lesssim &\Vert \nabla \partial_Yw_2\Vert_{L^2}\Vert  w_1\Vert_{H^1},
\end{aligned}
\end{equation}
By trace theorem, we have
\begin{equation}\nonumber
\begin{aligned}
I_3=\bigg|\int_0^L&\partial_Y \left(\frac{h^0_e}{u^0_e}\right)\partial_Y w_1\cdot \partial_Yw_2\big|_{Y=0} \bigg|\\
\lesssim &\left\Vert \partial_Y\left(\partial_Y \left(\frac{h^0_e}{u^0_e}\right)\partial_Y w_1\right)\right\Vert_{L^2}\Vert \partial_Yw_2\Vert_{L^2}\\
&+\left\Vert \partial_Y \left(\frac{h^0_e}{u^0_e}\right)\partial_Y w_1\right\Vert_{L^2}\Vert \partial_Y^2 w_2\Vert_{L^2}\\
\lesssim &\Vert \nabla \partial_Y w_1\Vert_{L^2}\Vert \partial_Y w_2\Vert_{L^2}+\Vert \partial_Y w_1\Vert_{L^2}\Vert \nabla \partial_Y w_2\Vert_{L^2}\\
\lesssim&\Vert \nabla \partial_Y w_1\Vert_{L^2}\Vert   w_2\Vert_{H^1}+\Vert  w_1\Vert_{H^1}\Vert \nabla \partial_Y w_2\Vert_{L^2}.
\end{aligned}
\end{equation}
The Cauchy-Schwatz inequality gives that
\begin{equation}\nonumber
\begin{aligned}
I_4=\iint \partial_Y\left(\frac{h^0_e}{u^0_e}\right)\nabla w_2\cdot \nabla \partial_Yw_1\lesssim\Vert \nabla \partial_Yw_1\Vert_{L^2}\Vert  w_2\Vert_{H^1},
\end{aligned}
\end{equation}
\begin{equation}\nonumber
\begin{aligned}
I_5=\iint \partial_Y^2\left(\frac{h^0_e}{u^0_e}\right)\partial_Y w_1\cdot \partial_Yw_2 \lesssim\Vert  w_1\Vert_{H^1}\Vert  w_2\Vert_{H^1},
\end{aligned}
\end{equation}
and
\begin{equation}\nonumber
\begin{aligned}
I_6=-\iint G_e\partial_{YY} w_1\lesssim \Vert G_e\Vert_{L^2}\Vert \partial_Y^2 w_1\Vert_{L^2}.
\end{aligned}
\end{equation}
Therefore, we deduce that
\begin{equation}\label{4.023}
\begin{aligned}
\Vert \nabla \partial_Y w_1\Vert_{L^2}^2\lesssim &\sup_Y \left|\frac{h^0_e}{u^0_e}\right|\Vert \nabla \partial_Yw_1\Vert_{L^2}\Vert \nabla \partial_Yw_2\Vert_{L^2}\\
&+\Vert \nabla \partial_Y w_1\Vert_{L^2}\Vert   w_2\Vert_{H^1}+\Vert  w_1\Vert_{H^1}\Vert \nabla \partial_Y w_2\Vert_{L^2}\\
&+\Vert  w_1\Vert_{H^1}\Vert  w_2\Vert_{H^1}+\Vert G_e\Vert_{L^2}\Vert \partial_Y^2 w_1\Vert_{L^2}.
\end{aligned}
\end{equation}
On the other hand, we find that
\begin{equation}\label{4.023a}
\begin{aligned}
\nabla \partial_Y w_2=&\begin{pmatrix} \partial_Y \left(\frac{h^0_e}{u^0_e}\right)\partial_x w_1 +\frac{h^0_e}{u^0_e}\partial_{xY}w_1\\ 2\partial_Y\left(\frac{h^0_e}{u^0_e}\right)\partial_Yw_1+\partial_Y^2\left(\frac{h^0_e}{u^0_e}\right)w_1+\frac{h^0_e}{u^0_e}\partial_Y^2w_1\end{pmatrix}\\
&+\nabla\partial_Y \left(\frac{h^0_e}{u^0_e}B_v+\frac{b}{u^0_e}-B_g\right).
\end{aligned}
\end{equation}
Thus we have
\begin{equation}\label{4.024}
\begin{aligned}
\Vert \nabla \partial_Y w_2 \Vert_{L^2}\lesssim 1+\Vert w_1\Vert_{H^1}+\sup_Y\left|\frac{h^0_e}{u^0_e}\right|\Vert \nabla \partial_Y w_1\Vert_{L^2}.
\end{aligned}
\end{equation}
Combining (\ref{4.023}) and (\ref{4.024}), using the Young's inequality and the smallness of $\sup\limits_Y\left|\frac{h^0_e}{u^0_e}\right|$, we have
\begin{equation}\label{4.025}
\begin{aligned}
\Vert \nabla \partial_Y w_i \Vert_{L^2}\leq C, \ \ i=1,2.
\end{aligned}
\end{equation}
Similar arguments yield the estimates for $\partial_{xx} w_i\ (i=1,2)$ by using the equations (\ref{4.016}). Thus, the estimates of $w_i\ (i=1,2)$ in $H^2$-norms are obtained, and hence of $v^1_e,g^1_e$, uniformly in small $L>0$.

For the $L^\infty$ norms, since $w_i=0$ on boundaries, then for $i=1,2$, one gets that
\begin{equation}\label{4.026}
\begin{aligned}
|w_i(x,z)| \leq &\int_0^x|\partial_x w_i(s,z)|\mathrm{d}s\\
 \lesssim &\int_0^x\left(\int_0^z |\partial_x w_i \partial_{xY}w_i|(s,\eta)\mathrm{d}\eta\right)^{1/2}\mathrm{d}s\\
 \lesssim&\sqrt{x}\Vert \partial_x w_i\Vert_{L^2}^{\frac{1}{2}}\Vert \partial_{xY} w_i\Vert_{L^2}^{\frac{1}{2}} \lesssim \sqrt{L},
\end{aligned}
\end{equation}
where the uniform $H^2$ bounds of $w_i\ (i=1,2)$ are used in the last inequality.

For $n\geq 1$, to derive the weighted estimates, we consider the following elliptic problem for $\langle Y \rangle^n w_i\ (i=1,2)$, which satisfies
\begin{equation}\label{4.027}
\left \{
\begin{aligned}
&-\Delta \left(\langle Y\rangle^nw_1\right)+\frac{h^0_e}{u^0_e}\Delta \left(\langle Y\rangle^nw_2\right)=\frac{\langle Y\rangle^n}{u^0_e}\big(E_b-F_e-\partial_Y^2 u^0_e\cdot w_1\\
&\quad \quad +\partial_Y^2h^0_e\cdot w_2\big)-w_1\partial_Y^2\langle Y\rangle^n-2\partial_Yw_1\cdot \partial_Y \langle Y\rangle^n\\
&\quad\quad +\frac{h^0_e}{u^0_e}\big(w_2\partial_Y^2\langle Y\rangle^n+2\partial_Yw_2\cdot\partial_Y \langle Y\rangle^n\big),\\
&\langle Y\rangle^nw_2=\frac{h^0_e}{u^0_e}\langle Y\rangle^n(w_1+B_v)+\frac{b}{u^0_e}\langle Y\rangle^n-\langle Y\rangle^nB_g,\\
\end{aligned}
\right.
\end{equation}
with the homogenous boundary conditions.

By the induction, suppose $\langle Y\rangle^{n-1}w_1,\langle Y\rangle^{n-1}w_2$ are uniformly bounded in $H^2$-norm and then the right-hand side of the two equations in (\ref{4.027}) are uniformly bounded in $H^1$. Following the same arguments as above for the unweighed estimates yield that
$$\Vert \langle Y\rangle^n (w_1,w_2)\Vert_{H^2}\leq  C\quad \hbox{for\ any}\quad n\geq 1.$$

Next, we turn to derive higher regularity estimates for $w_i\ (i=1,2)$. It is direct to see that
$$\Vert \langle Y \rangle^n\partial_Y^k(E_b-F_e)\Vert_{L^q} \leq C (1+\epsilon^{-k+\frac{1}{q}}), \ \ q\geq1, \ \ n,k\geq 0.$$
Consider the following problems for $\partial_Y w_i$ and $\partial_Y^2 w_i\ (i=1,2)$ respectively:
\begin{equation}\label{4.028}
\left \{
\begin{aligned}
&-\Delta \partial_Y w_1+\frac{h^0_e}{u^0_e}\Delta \partial_Y w_2+\partial_Y \left(\frac{h^0_e}{u^0_e}\right)\Delta  w_2=\partial_Y G_e,\\
&w_2=\frac{h^0_e}{u^0_e}(w_1+B_v)+\frac{b}{u^0_e}-B_g,\\
&\partial_Y w_i|_{x=0,L}=\left(-\partial_Y^2 w_1+\frac{h_e}{u_e}\partial_Y^2 w_2\right)\bigg|_{Y=0}=0
\end{aligned}
\right.
\end{equation}
and
\begin{equation}\label{4.029}
\left \{
\begin{aligned}
&-\Delta \partial_Y^2 w_1+\frac{h^0_e}{u^0_e}\Delta \partial_Y^2 w_2\\
&\quad \quad =\partial_Y^2 G_e-2\partial_Y\left(\frac{h^0_e}{u^0_e}\right)\Delta \partial_Y w_2-\Delta w_2\cdot\partial_Y^2\left(\frac{h^0_e}{u^0_e}\right),\\
&\partial_Y w_i|_{x=0,L}=\left(-\partial_Y^2 w_1+\frac{h_e}{u_e}\partial_Y^2 w_2\right)\bigg|_{Y=0}=0.
\end{aligned}
\right.
\end{equation}
Notice that the right-hand side terms in the first equations in both (\ref{4.028}) and (\ref{4.029}) are bounded by $C\epsilon^{-1/2}$ and $C\epsilon^{-3/2}$ in $L^2$-norm, respectively. Therefore, one can deduce that
\begin{equation}\label{4.030}
\Vert \langle Y\rangle^n\partial_Y^kw_i\Vert_{H^2}\leq    C \epsilon^{-k+1/2}, \ \ k=1,2, \ i=1,2,\ \ n\geq0.
\end{equation}
It is noted that the above estimates (\ref{4.030}) can be achieved by the above standard arguments. For example, testing the first equation in (\ref{4.028}) by $\langle Y\rangle^{2n}\partial_Y w_1$ and combining the second equation in (\ref{4.028}), using the smallness of $\sup\limits_Y\left|\frac{h^0_e}{u^0_e}\right|$, one can obtain the desired estimates for $\Vert \langle Y\rangle^n\partial_Yw_i\Vert_{H^2}$ in (\ref{4.030}). For the case that $k=2$, the arguments are similar.

To obtain the desired $H^3, H^4$ norms of $w_i\ (i=1,2)$, it remains to estimate $\partial_x^3 w_i$ in $L^2$ and $H^1$, respectively. Recall that
\begin{equation}\label{4.031}
\begin{aligned}
-\partial_x^3w_1+\frac{h^0_e}{u^0_e}\partial_x^3w_2&=\partial_{xYY}w_1-\frac{h^0_e}{u^0_e}\partial_{xYY}w_2+(-\Delta \partial_xw_1+\frac{h^0_e}{u^0_e}\Delta \partial_xw_2)\\
&=\partial_{xYY}w_1-\frac{h^0_e}{u^0_e}\partial_{xYY}w_2+\partial_xG_e.
\end{aligned}
\end{equation}
Applying $\partial_x^3$ on the second equation in (\ref{4.028}) and combining (\ref{4.031}), one can obtain the weighted $L^2$ and $H^1$ norms on $\partial_x^3 w_i\ (i=1,2)$ by using the similar arguments as above, and hence the full weighted $H^3$ and $H^4$ estimates on $w_i\ (i=1,2)$.

To derive the $W^{k,q}$ estimates, after making the odd extension to $Y<0$ for (\ref{4.020}), the estimates are derived by the standard elliptic theory in $[0,L]\times \mathbb{R}$. It should be pointed out that the construction of boundary layer corrector guarantees that odd extension of $G_e\in W^{2,q}$ is possible. Other arguments are similar to those in \cite{YGuo2}, which are standard in the elliptic theory, see \cite{YGuo2} and the references therein for details.
\end{proof}

\subsection{Ideal MHD profiles}\label{IMP}
After constructing $(v^1_e,g^1_e)$, we define the ideal MHD profiles $(u^1_e,v^1_e,h^1_e,g^1_e,p^1_e)$ as follows. Let $(v^1_e,g^1_e)$ be determined in the Lemma \ref{1stMHD}, thanks to (\ref{4.09}) and the divergence-free conditions, we take
\begin{equation}\label{4.032}
\left \{
\begin{aligned}
u^1_e&:=u^1_b(Y)-\int_0^x\partial_Y v^1_e(s,Y)\mathrm{d}s,\\
h^1_e&:=h^1_b(Y)-\int_0^x\partial_Y g^1_e(s,Y)\mathrm{d}s,\\
p^1_e&:=\int_Y^\infty \left[u^0_e\partial_x v^1_e-h^0_e\partial_x g^1_e\right](x,\theta)\mathrm{d}\theta,
\end{aligned}
\right.
\end{equation}
where $(u^1_b,h^1_b)(Y):=(u^1_e,h^1_e)(0,Y).$

Let $\eta(y)$ be a smooth cut-off function satisfying
\begin{equation}\nonumber
\eta(y)=\left \{
\begin{aligned}
&1,y\in [0,1],\\
&0,y\in [2,\infty).
\end{aligned}
\right.
\end{equation}
To ensure that the boundary conditions of the magnetic field can be preserved in the construction of the approximate solutions, we introduce a boundary corrector $\rho(x,y)$ as follows
\begin{equation}\nonumber
\rho(x,y):=-\overline{\partial_Y h^1_e}(x)\cdot y\eta(y).
\end{equation}
It is direct to check that
$$\partial_y\rho(x,0)=-\overline{\partial_Y h^1_e}(x).$$
Then, we define
\begin{equation}\label{4.032i}
\left \{
\begin{aligned}
\tilde{h^1_e}&=h^1_e+\sqrt{\epsilon}\rho,\\
\tilde{g^1_e}&=g^1_e-\epsilon\int_0^y\partial_x \rho(x,s)\mathrm{d}s.
\end{aligned}
\right.
\end{equation}
It is easy to get that
$$\partial_x \tilde{h^1_e}+\partial_Y \tilde{g^1_e}=0, \ \ (\partial_Y \tilde{h^1_e},\tilde{g^1_e})(x,0)=(0,\overline{g^1_e}(x)).$$
Without causing confusion, we still use $h^1_e, g^1_e$ to replace $\tilde{h^1_e}, \tilde{g^1_e}$ for the simplicity of the presentation from now on.

By the definitions of $(u^1_e,h^1_e)$ in (\ref{4.032}) and (\ref{4.032i}) and Lemma \ref{1stMHD}, we have
\begin{equation}\label{4.033}
\Vert (u_e^1,h^1_e)\Vert_{L^\infty}+\Vert \langle Y\rangle^n(u^1_e,h^1_e)\Vert_{H^1}\leq C_0, \ \ n\geq 0.
\end{equation}

By a direct calculation, the ideal MHD profiles constructed above satisfy
\begin{equation}\label{4.034}
\begin{aligned}
u^0_e\partial_x u^1_e&+v^1_e\partial_Yu^0_e-h^0_e\partial_x h^1_e-g^1_e\partial_Y h^0_e+\partial_x p^1_e\\
&=-\int_Y^\infty E_b(x,\theta)\mathrm{d}\theta-\sqrt{\epsilon}h^0_e\partial_x\rho-\epsilon\int_0^y\partial_x \rho(x,s)\mathrm{d}s\cdot\partial_Yh^0_e,
\end{aligned}
\end{equation}
where (\ref{4.011}) is used. And a new error term is created, which will be merged into (\ref{4.001}), that is
\begin{equation}\label{4.035}
\begin{aligned}
R^{u,1}=&\sqrt{\epsilon}\left[(v^0_p+v^1_e)\partial_Y u^1_e-(g^0_p+g^1_e)\partial_Y h^1_e\right]-\nu \epsilon \partial_Y^2u^1_e\\
&+\int_Y^\infty E  _b(x,\theta)\mathrm{d}\theta-\sqrt{\epsilon}h^0_e\partial_x\rho-\epsilon\int_0^y\partial_x \rho(x,s)\mathrm{d}s\cdot\partial_Yh^0_e.
\end{aligned}
\end{equation}

Next, we estimate each error term in (\ref{4.035}). Keep in mind that we work with the coordinates $(x,y)$, whereas the ideal MHD flows are evaluated at $(x,Y)=(x,\sqrt{\epsilon}y)$. Thanks to Lemma \ref{1stMHD}, we have
\begin{equation}\nonumber
\begin{aligned}
\sqrt{\epsilon}&\|\left[(v^0_p+v^1_e)\partial_Y u^1_e-(g^0_p+g^1_e)\partial_Y h^1_e\right]\|_{L^2} \\
\lesssim &\sqrt{\epsilon}\bigg(\Vert v^0_p+v^1_e\Vert_{L^\infty}\Vert \partial_Y u^1_e(x,\sqrt{\epsilon}y)\Vert_{L^2}\\
&+\Vert g^0_p+g^1_e\Vert_{L^\infty}\Vert \partial_Y h^1_e(x,\sqrt{\epsilon}y)\Vert_{L^2}\bigg)\\
\lesssim&\epsilon^{\frac{1}{4}}\left(\Vert \partial_Y u^1_e(x,Y)\Vert_{L^2}+\Vert \partial_Y h^1_e(x,y)\Vert_{L^2}\right).
\end{aligned}
\end{equation}
Similarly,
\begin{equation}\nonumber
\begin{aligned}
\nu \epsilon \Vert \partial_Y^2u^1_e(x,\sqrt{\epsilon}y)\Vert_{L^2} \lesssim \epsilon^{\frac{3}{4}}\Vert \partial_Y^2u^1_b(Y)\Vert_{L^2}+\epsilon^{\frac{3}{4}}\Vert \langle Y\rangle^n\partial_Y^3v^1_e\Vert_{L^2} \lesssim \epsilon^{\frac{1}{4}},
\end{aligned}
\end{equation}
\begin{equation}\nonumber
\begin{aligned}
\left\Vert \int_{\sqrt{\epsilon}y}^\infty E_b(x,\theta)\mathrm{d}\theta\right\Vert_{L^2}\lesssim\epsilon^{-\frac{1}{4}}\Vert \langle Y\rangle^nE_b\Vert_{L^2}\lesssim\epsilon^{\frac{1}{4}},
\end{aligned}
\end{equation}
\begin{equation}\nonumber
\begin{aligned}
\sqrt{\epsilon}\Vert h^0_e\partial_x\rho \Vert_{L^2} \lesssim &\sqrt{\epsilon}\Vert h^0_e\Vert_{L^\infty}\Vert \overline{\partial_Y^2g^1_e}(x)\Vert_{L^2(0,L)}\Vert y\eta(y)\Vert_{L^2(0,\infty)}\\
\lesssim & \sqrt{\epsilon}\Vert\partial_Y^3 g^1_e\Vert_{L^2}^{\frac{1}{2}}\Vert\partial_Y^2 g^1_e\Vert_{L^2}^{\frac{1}{2}} \lesssim\epsilon^{\frac{1}{4}},
\end{aligned}
\end{equation}
and also,
\begin{equation}\nonumber
\begin{aligned}
\epsilon\left\Vert \int_0^y\partial_x \rho(x,s)\mathrm{d}s\cdot\partial_Yh^0_e \right\Vert_{L^2}\lesssim \epsilon^{\frac{1}{4}},
\end{aligned}
\end{equation}
in which, the trace theorem, (\ref{4.017a}) and Lemma \ref{1stMHD} have been used in the above estimates.

Therefore, we conclude that
\begin{equation}\label{4.036}
\begin{aligned}
\Vert R^{u,1}\Vert_{L^2}\leq C\epsilon^{\frac{1}{4}}.
\end{aligned}
\end{equation}

Similar arguments yield that
\begin{equation}\label{4.037}
\begin{aligned}
\Vert (R^{h,1},R^{v,0},R^{g,0})\Vert_{L^2}\leq C\epsilon^{\frac{1}{4}}.
\end{aligned}
\end{equation}

Finally, $E_1,E_3$ and $E_4$ defined in (\ref{2.6a}) are estimated as follows:
\begin{equation}\label{4.038}
\begin{aligned}
\Vert (E_1,E_3) \Vert_{L^2} \lesssim\epsilon^{3/4},\ \ \
\Vert E_4 \Vert_{L^2} \lesssim\epsilon^{1/4},
\end{aligned}
\end{equation}
which is resulted from the estimates obtained before and a similar argument as that in \cite{YGuo2}.

\subsection{First-order MHD boundary layer correctors}\label{1stmhdbdprofile}

In this subsection, we are devoted to constructing the MHD boundary layer correctors $(u^1_p,v^1_p,h^1_p,g^1_p,p^1_p)$ by solving (\ref{4.01}).
For convenience, we define
$$u^0:=u^0_e(\sqrt{\epsilon}y)+u^0_p(x,y), \qquad \qquad h^0:=h^0_e(\sqrt{\epsilon}y)+h^0_p(x,y).$$
Without causing confusion, the superscript 1 is omitted here for simplification of notation,
then $(u_p,v_p,h_p,g_p,p^1_p)$ enjoy that
\begin{equation}\label{4.039}
\left \{
\begin{aligned}
u^0 \partial_x &u_p+u_p \partial_x u^0+v_p\partial_y u^0+(v^0_p+v^1_e)\partial_y u_p-\nu \partial_y^2 u_p\\
&-h^0\partial_x h_p-h_p\partial_x h^0- g_p\partial_y h^0-(g^0_p+g^1_e)\partial_y h_p\\
=&-\partial_Y u^0_e[y\partial_x u^0_p+v^0_p]-y\partial_Y v^1_e\partial_y u^0_p-u^1_e\partial_x u^0_p-u^0_p\partial_x u^1_e-\partial_x p^1_p\\
&+\partial_Yh^0_e[y\partial_x h^0_p+g^0_p]+y\partial_Y g^1_e \partial_y h^0_p+h^1_e\partial_x h^0_p+h^0_p\partial_x h^1_e=:F^1_p,\\
u^0 \partial_x &h_p+u_p \partial_x h^0+v_p\partial_y h^0+(v^0_p+v^1_e)\partial_y h_p-\kappa \partial_y^2 h_p\\
&-h^0\partial_x u_p-h_p\partial_x u^0- g_p\partial_y u^0-(g^0_p+g^1_e)\partial_y u_p\\
=&-\partial_Y h^0_ev^0_p-y\partial_x h^0_p\partial_Yu^0_e-y\partial_Y v^1_e\partial_y h^0_p-u^0_p\partial_x h^1_e-u^1_e\partial_x h^0_p\\
&+\partial_Y u^0_e g^0_p+y\partial_Y h^0_e\partial_x u^0_p+y\partial_Y g^1_e\partial_y u^0_p+h^1_e\partial_x u^0_p+h^0_p\partial_x u^1_e=:F^2_p,
\end{aligned}
\right.
\end{equation}
Notice that the source term $F^1_p$ includes the unknown pressure $p^1_p$. Since
$$p^1_p=p^1_p(x),$$
then evaluating the first equation in (\ref{4.039}) at $y=\infty$ yields that $\partial_x p^1_p=0$. The system of equations (\ref{4.039}) should be solved  with the divergence free conditions:
\begin{align}
\label{DF}\partial_x u_p+\partial_y v_p=0,\qquad \qquad \partial_x h_p+\partial_y g_p=0,
\end{align}
and the following boundary conditions.
\begin{equation}\label{4.040}
\left \{
\begin{aligned}
&(u_p,h_p)(0,y)=(\overline{u}_1,\overline{h}_1)(y),\\
&(u_p,\partial_y h_p)(x,0)=-(u^1_e,\partial_Y h^0_e)(x,0),\\
&(v_p,g_p)(x,0)=(0,0),\\
&(u_p,h_p) \to (0,0) \ \ \ \mathrm{as} \ \ y \to \infty.
\end{aligned}
\right.
\end{equation}

To determine $(u_p,v_p,h_p,g_p)$, we will study the following system of equations, instead of (\ref{4.039}):
\begin{equation}\label{4.041}
\left \{
\begin{aligned}
(u_e&+u_p^0)\partial_x u_p+u_p\partial_x u^0_p+v_p\partial_y u^0_p+(v^0_p+\overline{v^1_e})\partial_y u_p-\nu \partial_y^2 u_p\\
&-(h_e+h^0_p)\partial_x h_p-h_p\partial_x h^0_p-g_p\partial_y h^0_p-(g^0_p+\overline{g^1_e})\partial_y h_p\\
=&-\overline{\partial_Y u_e^0}[y\partial_x u^0_p+v^0_p]-y\overline{\partial_Yv^1_e}\partial_y u^0_p-\overline{u_e^1}\partial_x u^0_p-u^0_p\overline{\partial_x u^1_e}\\
&+\overline{\partial_Y h_e^0}[y\partial_x h^0_p+g^0_p]+y\overline{\partial_Yg^1_e}\partial_y h^0_p+\overline{h_e^1}\partial_x h^0_p+h^0_p\overline{\partial_x h^1_e}=:\tilde{F}^1_p,\\
(u_e&+u^0_p)\partial_x h_p+u_p\partial_x h^0_p+v_p\partial_y h^0_p+(v^0_p+\overline{v^1_e})\partial_y h_p-\kappa \partial_y^2 h_p\\
&-(h_e+h^0_p)\partial_x u_p-h_p\partial_x u^0_p-g_p\partial_y u^0_p-(g^0_p+\overline{g^1_e})\partial_y u_p\\
=&-\overline{\partial_Y h^0_e}v^0_p-y\overline{\partial_Yu^0_e}\partial_x h^0_p-y\overline{\partial_Y v^1_e}\partial_y h^0_p-\overline{\partial_x h^1_e}u^0_p-\overline{u^1_e}\partial_xh^0_p\\
&+\overline{\partial_Yu^0_e}g^0_p+y\overline{\partial_Yh^0_e}\partial_x u^0_p+y\overline{\partial_Yg^1_e}\partial_yu^0_p+\overline{h^1_e}\partial_x u^0_p+\overline{\partial_x u^1_e}h^0_p=:\tilde{F}^2_p.
\end{aligned}
\right.
\end{equation}
Consequently, there will produce some new additional error terms in the order of $O(\sqrt{\epsilon})$, which can be summarized as follows:
\begin{equation}\label{4.042}
\left \{
\begin{aligned}
Er_1:=&\sqrt{\epsilon}\bigg\{(u^0_e-u_e)\partial_x u_p+v_p\partial_y u^0_e+(v^1_e-\overline{v^1_e})\partial_y u_p-(h^0_e-h_e)\partial_x h_p\\
&-g_p\partial_y h^0_e-(g^1_e-\overline{g^1_e})\partial_y h_p+(\partial_Y u^0_e -\overline{\partial_Yu^0_e})(y\partial_x u^0_p+v^0_p)\\
&+y(\partial_Yv^1_e-\overline{\partial_Yv^1_e})\partial_yu^0_p+(u^1_e-\overline{u^1_e})\partial_xu^0_p+(\partial_x u^1_e-\overline{\partial_x u^1_e})u^0_p\\
&-(\partial_Yh^0_e-\overline{\partial_Yh^0_e})(y\partial_x h^0_p+g^0_p)-y(\partial_Yg^1_e-\overline{\partial_Y g^1_e})\partial_yh^0_p\\
&-(h^1_e-\overline{h^1_e})\partial_x u^0_p-(\partial_xh^1_e-\overline{\partial_x h^1_e})h^0_p\bigg\},\\
Er_2:=&\sqrt{\epsilon}\bigg\{(u^0_e-u_e)\partial_x h_p+v_p\partial_yh^0_e+(v^1_e-\overline{v^1_e})\partial_y h_p\\
&-(h^0_e-h_e)\partial_x u_p-g_p\partial_y u^0_e-(g^1_e-\overline{g^1_e})\partial_yu_p\\
&+(\partial_Yh^0_e-\overline{\partial_Yh^0_e})v^0_p+y\partial_x h^0_p(\partial_Yu^0_e-\overline{\partial_Y u^0_e})\\
&+y\partial_y h^0_p(\partial_Y v^1_e-\overline{\partial_Yv^1_e})+(\partial_x h^1_e-\overline{\partial_xh^1_e})u^0_p+(u^1_e-\overline{u^1_e})\partial_x h^0_p\\
&-(\partial_Yu^0_e-\overline{\partial_Y u^0_e})g^0_p-y(\partial_Yh^0_e-\overline{\partial_Yh^0_e})\partial_x u^0_p\\
&-y(\partial_Yg^1_e-\overline{\partial_Yg^1_e})\partial_yu^0_p-(h^1_e-\overline{h^1_e})\partial_x u^0_p-(\partial_x u^1_e-\overline{\partial_xu^1_e})h^0_p\bigg\}.
\end{aligned}
\right.
\end{equation}

We will establish the well-posedness of the solution $(u_p,v_p,h_p,g_p)$ to the problem (\ref{4.041}) and (\ref{DF})-(\ref{4.040}). To this end, the similar ideas used to achieve the well-posedness of $(u^0_p,v^0_p,h^0_p,g^0_p)$ in Subsection \ref{0th} are needed again.

More precisely, from the divergence free conditions, we rewrite the second equation in (\ref{4.041}) as
\begin{equation}\label{4.043}
\begin{aligned}
&\partial_y \big[-(u_e+u^0_p)g_p-(g^0_p+\overline{g^1_e})u_p+(v^0_p+\overline{v^1_e})h_p+(h_e+h^0_p)v_p\big]-\kappa \partial_y^2 h_p\\
=&\partial_y\big[-y\overline{\partial_Yh^0_e}v^0_p+y\overline{\partial_Yu^0_e}g^0_p-y\overline{\partial_Yv^1_e}h^0_p
+y\overline{\partial_Yg^1_e}u^0_p+\overline{u^1_e}g^0_p-\overline{h^1_e}v^0_p\big].
\end{aligned}
\end{equation}
Introduce the stream function $\tilde{\psi}$, such that
$$\tilde{\psi}=\int_0^y h_p(x,s)\mathrm{d}s,$$
then from $\partial_x h_p+\partial_y g_p=0$ and $g_p(x,0)=0$, we have
$$\partial_x \tilde{\psi}=-g_p.$$
Integrating the equation (\ref{4.043}) over $[0,y]$, one has
\begin{equation}\label{4.044}
\begin{aligned}
&(u_e+u^0_p)\partial_x \tilde{\psi}+(v^0_p+\overline{v^1_e})\partial_y\tilde{\psi}-(g^0_p+\overline{g^1_e})u_p+(h_e+h^0_p)v_p\big]-\kappa \partial_y^2 \tilde{\psi}\\
=&-y\overline{\partial_Yh^0_e}v^0_p+y\overline{\partial_Yu^0_e}g^0_p-y\overline{\partial_Yv^1_e}h^0_p
+y\overline{\partial_Yg^1_e}u^0_p+\overline{u^1_e}g^0_p-\overline{h^1_e}v^0_p\\
&+\overline{u^1_eg^1_e-h^1_ev^1_e+\kappa\partial_Yh^0_e},
\end{aligned}
\end{equation}
where the following boundary conditions are used.
$$(v^0_p,g^0_p,v_p,g_p,\partial_yh_p)(x,0)=(-\overline{v^1_e},-\overline{g^1_e},0,0,-\overline{\partial_Yh^0_e}).$$

Recall from Subsection \ref{0th} that for small $L>0$, we have
$$u_e+u^0_p-(h_e+h^0_p)\geq c_0>0,$$
then the system (\ref{4.041}) is non-degenerate.

For simplicity, here we only give the outline about the application of the energy estimates method used in Subsection \ref{0th}, also see \cite{CLiu2,CLiu3}.

First, the weighted $L^2$-estimates of solution to (\ref{4.041})
$$D^\alpha (u_p,h_p), \ \ D^\alpha=\partial_x^\beta\partial_y^k, \ \alpha=(\beta,k) \in \mathbb{N}^2, \ |\alpha|\leq m, \ \beta \leq m-1 $$
can be done by the standard energy methods.

Next, to derive the weighted $L^2$-estimates of highest order tangential derivatives
$$\partial_x^\alpha (u_p,h_p), \ \ |\alpha|=m, $$
from Proposition \ref{wellposedness}, it holds, for small $L>0$, that
$$h_e+h^0_p\geq \frac{\vartheta_0}{2}>0.$$
Introduce the new unknowns as follows.
$$u^\alpha_p:=\partial_x^\alpha u_p-\frac{\partial_y u^0_p}{h_e+h^0_p}\partial_x^\alpha \tilde{\psi},\ \ h^\alpha_p:=\partial_x^\alpha h_p-\frac{\partial_y h^0_p}{h_e+h^0_p}\partial_x^\alpha \tilde{\psi},$$
combining (\ref{4.041}) and  (\ref{4.044}), one can derive the equations of $u^\alpha_p,h^\alpha_p$, in which the terms involving $\partial_x^\alpha(v_p,g_p)$ vanish. Therefore, it is possible to obtain the weighted $L^2$-estimates of $u^\alpha_p,h^\alpha_p$.

Finally, similar as that in the Appendix \ref{ap1}, one need to prove the equivalence of $L^2$-norms between $\partial_x^\alpha (u_p,h_p)$ and $u^\alpha_p,h^\alpha_p$ for $|\alpha|=m$. With this, we will obtain the desired estimates of $\partial_x^\alpha (u_p,h_p)$ and close the whole energy estimates. Therefore, the well-posedness and weighted estimates of the solutions $(u_p,v_p,h_p,g_p)$ can be concluded as follows.
\begin{proposition}\label{1thBoundary}
Let $(u^0_p,v^0_p,h^0_p,g^0_p)$ and $(u^1_e,v^1_e,h^1_e,g^1_e)$ be the solutions constructed in Proposition \ref{wellposedness} and Subsection \ref{1stmhdbdprofile}, respectively. Then there exists $0<L_2\leq L_1$, such that the problem (\ref{4.041})and (\ref{DF})-(\ref{4.040}) admits a local-in-$x$ smooth solution $(u_p,v_p,h_p,g_p)$ in $[0,L_2]\times [0,\infty)$ with
\begin{equation}\label{4.045}
\begin{aligned}
\Vert &(u_p,v_p,h_p,g_p)\Vert_{L^\infty}+\sup\limits_{0\leq x \leq L_2}\Vert\langle y\rangle^l\partial_{yy} (v_p,g_p)\Vert_{L^2(0,\infty)}\\
&+\Vert \langle y\rangle^l\partial_{xy}(v_p,g_p)\Vert_{L^2}\leq C(L,\zeta)\epsilon^{-\zeta},\\
\sup\limits_{0\leq x \leq L_2}&\Vert \langle y\rangle^l\partial_{xyy}(v_p,g_p)\Vert_{L^2(0,\infty)}+\Vert \langle y\rangle^l\partial_{xxy}(v_p,g_p)\Vert_{L^2}\lesssim C(L)\epsilon^{-1},
\end{aligned}
\end{equation}
here $\zeta=1-\frac{1}{q}>0$ is arbitrary small constant provided that $q$ is close to 1 enough.
\end{proposition}

Moreover, by performing the similar arguments as those in the Subsection \ref{IMP}, we derived the estimates for the error terms $Er_i(i=1,2)$ as follows:
\begin{equation}\label{3.046}
\begin{aligned}
\Vert (Er_1,Er_2)\Vert_{L^2} \leq C(L,\kappa)\epsilon^{3/4-\zeta}.
\end{aligned}
\end{equation}

\subsection{Cut-off boundary layers}\label{cutofflayer}
Next, we will introduce the modified boundary layer functions which will be used in the subsequent analysis. Let $(u_p,v_p,h_p,g_p)$ be constructed as in the Subsection \ref{1stmhdbdprofile}.
Define a cut-off function $\chi(\cdot)$ with its support being contained in $[0,1]$ and introduce
\begin{equation}\label{4.047}
\left \{
\begin{aligned}
(u^1_p,h^1_p)&:=\chi(\sqrt{\epsilon}y)(u_p,h_p)+\sqrt{\epsilon}\chi'(\sqrt{\epsilon}y)\int_0^y (u_p,h_p)(x,s)\mathrm{d}s,\\
(v^1_p,g^1_p)&:=\chi(\sqrt{\epsilon}y)(v_p,g_p).
\end{aligned}
\right.
\end{equation}
It is direct to check that
$$\partial_xu^1_p+\partial_yv^1_p=\partial_x h^1_p+\partial_y g^1_p=0.$$
Based on the estimates for $u_p,h_p$ in Proposition \ref{1thBoundary}, we have
$$\left|\sqrt{\epsilon}\chi'(\sqrt{\epsilon}y)\int_0^y u_p(x,s)\mathrm{d}s\right|\leq \sqrt{\epsilon}y|\chi'|\Vert u_p\Vert_{L^\infty} \leq C(L,\zeta)\epsilon^{-\zeta}$$
and
$$\left|\sqrt{\epsilon}\chi'(\sqrt{\epsilon}y)\int_0^y h_p(x,s)\mathrm{d}s\right|\leq \sqrt{\epsilon}y|\chi'|\Vert h_p\Vert_{L^\infty} \leq C(L,\zeta)\epsilon^{-\zeta}.$$
Therefore,
\begin{equation}\label{4.048}
\begin{aligned}
\Vert (u^1_p,v^1_p,h^1_p,g^1_p)\Vert_{L^\infty}+&\sup\limits_{0\leq x \leq L_2}\Vert\langle y\rangle^l\partial_{yy} (v^1_p,g^1_p)\Vert_{L^2(0,\infty)}\\
&+\Vert \langle y\rangle^l\partial_{xy}(v^1_p,g^1_p)\Vert_{L^2}\leq C(L,\zeta)\epsilon^{-\zeta},\\
\sup\limits_{0\leq x \leq L_2}\Vert \langle y\rangle^l\partial_{xyy}(v^1_p,g^1_p)&\Vert_{L^2(0,\infty)}+\Vert \langle y\rangle^l\partial_{xxy}(v^1_p,g^1_p)\Vert_{L^2}\lesssim C(L)\epsilon^{-1}.
\end{aligned}
\end{equation}
Thanks to the compact support of the cut-off function $\chi(\cdot)$, it holds that
\begin{equation}\label{4.049a}
\begin{aligned}
\Vert \partial_x v^1_p\Vert_{L^2}^2 \leq &\iint_{\{\sqrt{\epsilon}y\leq 1\}} | \partial_x v^1_p|^2\mathrm{d}x\mathrm{d}y \\
&\lesssim\epsilon^{-1/2}\iint \langle y\rangle^l |\partial_{xy}v^1_p|^2\mathrm{d}x\mathrm{d}y \leq C(L,\zeta)\epsilon^{-1/2-2\zeta}
\end{aligned}
\end{equation}
due to the weighted estimates (\ref{4.048}). Similarly,
\begin{equation}\label{4.049}
\begin{aligned}
\Vert \partial_x (v^1_p,g^1_p)\Vert_{L^2}^2 \leq& C(L,\zeta)\epsilon^{-1/2-2\zeta},\\
\Vert \partial_y (v^1_p,g^1_p)\Vert_{L^2}^2 \leq &C(L,\zeta)\epsilon^{-1/2-2\zeta},\\
\Vert \partial_{xx} (v^1_p,g^1_p)\Vert_{L^2}^2 \leq& C(L)\epsilon^{-5/2}.
\end{aligned}
\end{equation}

Putting $(u^1_p,v^1_p,h^1_p,g^1_p)$ into (\ref{4.039}), or equivalently into (\ref{4.01}), it will produce other new error terms due to the cut-off function $\chi(\cdot)$. Denoted by $R^{u,1}_p$ and $R^{h,1}_p$
\begin{equation}\label{4.050}
\left \{
\begin{aligned}
R^{u,1}_p:=&\bigg(u^0\partial_x +\partial_xu^0+(v^0_p+v^1_e)\partial_y-\nu\partial_y^2\bigg)\left(\sqrt{\epsilon}\chi'(\sqrt{\epsilon }y)\int_0^y u_p\mathrm{d}s\right)\\
&-2\nu\sqrt{\epsilon}\chi'(\sqrt{\epsilon}y)\partial_y u_p+u_p[(v^0_p+v^1_e)\sqrt{\epsilon}\chi'(\sqrt{\epsilon}y)-\nu\epsilon\chi''\sqrt{\epsilon}y)]\\
&+(1-\chi(\sqrt{\epsilon}y))\big(u^0_{eY}(yu^0_{px}+v^0_p)+yv^1_{eY} u^0_{py}+u^1_e u^0_{px}+u^0_pu^1_{ex}\big)\\
&-\bigg(h^0\partial_x +\partial_xh^0+(g^0_p+g^1_e)\partial_y\bigg)\left(\sqrt{\epsilon}\chi'(\sqrt{\epsilon }y)\int_0^y h_p\mathrm{d}s\right)\\
&-h_p(g^0_p+g^1_e)\sqrt{\epsilon}\chi'(\sqrt{\epsilon}y)\\
&-(1-\chi(\sqrt{\epsilon}y))\big(h^0_{eY}(yh^0_{px}+g^0_p)+yg^1_{eY} h^0_{py}+h^1_e h^0_{px}+h^0_ph^1_{ex}\big),\\
\
R^{h,1}_p:=&\bigg(u^0\partial_x+(v^0_p+v^1_e)\partial_y-\kappa\partial_y^2\bigg)\left(\sqrt{\epsilon}\chi'(\sqrt{\epsilon}y)\int_0^y h_p\mathrm{d}s\right)\\
&+\partial_xh^0\left(\sqrt{\epsilon}\chi'(\sqrt{\epsilon}y)\int_0^y u_p\mathrm{d}s\right)+h_p[(v^0_p+v^1_e)\sqrt{\epsilon}\chi'-\kappa \epsilon\chi'']\\
&-2\kappa\sqrt{\epsilon}\chi'\partial_yh_p\\
&+(1-\chi)( h^0_{eY}v^0_p+yh^0_{px}u^0_{eY}+y v^1_{eY} h^0_{py}+u^0_p h^1_{ex}+u^1_e  h^0_{px})\\
&-\bigg(h^0\partial_x+(g^0_p+g^1_e)\partial_y\bigg)\left(\sqrt{\epsilon}\chi'(\sqrt{\epsilon}y)\int_0^y u_p\mathrm{d}s\right)\\
&-\partial_xu^0\left(\sqrt{\epsilon}\chi'(\sqrt{\epsilon}y)\int_0^y h_p\mathrm{d}s\right)-u_p(g^0_p+g^1_e)\sqrt{\epsilon}\chi'\\
&-(1-\chi)(u^0_{eY} g^0_p+y  h^0_{eY}  u^0_{px}+y  g^1_{eY}  u^0_{py}+h^1_e u^0_{px}+h^0_p u^1_{ex}),
\end{aligned}
\right.
\end{equation}
These terms will contribute to $R^{u,1},R^{h,1}$, and hence $R^1_{app},R^3_{app}$. Note that the $u^0_p,h^0_p$ are rapidly decaying at $\infty$ and then $u^0_x,h^0_x$ also decay rapidly at $\infty$. Therefore the integrals of $u^0_x\int_0^y u_p\mathrm{d}s, u^0_x\int_0^y h_p\mathrm{d}s, h^0_x\int_0^y u_p\mathrm{d}s$ and $h^0_x\int_0^y h_p\mathrm{d}s$ are uniformly bounded by $\epsilon^{-\zeta}$.

Applying the estimates for the ideal MHD and boundary layer profiles, we know that the $L^2$ norms for the terms involving $u_p,h_p$ in $R^{u,1}_p,R^{h,1}_p$ are bounded by
\begin{equation}\label{4.051}
\begin{aligned}
C\epsilon^{-\zeta}\sqrt{\epsilon}\Vert \chi(\sqrt{\epsilon}y)\Vert_{L^2} \leq C\epsilon^{1/4-\zeta}.
\end{aligned}
\end{equation}
Now we estimate the rest terms of the form $(1-\chi)(\cdots)$ in $R^{u,1}_p,R^{h,1}_p$. Recall that the boundary layers $u^0_p,v^0_p,h^0_p,g^0_p$ decay rapidly as $y\to \infty$, which are of order $\epsilon^l$ when $\sqrt{\epsilon}y\geq 1$, for large enough $l\geq 0$. Therefore we obtain
\begin{equation}\label{4.052}
\begin{aligned}
\Vert (R^{u,1}_p,R^{h,1}_p)\Vert_{L^2}\leq C(L,\zeta)\epsilon^{1/4-\zeta}.
\end{aligned}
\end{equation}

For $p^2_p$, it is defined as follows:
\begin{equation}\label{4.053}
\begin{aligned}
\partial_x p^2_p(x,y)=&\int_y^\infty \partial_x\bigg[(u^0_e+u^0_p)\partial_x v^0_p+u^0_p\partial_x v^1_e+(v^0_p+v^1_e)\partial_y v^0_p\\
&-(h^0_e+h^0_p)\partial_x g^0_p-h^0_p\partial_x g^1_e-(g^0_p+g^1_e)\partial_y g^0_p-\nu \partial_y^2 v^0_p\bigg](x,\theta)\mathrm{d}\theta\\
=&\int_y^\infty\bigg[(u^0_e+u^0_p)\partial_{xx}v^0_p+u^0_p\partial_{xx}v^1_e+(v^0_p+v^1_e)\partial_{xy}v^0_p-\nu\partial_{xyy}v^0_p\\
&-(h^0_e+h^0_p)\partial_{xx}g^0_p-h^0_p\partial_{xx}g^1_e-(g^0_p+g^1_e)\partial_{xy}g^0_p\bigg](x,\theta)\mathrm{d}\theta,
\end{aligned}
\end{equation}
where we used the divergence-free conditions for $(u^0_p,v^0_p)$ and $(h^0_p,g^0_p)$ in the second equality.

Each term in the right hand side of (\ref{4.053}) can be estimated as follows:
$$\int_y^\infty (u^0_e+u^0_p)\partial_{xx}v^0_p \leq C\langle y\rangle^{-(n-1)}\Vert u^0\Vert_{L^\infty}\Vert \langle y\rangle^n\partial_x^2v^0_p\Vert_{L^2(0,\infty)},$$
$$\int_y^\infty u^0_p\partial_x^2 v^1_e \leq C\langle y\rangle^{-(n-1)}\Vert \langle y\rangle^n u^0_p\Vert_{L^2(0,\infty)}\Vert \partial_x^2v^1_e(x,\sqrt{\epsilon}y)\Vert_{L^2(0,\infty)},$$
$$\int_y^\infty (v^0_p+v^1_e)\partial_{xy} v^0_p \leq C\langle y\rangle^{-(n-1)}\Vert v^0_p+v^1_e\Vert_{L^\infty}\Vert \langle y\rangle^n\partial_{xy}v^0_p \Vert_{L^2(0,\infty)},$$
and
$$\int_y^\infty \nu\partial_{xyy} v^0_p \leq C\langle y\rangle^{-(n-1)}\Vert \langle y\rangle^n\partial_{xyy}v^0_p \Vert_{L^2(0,\infty)},$$
with $n\geq 2$. The similar arguments can be applied on the estimates for the other terms involving $h^0_p,g^0_p$ in (\ref{4.053}). Recall that the ideal MHD flows are evaluated at $(x,Y)$, we have
\begin{equation}\label{4.054}
\begin{aligned}
\Vert \partial_x p^2_p\Vert_{L^2}\lesssim \epsilon^{-1/4}.
\end{aligned}
\end{equation}
Base on the above constructions and energy estimates, the main results of this section can be concluded into the following proposition.
\begin{proposition}\label{approximatesolution}
Under the assumptions in the Theorem \ref{mainresult}, it holds that
\begin{equation}\nonumber
\begin{array}{lll}
\Vert R^1_{app}\Vert_{L^2}+\Vert R^3_{app}\Vert_{L^2}+\sqrt{\epsilon}\left(\Vert R^2_{app}\Vert_{L^2}+\Vert R^4_{app}\Vert_{L^2}\right)\leq C(L,\zeta)\epsilon^{3/4-\zeta},
\end{array}
\end{equation}
for arbitrarily small $\zeta>0$.
\end{proposition}
\begin{proof}
Collecting all error terms from the constructions of the approximate ideal MHD flows and approximate boundary layer profiles in Subsections \ref{0th}--\ref{cutofflayer}, we get that
\begin{equation}\label{4.055}
\left \{
\begin{aligned}
R^{1}_{app}:=&E_1-\nu\epsilon\partial_{YY}u^0_e+\sqrt{\epsilon}(R^{u,1}+R^{u,1}_p)+\epsilon\big[(u^1_e+u^1_p)\partial_x\\
&+v^1_p\partial_y\big](u^1_e+u^1_p)-\epsilon\big[(h^1_e+h^1_p)\partial_x +g^1_p\partial_y\big](h^1_e+h^1_p)\\
&+\epsilon\partial_xp^2_p-\nu\epsilon\partial_x^2\big(u^0_p+\sqrt{\epsilon}(u^1_e+u^1_p)\big]+Er_1,\\
R^{3}_{app}:=&E_3-\kappa\epsilon\partial_{YY}h^0_e+\sqrt{\epsilon}(R^{h,1}+R^{h,1}_p)+\epsilon\big[(u^1_e+u^1_p)\partial_x\\
&+v^1_p\partial_y\big](h^1_e+h^1_p)-\big[(h^1_e+h^1_p)\partial_x+g^1_p\partial_y](h^1_e+h^1_p)\\
&-\kappa\epsilon \partial_x^2\big[h^0_p+\sqrt{\epsilon}(h^1_e+h^1_p)]+Er_2.
\end{aligned}
\right.
\end{equation}
Based on the estimates of $E_1,E_3,R^{u,1},R^{u,1}_p,R^{h,1},R^{h,1}_p,Er_1,Er_2$ above, we have
\begin{equation}\label{4.056}
\begin{aligned}
\Vert E_1-\nu\epsilon\partial_{YY}u^0_e+\sqrt{\epsilon}(R^{u,1}+R^{u,1}_p)+\epsilon\partial_xp^2_p+Er_1\Vert_{L^2}\leq &C(L,\zeta)\epsilon^{3/4-\zeta},\\
\Vert E_3-\kappa\epsilon\partial_{YY}h^0_e+\sqrt{\epsilon}(R^{h,1}+R^{h,1}_p)+Er_2\Vert_{L^2}\leq &C(L,\zeta)\epsilon^{3/4-\zeta}.\\
\end{aligned}
\end{equation}
From the estimates of the inner ideal MHD flows and boundary layer profiles, it follows that
\begin{equation}\label{4.057}
\begin{aligned}
\epsilon\Vert (u^1_e+u^1_p)\partial_x(u^1_e+u^1_p)\Vert_{L^2}\leq &(\Vert u^1_e\Vert_{L^\infty}+\Vert u^1_p\Vert_{L^\infty})\\
&\times(\Vert\partial_xu^1_e\Vert_{L^2}+\Vert \partial_x u^1_p\Vert_{L^2})\\
\leq &C(L,\zeta)\epsilon^{3/4-\zeta},\\
\epsilon\Vert v^1_p\partial_y(u^1_e+u^1_p)\Vert_{L^2} \leq &\epsilon \Vert v^1_p\Vert_{L^\infty}(\sqrt{\epsilon}\Vert \partial_Y u^1_e\Vert_{L^2}+\Vert\partial_y u^1_p\Vert_{L^2})\\
\leq &C(L,\zeta)\epsilon^{1-\zeta},\\
\epsilon\Vert (h^1_e+h^1_p)\partial_x(h^1_e+h^1_p)\Vert_{L^2}\leq &(\Vert h^1_e\Vert_{L^\infty}+\Vert h^1_p\Vert_{L^\infty})\\
&\times (\Vert\partial_xh^1_e\Vert_{L^2}+\Vert \partial_x h^1_p\Vert_{L^2})\\
\leq &C(L,\zeta)\epsilon^{3/4-\zeta},\\
\epsilon\Vert g^1_p\partial_y(h^1_e+h^1_p)\Vert_{L^2} \leq &\epsilon \Vert g^1_p\Vert_{L^\infty}(\sqrt{\epsilon}\Vert \partial_Y h^1_e\Vert_{L^2}+\Vert\partial_y h^1_p\Vert_{L^2})\\
\leq &C(L,\zeta)\epsilon^{1-\zeta},\\
\epsilon\Vert \nu\partial_x^2\big(u^0_p+\sqrt{\epsilon}(u^1_e+u^1_p)\big]\Vert_{L^2} \leq &\epsilon \Vert \partial_x^2 u^0_p\Vert_{L^2}\\
&+\epsilon^{3/2}(\Vert \partial_x^2 u^1_e\Vert_{L^2}+\Vert \partial_x^2 u^1_p\Vert_{L^2})\\
\leq &C(L)\epsilon,
\end{aligned}
\end{equation}
here $\zeta>0$ is arbitrarily small constant. Therefore, we deduce that
$$\Vert R^1_{app}\Vert_{L^2}\leq C(L,\zeta)\epsilon^{3/4-\zeta}.$$
Similar arguments yield that
$$\Vert R^3_{app}\Vert_{L^2}\leq C(L,\zeta)\epsilon^{3/4-\zeta}.$$

We turn to estimate $R^2_{app}$ and $R^4_{app}$. Collecting all error terms together give that
\begin{equation}\label{4.058}
\left \{
\begin{aligned}
R^{2}_{app}:=&R^{v,0}+\sqrt{\epsilon}\big[(u^0_e+u^0_p+\sqrt{\epsilon}(u^1_e+u^1_p))\partial_x+(v^0_p+v^1_e+\sqrt{\epsilon}v^1_p)\partial_y\big]v^1_p\\
&-\sqrt{\epsilon}\big[(h^0_e+h^0_p+\sqrt{\epsilon}(h^1_e+h^1_p))\partial_x+(g^0_p+g^1_e+\sqrt{\epsilon}g^1_p)\partial_y\big]g^1_p\\
&+\sqrt{\epsilon}\big[(u^1_e+u^1_p)\partial_x+v^1_p\partial_y\big](v^0_p+v^1_e)-\nu\sqrt{\epsilon}\partial_{yy}v^1_p\\
&-\sqrt{\epsilon}\big[(h^1_e+h^1_p)\partial_x+g^1_p\partial_y\big](g^0_p+g^1_e)-\nu\epsilon \partial_x^2(v^0_p+v^1_e+\sqrt{\epsilon}v^1_p),\\
R^{4}_{app}:=&R^{g,0}+\sqrt{\epsilon}\big[(u^0_e+u^0_p+\sqrt{\epsilon}(u^1_e+u^1_p))\partial_x+(v^0_p+v^1_e+\sqrt{\epsilon}v^1_p)\partial_y\big]g^1_p\\
&-\sqrt{\epsilon}\big[(h^0_e+h^0_p+\sqrt{\epsilon}(h^1_e+h^1_p))\partial_x+(g^0_p+g^1_e+\sqrt{\epsilon}g^1_p)\partial_y\big]v^1_p\\
&+\sqrt{\epsilon}\big[(u^1_e+u^1_p)\partial_x+v^1_p\partial_y\big](g^0_p+g^1_e)-\kappa\sqrt{\epsilon}\partial_{yy}g^1_p\\
&-\sqrt{\epsilon}\big[(h^1_e+h^1_p)\partial_x+g^1_p\partial_y\big](v^0_p+v^1_e)\\
&-\kappa\epsilon \partial_x^2(g^0_p+g^1_e+\sqrt{\epsilon}g^1_p)+E_4.
\end{aligned}
\right.
\end{equation}
Recall the estimates of $R^{v,0}, R^{g,0},E_4$ in (\ref{4.037})-(\ref{4.038}), one has
$$\Vert ( R^{v,0},R^{g,0}, E_4)\Vert_{L^2}\lesssim \epsilon^{1/4}.$$
And other terms can be estimated as follows:
\begin{equation}\label{4.059}
\begin{aligned}
\sqrt{\epsilon}\bigg\Vert&\big[(u^0_e+u^0_p+\sqrt{\epsilon}(u^1_e+u^1_p))\partial_x+(v^0_p+v^1_e+\sqrt{\epsilon}v^1_p)\partial_y\big]v^1_p\bigg\Vert_{L^2}\\
&\lesssim\sqrt{\epsilon}\Vert (u^0_e,u^0_p,v^0_p,u^1_e,v^1_e)\Vert_{L^\infty}\Vert \nabla v^1_p\Vert_{L^2}\\
&\leq C(L,\zeta)\epsilon^{1/4-\zeta},\\
\sqrt{\epsilon}\bigg\Vert&\big[(h^0_e+h^0_p+\sqrt{\epsilon}(h^1_e+h^1_p))\partial_x+(g^0_p+g^1_e+\sqrt{\epsilon}g^1_p)\partial_y\big]g^1_p\bigg\Vert_{L^2}\\
&\lesssim\sqrt{\epsilon}\Vert (h^0_e,h^0_p,g^0_p,h^1_e,g^1_e)\Vert_{L^\infty}\Vert \nabla g^1_p\Vert_{L^2}\\
&\leq C(L,\zeta)\epsilon^{1/4-\zeta},\\
\end{aligned}
\end{equation}
and
\begin{equation}\label{4.060}
\begin{aligned}
\sqrt{\epsilon}\bigg\Vert&\big[(u^1_e+u^1_p)\partial_x+v^1_p\partial_y\big](v^0_p+v^1_e)\bigg\Vert_{L^2}\\
&\lesssim\sqrt{\epsilon}\Vert (u^1_e,u^1_p,v^1_p)\Vert_{L^\infty}(\Vert \partial_x v^0_p+\partial_x v^1_e\Vert_{L^2}+\Vert \partial_y v^0_p+\sqrt{\epsilon}\partial_Y v^1_e\Vert_{L^2})\\
&\leq C(L,\zeta)\epsilon^{1/4-\zeta},\\
\sqrt{\epsilon}\bigg\Vert&\big[(h^1_e+h^1_p)\partial_x+g^1_p\partial_y\big](g^0_p+g^1_e)\bigg\Vert_{L^2}\\
&\lesssim\sqrt{\epsilon}\Vert (h^1_e,h^1_p,g^1_p)\Vert_{L^\infty}(\Vert \partial_x g^0_p+\partial_x g^1_e\Vert_{L^2}+\Vert \partial_y g^0_p+\sqrt{\epsilon}\partial_Y g^1_e\Vert_{L^2})\\
&\leq C(L,\zeta)\epsilon^{1/4-\zeta},
\end{aligned}
\end{equation}
where we have used the facts that
$$\Vert \partial_x (v^1_p,g^1_p)\Vert_{L^2} \leq C(L,\zeta)\epsilon^{-1/4-\zeta}$$
and
$$\Vert\partial_x(v^1_e,g^1_e)(x,\sqrt{\epsilon}\cdot)\Vert_{L^2}\lesssim \epsilon^{-1/4}.$$
It is direct to calculate that
$$\Vert \nu\sqrt{\epsilon}\partial_{yy}v^1_p+\nu\epsilon \partial_x^2(v^0_p+v^1_e+\sqrt{\epsilon}v^1_p)\Vert_{L^2} \lesssim C(L)\epsilon^{1/4-\zeta},$$
where the estimates for $v^1_p,v^0_p,v^1_e$ are also used. Therefore,
$$\Vert R^2_{app}\Vert_{L^2} \leq C(L,\zeta)\epsilon^{1/4-\zeta}.$$
Similar arguments imply that
$$\Vert R^4_{app}\Vert_{L^2} \leq C(L,\zeta)\epsilon^{1/4-\zeta}.$$
The proof is completed.
\end{proof}

\section{convergence estimates: proof of the main theorem}\label{sec3}
After the construction of approximate solutions in Section \ref{sec2}, we will turn to prove the main theorem in this section.
\subsection{Linear stability estimates}
First, we focus on the energy estimates for the linearized problem. To state the linearized problem, we denote
\begin{equation}\label{5.01}
\left \{
\begin{aligned}
u_s(x,y)=&u^0_e(\sqrt{\epsilon}y)+u^0_p(x,y)+\sqrt{\epsilon}u^1_e(x,\sqrt{\epsilon}y),\\
v_s(x,y)=&v^0_p(x,y)+v^1_e(x,\sqrt{\epsilon }y),\\
h_s(x,y)=&h^0_e(\sqrt{\epsilon}y)+h^0_p(x,y)+\sqrt{\epsilon}h^1_e(x,\sqrt{\epsilon}y),\\
g_s(x,y)=&g^0_p(x,y)+g^1_e(x,\sqrt{\epsilon }y),
\end{aligned}
\right.
\end{equation}
then we linearize the scaled viscous MHD equations (\ref{1.6}) around  the approximate solution $(u_s, v_s, h_s, g_s)$ to obtain that
\begin{equation}\label{5.02}
\left \{
\begin{aligned}
u_s\partial_x u^\epsilon&+u^\epsilon\partial_x u_s+v_s\partial_y u^\epsilon+v^\epsilon\partial_y u_s+\partial_x p^\epsilon-\nu \Delta_\epsilon u^\epsilon\\
& -(h_s\partial_x h^\epsilon+h^\epsilon\partial_x h_s+g_s\partial_y h^\epsilon+g^\epsilon\partial_y h_s)=f_1,\\
u_s\partial_x v^\epsilon&+u^\epsilon\partial_x v_s+v_s\partial_y v^\epsilon+v^\epsilon\partial_y v_s+\frac{\partial_y p^\epsilon}{\epsilon}-\nu \Delta_\epsilon v^\epsilon\\
&-(h_s\partial_x g^\epsilon+h^\epsilon\partial_x g_s+g_s\partial_y g^\epsilon+g^\epsilon\partial_y g_s)=f_2,\\
u_s\partial_x h^\epsilon&+u^\epsilon\partial_x h_s+v_s\partial_y h^\epsilon+v^\epsilon\partial_y h_s-\kappa \Delta_\epsilon h^\epsilon\\
& -(h_s\partial_x u^\epsilon+h^\epsilon\partial_x u_s+g_s\partial_y u^\epsilon+g^\epsilon\partial_y u_s)=f_3,\\
u_s\partial_x g^\epsilon&+u^\epsilon\partial_x g_s+v_s\partial_y g^\epsilon+v^\epsilon\partial_y g_s-\kappa \Delta_\epsilon g^\epsilon\\
& -(h_s\partial_x v^\epsilon+h^\epsilon\partial_x v_s+g_s\partial_y v^\epsilon+g^\epsilon\partial_y v_s)=f_4,\\
\partial_x u^\epsilon+&\partial_yv^\epsilon=\partial_xh^\epsilon+\partial_yg^\epsilon=0,\\
(u^\epsilon,v^\epsilon,&\partial_y h^\epsilon,g^\epsilon)|_{y=0}=(0,0,0,0),\ (u^\epsilon,v^\epsilon, h^\epsilon,g^\epsilon)|_{x=0}=(0,0,0,0),\\
p^\epsilon-2&\nu\epsilon \partial_x u^\epsilon=0, \ \partial_y u^\epsilon+\nu\epsilon \partial_x v^\epsilon=h^\epsilon=\partial_x g^\epsilon=0 \ \ \mathrm{on} \ \ \{x=L\}.
\end{aligned}
\right.
\end{equation}
According to the arguments stated in the previous sections, we have
$$\Vert y\partial_y(u_s,h_s)\Vert_{L^\infty}<C\sigma_0$$
for some suitably small constant $\sigma_0$ and for some small $0<L<<1$.

The linear stability of (\ref{5.02}) can be stated as the following Proposition \ref{linearstability}.
\begin{proposition}\label{linearstability}
For any given $f_i \in L^2\ (i=1,2,3,4)$, there exists $L>0$ such that the linearized problem (\ref{5.02}) has a unique solution $(u^\epsilon,v^\epsilon,h^\epsilon,g^\epsilon,p^\epsilon)$ on $[0,L]\times [0,\infty)$, satisfying the following estimate
\begin{equation}\label{5.02a}
\begin{aligned}
\Vert \nabla_\epsilon u^\epsilon\Vert_{L^2}&+\Vert \nabla_\epsilon v^\epsilon\Vert_{L^2}+\Vert \nabla_\epsilon h^\epsilon\Vert_{L^2}+\Vert \nabla_\epsilon g^\epsilon\Vert_{L^2}\\
&\lesssim \Vert f_1\Vert_{L^2}+\Vert f_3\Vert_{L^2}+\sqrt{\epsilon}(\Vert  f_2\Vert_{L^2}+\Vert f_4\Vert_{L^2}).
\end{aligned}
\end{equation}
\end{proposition}
The existence of the solution can be obtained via the standard fixed point arguments and the estimates of classical elliptic system, we refer to \cite{YGuo2} for details. Here we only focus on the derivation of uniform a priori energy estimates.
\begin{lemma}\label{standardenergy}
Let $(u^\epsilon,v^\epsilon,h^\epsilon,g^\epsilon)$ be the solutions to the problem (\ref{5.02}), and suppose that $\epsilon<<L$, then there holds that
\begin{equation}\label{5.03}
\begin{aligned}
\nu\Vert \nabla_\epsilon u^\epsilon\Vert_{L^2}^2+\kappa \Vert \nabla_\epsilon h^\epsilon\Vert_{L^2}^2+\int_{x=L}u_s( |u^\epsilon|^2+\epsilon|v^\epsilon|^2+\epsilon|g^\epsilon|^2)\\
\lesssim L(\Vert \nabla_\epsilon v^\epsilon\Vert_{L^2}^2+\Vert \nabla_\epsilon g^\epsilon\Vert_{L^2}^2)+\Vert (f_1,f_3)\Vert_{L^2}^2+\epsilon\Vert (f_2,f_4)\Vert_{L^2}^2.
\end{aligned}
\end{equation}
\end{lemma}
The proof of this lemma will be left to Appendix \ref{ap2}.

Before giving the following positivity estimates, by the definitions of $u_s, h_s$, using the facts that $u^0_e>>h^0_e$ and $|u^0_e+u^0_p|>>|h^0_e+h^0_p|$ uniform in $y$ and the estimates for $u^0_p,u^1_e,h^0_p,h^1_e$ and the smallness of $\epsilon$, we have
$$\left\Vert \frac{h_s}{u_s}\right\Vert_{L^\infty}<<1$$
uniform in $x\in [0,L]$ for $0<L<<1$.
\begin{lemma}[Positivity estimates]\label{positivity}
If  $\left\Vert \frac{h_s}{u_s}\right\Vert_{L^\infty}<<1$ and $\Vert y \partial_y(u_s,h_s)\Vert_{L^\infty}<C\sigma_0$ uniform in $0<L<<1$ for suitably small $\sigma_0$, then there holds that
\begin{equation}\label{5.026}
\begin{aligned}
&\Vert \nabla_\epsilon v^\epsilon\Vert_{L^2}^2+ \Vert \nabla_\epsilon g^\epsilon\Vert_{L^2}^2+\int_{x=0}\frac{\epsilon^2 (\nu|v^\epsilon_x|^2+\kappa|g^\epsilon_x|^2)}{u_s}+\epsilon\int_{x=L}\frac{|v^\epsilon_y|^2}{u_s}\\
\leq &C\bigg[\Vert (f_1,f_3)\Vert_{L^2}^2+\epsilon\Vert (f_2,f_4)\Vert_{L^2}^2+\Vert \nabla_\epsilon u^\epsilon\Vert_{L^2}^2+\Vert \nabla_\epsilon h^\epsilon\Vert_{L^2}^2\\
&+\left(L+\left\Vert \frac{h_s}{u_s}\right\Vert_{L^\infty}+\Vert y\partial_y (u_s,h_s)\Vert_{L^\infty}\right)(\Vert \nabla_\epsilon v^\epsilon\Vert_{L^2}^2+\Vert \nabla_\epsilon g^\epsilon\Vert_{L^2}^2)\bigg]
\end{aligned}
\end{equation}
for some constant $C$ which is independent of $\epsilon$ and $L$.
\end{lemma}
The proof of this lemma will be shown in Appendix \ref{ap3}.

In order to study the nonlinear problem, the $L^\infty$ estimates of the solutions to the nonlinear problem are needed.
\begin{lemma}\label{linfity}
Suppose that $u^\epsilon,v^\epsilon,h^\epsilon,g^\epsilon,p^\epsilon$ are the solutions to the system
\begin{equation}\label{5.027}
\left\{
\begin{aligned}
&-\nu\Delta_\epsilon u^\epsilon+\partial_x p^\epsilon=F_1,\\
&-\nu\Delta_\epsilon v^\epsilon+\frac{\partial_y p^\epsilon}{\epsilon}=F_2,\\
&-\kappa \Delta_\epsilon h^\epsilon=F_3,\\
&-\kappa\Delta_\epsilon g^\epsilon=F_4,\\
&\partial_xu^\epsilon+\partial_y v^\epsilon=\partial_x h^\epsilon+\partial_yg^\epsilon=0,
\end{aligned}
\right.
\end{equation}
with the following boundary conditions
\begin{equation}\label{5.028}
\left\{
\begin{aligned}
&(u^\epsilon,v^\epsilon,\partial_y h^\epsilon,g^\epsilon)|_{y=0}=(0,0,0,0),\\
&(u^\epsilon,v^\epsilon,h^\epsilon,g^\epsilon)|_{x=0}=(0,0,0,0),\\
&p^\epsilon-2\nu\epsilon \partial_x u^\epsilon=0, \ \partial_y u^\epsilon+\nu\epsilon \partial_x v^\epsilon=h^\epsilon=\partial_x g^\epsilon=0 \ \ \mathrm{on} \ \ \{x=L\}.
\end{aligned}
\right.
\end{equation}
Then it holds that
\begin{equation}\label{5.029}
\begin{aligned}
\epsilon^{\frac{\gamma}{4}}&(\Vert u^\epsilon\Vert_{L^\infty}+\Vert h^\epsilon\Vert_{L^\infty})+\epsilon^{\frac{\gamma}{4}+\frac{1}{2}}(\Vert v^\epsilon\Vert_{L^\infty}+\Vert g^\epsilon\Vert_{L^\infty}) \\
\leq & C_{\gamma,L} \bigg[\Vert\nabla_\epsilon u^\epsilon\Vert_{L^2}+\Vert\nabla_\epsilon h^\epsilon\Vert_{L^2}+\Vert\nabla_\epsilon v^\epsilon\Vert_{L^2}+\Vert \nabla_\epsilon g^\epsilon\Vert_{L^2}\\
&+(\Vert F_1\Vert_{L^2}+\Vert F_3\Vert_{L^2})+\sqrt{\epsilon}(\Vert F_2\Vert_{L^2}+\Vert F_4\Vert_{L^2})\bigg].
\end{aligned}
\end{equation}
\end{lemma}
\begin{proof}
Note that each of $u^\epsilon, v^\epsilon$ satisfies a Stokes equation, then the estimates for the $u^\epsilon, v^\epsilon$ can be obtained by modifying the proof in \cite{YGuo2}. For the estimates of $h^\epsilon,g^\epsilon$, they obey the Possion equations, therefore the desired estimates can be obtained by the standard arguments for the elliptic system \cite{GT}.
\end{proof}

\subsection{Proof of the Theorem \ref{mainresult}}\label{pfofmainthm}
With the above preparation, it is ready to prove the Theorem \ref{mainresult}.

\begin{proof}[Proof of the Theorem \ref{mainresult}]
Consider the nonlinear scaled viscous MHD system (\ref{1.6}), and suppose the solutions to the viscous MHD can be decomposed as
\begin{equation}\label{5.030}
\begin{aligned}
(U^\epsilon,V^\epsilon,H^\epsilon,G^\epsilon,P^\epsilon)=&
(u_{app},v_{app},h_{app},g_{app},p_{app})(x,y)\\
&+\epsilon^{\frac{1}{2}+\gamma}(u^\epsilon,v^\epsilon,h^\epsilon,g^\epsilon,p^\epsilon)(x,y),
\end{aligned}
\end{equation}
where the approximate solutions are constructed as in the previous section. Recall that
\begin{equation}\label{5.031}
\left \{
\begin{aligned}
u_s(x,y)=&u^0_e(\sqrt{\epsilon}y)+u^0_p(x,y)+\sqrt{\epsilon}u^1_e(x,\sqrt{\epsilon}y),\\
v_s(x,y)=&v^0_p(x,y)+v^1_e(x,\sqrt{\epsilon }y),\\
h_s(x,y)=&h^0_e(\sqrt{\epsilon}y)+h^0_p(x,y)+\sqrt{\epsilon}h^1_e(x,\sqrt{\epsilon}y),\\
g_s(x,y)=&g^0_p(x,y)+g^1_e(x,\sqrt{\epsilon }y),
\end{aligned}
\right.
\end{equation}
then the remainder terms of $(u^\epsilon,v^\epsilon,h^\epsilon,g^\epsilon,p^\epsilon)$ satisfy
\begin{equation}\label{5.032}
\left \{
\begin{aligned}
u_s\partial_x u^\epsilon&+u^\epsilon\partial_x u_s+v_s\partial_y u^\epsilon+v^\epsilon\partial_y u_s+\partial_x p^\epsilon-\nu \Delta_\epsilon u^\epsilon\\
& -(h_s\partial_x h^\epsilon+h^\epsilon\partial_x h_s+g_s\partial_y h^\epsilon+g^\epsilon\partial_y h_s)=R_1(u^\epsilon,v^\epsilon,h^\epsilon,g^\epsilon),\\
u_s\partial_x v^\epsilon&+u^\epsilon\partial_x v_s+v_s\partial_y v^\epsilon+v^\epsilon\partial_y v_s+\frac{\partial_y p^\epsilon}{\epsilon}-\nu \Delta_\epsilon v^\epsilon\\
&-(h_s\partial_x g^\epsilon+h^\epsilon\partial_x g_s+g_s\partial_y g^\epsilon+g^\epsilon\partial_y g_s)=R_2(u^\epsilon,v^\epsilon,h^\epsilon,g^\epsilon),\\
u_s\partial_x h^\epsilon&+u^\epsilon\partial_x h_s+v_s\partial_y h^\epsilon+v^\epsilon\partial_y h_s-\kappa \Delta_\epsilon h^\epsilon\\
& -(h_s\partial_x u^\epsilon+h^\epsilon\partial_x u_s+g_s\partial_y u^\epsilon+g^\epsilon\partial_y u_s)=R_3(u^\epsilon,v^\epsilon,h^\epsilon,g^\epsilon),\\
u_s\partial_x g^\epsilon&+u^\epsilon\partial_x g_s+v_s\partial_y g^\epsilon+v^\epsilon\partial_y g_s-\kappa \Delta_\epsilon g^\epsilon\\
& -(h_s\partial_x v^\epsilon+h^\epsilon\partial_x v_s+g_s\partial_y v^\epsilon+g^\epsilon\partial_y v_s)=R_4(u^\epsilon,v^\epsilon,h^\epsilon,g^\epsilon),\\
\partial_x u^\epsilon+&\partial_yv^\epsilon=\partial_xh^\epsilon+\partial_yg^\epsilon=0,
\end{aligned}
\right.
\end{equation}
in which the source terms $R_i\ (i=1,2,3,4)$ are given by
\begin{equation}\label{5.033}
\left \{
\begin{aligned}
R_1:&=\epsilon^{-\frac{1}{2}-\gamma}R^1_{app}-\sqrt{\epsilon}\big[(u^1_p+\epsilon^\gamma u^\epsilon)\partial_xu^\epsilon+u^\epsilon \partial_x u^1_p+(v^1_p+\epsilon^\gamma v^\epsilon)\partial_y u^\epsilon\\
&+v^\epsilon \partial_y u^1_p-(h^1_p+\epsilon^\gamma h^\epsilon)\partial_xh^\epsilon-h^\epsilon \partial_x h^1_p-(g^1_p+\epsilon^\gamma g^\epsilon)\partial_y h^\epsilon-g^\epsilon \partial_y h^1_p\big],\\
R_2:&=\epsilon^{-\frac{1}{2}-\gamma}R^2_{app}-\sqrt{\epsilon}\big[(u^1_p+\epsilon^\gamma u^\epsilon)\partial_xv^\epsilon+u^\epsilon \partial_x v^1_p+(v^1_p+\epsilon^\gamma v^\epsilon)\partial_y v^\epsilon\\
&+v^\epsilon \partial_y v^1_p-(h^1_p+\epsilon^\gamma h^\epsilon)\partial_xg^\epsilon-h^\epsilon \partial_x g^1_p-(g^1_p+\epsilon^\gamma g^\epsilon)\partial_y g^\epsilon-g^\epsilon \partial_y g^1_p\big],\\
R_3:&=\epsilon^{-\frac{1}{2}-\gamma}R^3_{app}-\sqrt{\epsilon}\big[(u^1_p+\epsilon^\gamma u^\epsilon)\partial_xh^\epsilon+u^\epsilon \partial_x h^1_p+(v^1_p+\epsilon^\gamma v^\epsilon)\partial_y h^\epsilon\\
&+v^\epsilon \partial_y h^1_p-(h^1_p+\epsilon^\gamma h^\epsilon)\partial_xu^\epsilon-h^\epsilon \partial_x u^1_p-(g^1_p+\epsilon^\gamma g^\epsilon)\partial_y u^\epsilon-g^\epsilon \partial_y u^1_p\big],\\
R_4:&=\epsilon^{-\frac{1}{2}-\gamma}R^4_{app}-\sqrt{\epsilon}\big[(u^1_p+\epsilon^\gamma u^\epsilon)\partial_xg^\epsilon+u^\epsilon \partial_x g^1_p+(v^1_p+\epsilon^\gamma v^\epsilon)\partial_y g^\epsilon\\
&+v^\epsilon \partial_y g^1_p-(h^1_p+\epsilon^\gamma h^\epsilon)\partial_xv^\epsilon-h^\epsilon \partial_x v^1_p-(g^1_p+\epsilon^\gamma g^\epsilon)\partial_y v^\epsilon-g^\epsilon \partial_y v^1_p\big],\\
\end{aligned}
\right.
\end{equation}
here $R^i_{app}\ (i=1,2,3,4)$ are the error terms caused by approximation solutions, which are estimated in Proposition \ref{approximatesolution}.

To perform the standard contraction mapping argument, we define the function space $\mathcal{S}$ with the norm
\begin{equation}\label{5.034}
\begin{aligned}
\Vert (u^\epsilon,v^\epsilon,h^\epsilon,&g^\epsilon)\Vert_{\mathcal{S}}\\
:=&\Vert \nabla_\epsilon u^\epsilon\Vert_{L^2}+\Vert \nabla_\epsilon v^\epsilon\Vert_{L^2}+\Vert \nabla_\epsilon h^\epsilon\Vert_{L^2}+\Vert \nabla_\epsilon g^\epsilon\Vert_{L^2}\\
&+\epsilon^{\frac{\gamma}{2}}\Vert u^\epsilon\Vert_{L^\infty}+\epsilon^{\frac{\gamma}{2}+\frac{1}{2}}\Vert v^\epsilon\Vert_{L^\infty}+\epsilon^{\frac{\gamma}{2}}\Vert h^\epsilon\Vert_{L^\infty}+\epsilon^{\frac{\gamma}{2}+\frac{1}{2}}\Vert g^\epsilon\Vert_{L^\infty}.
\end{aligned}
\end{equation}

For each $(\overline{u}^\epsilon,\overline{v}^\epsilon,\overline{h}^\epsilon,\overline{g}^\epsilon)\in \mathcal{S}$, we will solve the following linear problem for $(u^\epsilon,v^\epsilon,h^\epsilon,g^\epsilon)$:
\begin{equation}\label{5.035}
\left \{
\begin{aligned}
u_s\partial_x u^\epsilon&+u^\epsilon\partial_x u_s+v_s\partial_y u^\epsilon+v^\epsilon\partial_y u_s+\partial_x p^\epsilon-\nu \Delta_\epsilon u^\epsilon\\
& -(h_s\partial_x h^\epsilon+h^\epsilon\partial_x h_s+g_s\partial_y h^\epsilon+g^\epsilon\partial_y h_s)=R_1(\overline{u}^\epsilon,\overline{v}^\epsilon,\overline{h}^\epsilon,\overline{g}^\epsilon),\\
u_s\partial_x v^\epsilon&+u^\epsilon\partial_x v_s+v_s\partial_y v^\epsilon+v^\epsilon\partial_y v_s+\frac{\partial_y p^\epsilon}{\epsilon}-\nu \Delta_\epsilon v^\epsilon\\
&-(h_s\partial_x g^\epsilon+h^\epsilon\partial_x g_s+g_s\partial_y g^\epsilon+g^\epsilon\partial_y g_s)=R_2(\overline{u}^\epsilon,\overline{v}^\epsilon,\overline{h}^\epsilon,\overline{g}^\epsilon),\\
u_s\partial_x h^\epsilon&+u^\epsilon\partial_x h_s+v_s\partial_y h^\epsilon+v^\epsilon\partial_y h_s-\kappa \Delta_\epsilon h^\epsilon\\
& -(h_s\partial_x u^\epsilon+h^\epsilon\partial_x u_s+g_s\partial_y u^\epsilon+g^\epsilon\partial_y u_s)=R_3(\overline{u}^\epsilon,\overline{v}^\epsilon,\overline{h}^\epsilon,\overline{g}^\epsilon),\\
u_s\partial_x g^\epsilon&+u^\epsilon\partial_x g_s+v_s\partial_y g^\epsilon+v^\epsilon\partial_y g_s-\kappa \Delta_\epsilon g^\epsilon\\
& -(h_s\partial_x v^\epsilon+h^\epsilon\partial_x v_s+g_s\partial_y v^\epsilon+g^\epsilon\partial_y v_s)=R_4(\overline{u}^\epsilon,\overline{v}^\epsilon,\overline{h}^\epsilon,\overline{g}^\epsilon),\\
\partial_x u^\epsilon+&\partial_yv^\epsilon=\partial_xh^\epsilon+\partial_yg^\epsilon=0.
\end{aligned}
\right.
\end{equation}
Applying the Proposition \ref{linearstability} to the linear system (\ref{5.035}), one gets that
\begin{equation}\label{5.036}
\begin{aligned}
\Vert \nabla_\epsilon u^\epsilon\Vert_{L^2}&+\Vert \nabla_\epsilon v^\epsilon\Vert_{L^2}+\Vert \nabla_\epsilon h^\epsilon\Vert_{L^2}+\Vert \nabla_\epsilon g^\epsilon\Vert_{L^2}\\
\lesssim &\Vert R_1(\overline{u}^\epsilon,\overline{v}^\epsilon,\overline{h}^\epsilon,\overline{g}^\epsilon)\Vert_{L^2}+\Vert R_3(\overline{u}^\epsilon,\overline{v}^\epsilon,\overline{h}^\epsilon,\overline{g}^\epsilon)\Vert_{L^2}\\
&+\sqrt{\epsilon}(\Vert  R_2(\overline{u}^\epsilon,\overline{v}^\epsilon,\overline{h}^\epsilon,\overline{g}^\epsilon)\Vert_{L^2}+\Vert R_4(\overline{u}^\epsilon,\overline{v}^\epsilon,\overline{h}^\epsilon,\overline{g}^\epsilon)\Vert_{L^2}).
\end{aligned}
\end{equation}
It remains to estimate every remainder $R_i\ (i=1,2,3,4)$.

First, from the Proposition \ref{approximatesolution}, it is direct to derive that
\begin{equation}\nonumber
\begin{aligned}
\epsilon^{-\gamma-\frac{1}{2}}\bigg[\Vert R^1_{app}\Vert_{L^2}&+\Vert R^3_{app}\Vert_{L^2}+\sqrt{\epsilon}(\Vert R^2_{app}\Vert_{L^2}+\Vert R^4_{app}\Vert_{L^2})\bigg]
\leq  C(L)\epsilon^{-\gamma-\zeta+\frac{1}{4}},
\end{aligned}
\end{equation}
where $\zeta>0$ is an arbitrary small constant. In what follows, we take any $\gamma<\frac14$ and $\zeta$ such that $\gamma+\zeta<\frac{1}{4}$.

For $R_1$, we have
\begin{align*}
&\sqrt{\epsilon}\Vert (u^1_p+\epsilon^\gamma \overline{u}^\epsilon)\partial_x\overline{u}^\epsilon\Vert_{L^2}\\
\leq &\sqrt{\epsilon}(\Vert u^1_p\Vert_{L^\infty}+\epsilon^\gamma\Vert \overline{u}^\epsilon\Vert_{L^\infty})\Vert \partial_x\overline{u}^\epsilon \Vert_{L^2}\\
\leq & \epsilon^{\frac{1}{2}-\zeta}C(u^1_p)+\epsilon^{\frac{\gamma}{2}}\Vert (\overline{u}^\epsilon,\overline{v}^\epsilon,\overline{h}^\epsilon,\overline{g}^\epsilon)\Vert_{\mathcal{S}}^2\\
\end{align*}
and
\begin{align*}
&\sqrt{\epsilon}\Vert (v^1_p+\epsilon^\gamma \overline{v}^\epsilon)\partial_y\overline{u}^\epsilon\Vert_{L^2}\\
\leq &\sqrt{\epsilon}(\Vert v^1_p\Vert_{L^\infty}+\epsilon^\gamma\Vert \overline{v}^\epsilon\Vert_{L^\infty})\Vert \partial_y\overline{u}^\epsilon \Vert_{L^2}\\
\leq & \epsilon^{\frac{1}{2}-\zeta}C(v^1_p)+\epsilon^{\frac{\gamma}{2}}\Vert (\overline{u}^\epsilon,\overline{v}^\epsilon,\overline{h}^\epsilon,\overline{g}^\epsilon)\Vert_{\mathcal{S}}^2.\\
\end{align*}
Similarly,
\begin{align*}
&\Vert(h^1_p+\epsilon^\gamma \overline{h}^\epsilon)\partial_x\overline{h}^\epsilon\Vert_{L^2}\\
\leq &\sqrt{\epsilon}(\Vert h^1_p\Vert_{L^\infty}+\epsilon^\gamma\Vert \overline{h}^\epsilon\Vert_{L^\infty})\Vert \partial_x\overline{h}^\epsilon \Vert_{L^2}\\
\leq & \left(\epsilon^{\frac{1}{2}-\zeta}C(h^1_p)+\epsilon^{\frac{\gamma}{2}}\Vert (\overline{u}^\epsilon,\overline{v}^\epsilon,\overline{h}^\epsilon,\overline{g}^\epsilon)\Vert_{\mathcal{S}}\right)\Vert (\overline{u}^\epsilon,\overline{v}^\epsilon,\overline{h}^\epsilon,\overline{g}^\epsilon)\Vert_{\mathcal{S}}
\end{align*}
and
\begin{align*}
&\sqrt{\epsilon}\Vert (g^1_p+\epsilon^\gamma \overline{g}^\epsilon)\partial_y\overline{h}^\epsilon\Vert_{L^2}\\
\leq &\sqrt{\epsilon}(\Vert g^1_p\Vert_{L^\infty}+\epsilon^\gamma\Vert \overline{g}^\epsilon\Vert_{L^\infty})\Vert \partial_y\overline{h}^\epsilon \Vert_{L^2}\\
\leq & \left(\epsilon^{\frac{1}{2}-\zeta}C(g^1_p)+\epsilon^{\frac{\gamma}{2}}\Vert (\overline{u}^\epsilon,\overline{v}^\epsilon,\overline{h}^\epsilon,\overline{g}^\epsilon)\Vert_{\mathcal{S}}\right)\Vert (\overline{u}^\epsilon,\overline{v}^\epsilon,\overline{h}^\epsilon,\overline{g}^\epsilon)\Vert_{\mathcal{S}}.
\end{align*}
Also, notice that $|(\overline{u}^\epsilon,\overline{v}^\epsilon,\overline{h}^\epsilon,\overline{g}^\epsilon)|\leq \sqrt{y}\Vert \partial_y(\overline{u}^\epsilon,\overline{v}^\epsilon,\overline{h}^\epsilon,\overline{g}^\epsilon)\Vert_{L^2(0,\infty)}$ and the weighted $H^1$ bounds on $u^1_p$ and $h^1_p$ in (\ref{4.048})-(\ref{4.049}),
 we deduce that
\begin{equation}\nonumber
\begin{aligned}
&\sqrt{\epsilon}\Vert \overline{u}^\epsilon\partial_x u^1_p+\overline{v}^\epsilon\partial_y u^1_p\Vert_{L^2}\\
&\leq \sqrt{\epsilon}\sup_x \left(\Vert \langle y\rangle^n \partial_x u^1_p\Vert_{L^2(0,\infty)}+\Vert \langle y\rangle^n \partial_y u^1_p\Vert_{L^2(0,\infty)}\right)\Vert \partial_y (\overline{u}^\epsilon,\overline{v}^\epsilon)\Vert_{L^2}\\
&\leq C(u^1_p)\epsilon^{\frac{1}{2}-\zeta}\Vert (\overline{u}^\epsilon,\overline{v}^\epsilon,\overline{h}^\epsilon,\overline{g}^\epsilon)\Vert_{\mathcal{S}}\\
\end{aligned}
\end{equation}
and
\begin{equation}\nonumber
\begin{aligned}
&\sqrt{\epsilon}\Vert \overline{h}^\epsilon\partial_x h^1_p+\overline{g}^\epsilon\partial_y h^1_p\Vert_{L^2}\\
&\leq \sqrt{\epsilon}\sup_x \left(\Vert \langle y\rangle^n \partial_x h^1_p\Vert_{L^2(0,\infty)}+\Vert \langle y\rangle^n \partial_y h^1_p\Vert_{L^2(0,\infty)}\right)\Vert \partial_y (\overline{h}^\epsilon,\overline{g}^\epsilon)\Vert_{L^2}\\
&\leq C(h^1_p)\epsilon^{\frac{1}{2}-\zeta}\Vert (\overline{h}^\epsilon,\overline{v}^\epsilon,\overline{h}^\epsilon,\overline{g}^\epsilon)\Vert_{\mathcal{S}}.
\end{aligned}
\end{equation}
For $R_2$, we get that
\begin{align*}
&\sqrt{\epsilon}\Vert (u^1_p+\epsilon^\gamma \overline{u}^\epsilon)\partial_x\overline{v}^\epsilon\Vert_{L^2}\\
\leq &\sqrt{\epsilon}(\Vert u^1_p\Vert_{L^\infty}+\epsilon^\gamma\Vert \overline{u}^\epsilon\Vert_{L^\infty})\Vert \partial_x\overline{v}^\epsilon \Vert_{L^2}\\
\leq & \left(\epsilon^{\frac{1}{2}-\zeta}C(u^1_p)+\epsilon^{\frac{\gamma}{2}}\Vert (\overline{u}^\epsilon,\overline{v}^\epsilon,\overline{h}^\epsilon,\overline{g}^\epsilon)\Vert_{\mathcal{S}}\right)\Vert (\overline{u}^\epsilon,\overline{v}^\epsilon,\overline{h}^\epsilon,\overline{g}^\epsilon)\Vert_{\mathcal{S}},
\end{align*}
and
\begin{align*}
&\sqrt{\epsilon}\Vert (v^1_p+\epsilon^\gamma \overline{v}^\epsilon)\partial_y\overline{v}^\epsilon\Vert_{L^2}\\
\leq &\sqrt{\epsilon}(\Vert v^1_p\Vert_{L^\infty}+\epsilon^\gamma\Vert \overline{v}^\epsilon\Vert_{L^\infty})\Vert \partial_y\overline{v}^\epsilon \Vert_{L^2}\\
\leq & \left(\epsilon^{\frac{1}{2}-\zeta}C(v^1_p)+\epsilon^{\frac{\gamma}{2}}\Vert (\overline{u}^\epsilon,\overline{v}^\epsilon,\overline{h}^\epsilon,\overline{g}^\epsilon)\Vert_{\mathcal{S}}\right)\Vert (\overline{u}^\epsilon,\overline{v}^\epsilon,\overline{h}^\epsilon,\overline{g}^\epsilon)\Vert_{\mathcal{S}}.
\end{align*}
Similarly,
\begin{align*}
&\Vert(h^1_p+\epsilon^\gamma \overline{h}^\epsilon)\partial_x\overline{g}^\epsilon\Vert_{L^2}\\
\leq &\sqrt{\epsilon}(\Vert h^1_p\Vert_{L^\infty}+\epsilon^\gamma\Vert \overline{h}^\epsilon\Vert_{L^\infty})\Vert \partial_x\overline{g}^\epsilon \Vert_{L^2}\\
\leq & \left(\epsilon^{\frac{1}{2}-\zeta}C(h^1_p)+\epsilon^{\frac{\gamma}{2}}\Vert (\overline{u}^\epsilon,\overline{v}^\epsilon,\overline{h}^\epsilon,\overline{g}^\epsilon)\Vert_{\mathcal{S}}\right)\Vert (\overline{u}^\epsilon,\overline{v}^\epsilon,\overline{h}^\epsilon,\overline{g}^\epsilon)\Vert_{\mathcal{S}},
\end{align*}
and
\begin{equation}\nonumber
\begin{aligned}
&\sqrt{\epsilon}\Vert (g^1_p+\epsilon^\gamma \overline{g}^\epsilon)\partial_y\overline{g}^\epsilon\Vert_{L^2}\\
\leq &\sqrt{\epsilon}(\Vert g^1_p\Vert_{L^\infty}+\epsilon^\gamma\Vert \overline{g}^\epsilon\Vert_{L^\infty})\Vert \partial_y\overline{g}^\epsilon \Vert_{L^2}\\
\leq & \left(\epsilon^{\frac{1}{2}-\zeta}C(g^1_p)+\epsilon^{\frac{\gamma}{2}}\Vert (\overline{u}^\epsilon,\overline{v}^\epsilon,\overline{h}^\epsilon,\overline{g}^\epsilon)\Vert_{\mathcal{S}}\right)\Vert (\overline{u}^\epsilon,\overline{v}^\epsilon,\overline{h}^\epsilon,\overline{g}^\epsilon)\Vert_{\mathcal{S}}.
\end{aligned}
\end{equation}
For the other terms in $R_2$, we have
\begin{equation}\nonumber
\begin{aligned}
&\sqrt{\epsilon}\Vert \overline{u}^\epsilon\partial_x v^1_p+\overline{v}^\epsilon\partial_y v^1_p\Vert_{L^2}\\
&\leq \sqrt{\epsilon}\bigg(\Vert \overline{u}^\epsilon\Vert_{L^\infty}\Vert \partial_x v^1_p\Vert_{L^2}+\sup_x\Vert \langle y\rangle^n\partial_y v^1_p\Vert_{L^2(0,\infty)}\Vert\partial_y \overline{v}^\epsilon\Vert_{L^2}\bigg)\\
&\leq C(v^1_p)\epsilon^{\frac{1}{4}-\frac{\gamma}{2}-\zeta}\Vert (\overline{u}^\epsilon,\overline{v}^\epsilon,\overline{h}^\epsilon,\overline{g}^\epsilon)\Vert_{\mathcal{S}}
\end{aligned}
\end{equation}
and
\begin{equation}\nonumber
\begin{aligned}
&\sqrt{\epsilon}\Vert \overline{h}^\epsilon\partial_x g^1_p+\overline{g}^\epsilon\partial_y g^1_p\Vert_{L^2}\\
&\leq \sqrt{\epsilon}\bigg(\Vert \overline{h}^\epsilon\Vert_{L^\infty}\Vert \partial_x g^1_p\Vert_{L^2}+\sup_x\Vert \langle y\rangle^n\partial_y g^1_p\Vert_{L^2(0,\infty)}\Vert\partial_y \overline{g}^\epsilon\Vert_{L^2}\bigg)\\
&\leq C(v^1_p)\epsilon^{\frac{1}{4}-\frac{\gamma}{2}-\zeta}\Vert (\overline{u}^\epsilon,\overline{v}^\epsilon,\overline{h}^\epsilon,\overline{g}^\epsilon)\Vert_{\mathcal{S}}.
\end{aligned}
\end{equation}
Therefore, we conclude that
\begin{equation}\label{5.037}
\begin{aligned}
&\Vert R_1(\overline{u}^\epsilon,\overline{v}^\epsilon,\overline{h}^\epsilon,\overline{g}^\epsilon)\Vert_{L^2}+\sqrt{\epsilon}\Vert R_2(\overline{u}^\epsilon,\overline{v}^\epsilon,\overline{h}^\epsilon,\overline{g}^\epsilon)\Vert_{L^2}\\
\leq& C(u_s,v_s,h_s,g_s)\epsilon^{-\gamma-\zeta+\frac{1}{4}}\\
&+C(u^1_p,v^1_p,h^1_p,g^1_p)\epsilon^{\frac{1}{2}+\frac{\gamma}{2}}\bigg(\Vert (\overline{u}^\epsilon,\overline{v}^\epsilon,\overline{h}^\epsilon,\overline{g}^\epsilon)\Vert_{\mathcal{S}}+\Vert (\overline{u}^\epsilon,\overline{v}^\epsilon,\overline{h}^\epsilon,\overline{g}^\epsilon)\Vert_{\mathcal{S}}^2\bigg).
\end{aligned}
\end{equation}
Similar arguments yield that
\begin{equation}\label{5.038}
\begin{aligned}
&\Vert R_3(\overline{u}^\epsilon,\overline{v}^\epsilon,\overline{h}^\epsilon,\overline{g}^\epsilon)\Vert_{L^2}+\sqrt{\epsilon}\Vert R_4(\overline{u}^\epsilon,\overline{v}^\epsilon,\overline{h}^\epsilon,\overline{g}^\epsilon)\Vert_{L^2}\\
\leq& C(u_s,v_s,h_s,g_s)\epsilon^{-\gamma-\zeta+\frac{1}{4}}\\
&+C(u^1_p,v^1_p,h^1_p,g^1_p)\epsilon^{\frac{1}{2}+\frac{\gamma}{2}}\bigg(\Vert (\overline{u}^\epsilon,\overline{v}^\epsilon,\overline{h}^\epsilon,\overline{g}^\epsilon)\Vert_{\mathcal{S}}+\Vert (\overline{u}^\epsilon,\overline{v}^\epsilon,\overline{h}^\epsilon,\overline{g}^\epsilon)\Vert_{\mathcal{S}}^2\bigg).
\end{aligned}
\end{equation}
It remains to estimate the $L^\infty$ norm. By the Lemma \ref{linfity}, we take
\begin{equation}\nonumber
\left \{
\begin{aligned}
F_1:=&-(u_s\partial_x u^\epsilon+u^\epsilon\partial_x u_s+v_s\partial_y u^\epsilon+v^\epsilon\partial_y u_s)\\
& +(h_s\partial_x h^\epsilon+h^\epsilon\partial_x h_s+g_s\partial_y h^\epsilon+g^\epsilon\partial_y h_s)+R_1,\\
F_2:=&-(u_s\partial_x v^\epsilon+u^\epsilon\partial_x v_s+v_s\partial_y v^\epsilon+v^\epsilon\partial_y v_s)\\
&+(h_s\partial_x g^\epsilon+h^\epsilon\partial_x g_s+g_s\partial_y g^\epsilon+g^\epsilon\partial_y g_s)+R_2,\\
F_3:=&-(u_s\partial_x h^\epsilon+u^\epsilon\partial_x h_s+v_s\partial_y h^\epsilon+v^\epsilon\partial_y h_s)\\
&+(h_s\partial_x u^\epsilon+h^\epsilon\partial_x u_s+g_s\partial_y u^\epsilon+g^\epsilon\partial_y u_s)+R_3,\\
F_4:=&-(u_s\partial_x g^\epsilon+u^\epsilon\partial_x g_s+v_s\partial_y g^\epsilon+v^\epsilon\partial_y g_s)\\
& +(h_s\partial_x v^\epsilon+h^\epsilon\partial_x v_s+g_s\partial_y v^\epsilon+g^\epsilon\partial_y v_s)+R_4,
\end{aligned}
\right.
\end{equation}
then we have
\begin{equation}\label{5.038}
\begin{aligned}
\epsilon^{\frac{\gamma}{2}}&(\Vert u^\epsilon\Vert_{L^\infty}+\Vert h^\epsilon\Vert_{L^\infty})+\epsilon^{\frac{\gamma}{2}+\frac{1}{2}}(\Vert v^\epsilon\Vert_{L^\infty}+\Vert g^\epsilon\Vert_{L^\infty})\\
\lesssim &\epsilon^{\frac{\gamma}{4}}\bigg(\Vert \nabla_\epsilon u^\epsilon\Vert_{L^2}+\Vert \nabla_\epsilon v^\epsilon\Vert_{L^2}+\Vert \nabla_\epsilon h^\epsilon\Vert_{L^2}+\Vert \nabla_\epsilon g^\epsilon\Vert_{L^2}\bigg)\\
&+\epsilon^{\frac{\gamma}{4}}\bigg(\Vert F_1\Vert_{L^2}+\Vert F_3\Vert_{L^2}\bigg)+\epsilon^{\frac{\gamma}{4}+\frac{1}{2}}\bigg(\Vert F_2\Vert_{L^2}+\Vert F_4\Vert_{L^2}\bigg).\\
\end{aligned}
\end{equation}
Note that the estimates for $\nabla_\epsilon (u^\epsilon,v^\epsilon,h^\epsilon,g^\epsilon)$ and $R_i\ (i=1,2,3,4)$ have been done, and it remains to estimate the other terms involving $u_s,v_s,h_s,g_s$ in $F_i\ (i=1,2,3,4)$. For these terms, we have
\begin{equation}\nonumber
\begin{aligned}
\epsilon^{\frac{\gamma}{4}}\Vert u_s\partial_x u^\epsilon+v_s \partial_y u^\epsilon\Vert_{L^2} \leq& \epsilon^{\frac{\gamma}{4}-\zeta}C(L,u_s,v_s)(\Vert \partial_x u^\epsilon \Vert_{L^2}+\Vert \partial_y u^\epsilon \Vert_{L^2})\\
\leq &\epsilon^{\frac{\gamma}{4}-\zeta}C(L,u_s,v_s)\Vert (u^\epsilon,v^\epsilon,h^\epsilon,g^\epsilon)\Vert_{\mathcal{S}},\\
\epsilon^{\frac{\gamma}{4}}\Vert u^\epsilon \partial_x u_s +v^\epsilon\partial_y u_s  \Vert_{L^2} \leq& \epsilon^{\frac{\gamma}{4}}\bigg(\Vert \partial_y u^\epsilon\Vert_{L^2}\sup_x \Vert \sqrt{y}\partial_x u_s\Vert_{L^2(0,\infty)}\\
&+\Vert \partial_y v^\epsilon\Vert_{L^2}\sup_x\Vert \sqrt{y}\partial_y u_s\Vert_{L^2(0,\infty)}\bigg)\\
\leq & C\epsilon^{\frac{\gamma}{4}-\zeta}\Vert (u^\epsilon,v^\epsilon,h^\epsilon,g^\epsilon)\Vert_{\mathcal{S}},\\
\epsilon^{\frac{\gamma}{4}}\Vert h_s\partial_x h^\epsilon+g_s \partial_y h^\epsilon\Vert_{L^2} \leq& \epsilon^{\frac{\gamma}{4}-\zeta}C(L,h_s,g_s)(\Vert \partial_x h^\epsilon \Vert_{L^2}+\Vert \partial_y h^\epsilon \Vert_{L^2})\\
\leq &\epsilon^{\frac{\gamma}{4}-\zeta}C(L,h_s,g_s)\Vert (u^\epsilon,v^\epsilon,h^\epsilon,g^\epsilon)\Vert_{\mathcal{S}},
\end{aligned}
\end{equation}
and
\begin{equation}\nonumber
\begin{aligned}
\epsilon^{\frac{\gamma}{4}}\Vert h^\epsilon \partial_x h_s +g^\epsilon\partial_y h_s  \Vert_{L^2} \leq& \epsilon^{\frac{\gamma}{4}}\bigg(\Vert \partial_y h^\epsilon\Vert_{L^2}\sup_x \Vert \sqrt{y}\partial_x h_s\Vert_{L^2(0,\infty)}\\
&+\Vert \partial_y g^\epsilon\Vert_{L^2}\sup_x\Vert \sqrt{y}\partial_y h_s\Vert_{L^2(0,\infty)}\bigg)\\
\leq & C\epsilon^{\frac{\gamma}{4}-\zeta}\Vert (u^\epsilon,v^\epsilon,h^\epsilon,g^\epsilon)\Vert_{\mathcal{S}}.
\end{aligned}
\end{equation}

Similarly, one has
\begin{equation}\nonumber
\begin{aligned}
\epsilon^{\frac{\gamma}{4}+\frac{1}{2}}\Vert u_s\partial_x v^\epsilon+v_s \partial_y v^\epsilon\Vert_{L^2} \leq& \epsilon^{\frac{\gamma}{4}+\frac{1}{2}-\zeta}C(L,u_s,v_s)(\Vert \partial_x v^\epsilon \Vert_{L^2}+\Vert \partial_y v^\epsilon \Vert_{L^2})\\
\leq &\epsilon^{\frac{\gamma}{4}}C(L,u_s,v_s)\Vert (u^\epsilon,v^\epsilon,h^\epsilon,g^\epsilon)\Vert_{\mathcal{S}},\\
\epsilon^{\frac{\gamma}{4}+\frac{1}{2}}\Vert u^\epsilon \partial_x v_s +v^\epsilon\partial_y v_s  \Vert_{L^2} \leq& \epsilon^{\frac{\gamma}{4}+\frac{1}{2}}\bigg(\Vert \partial_y u^\epsilon\Vert_{L^2}\sup_x \Vert \sqrt{y}\partial_x v_s\Vert_{L^2(0,\infty)}\\
&+\Vert \partial_y v^\epsilon\Vert_{L^2}\sup_x\Vert \sqrt{y}\partial_y v_s\Vert_{L^2(0,\infty)}\bigg)\\
\leq & C\epsilon^{\frac{\gamma}{4}-\zeta}\Vert (u^\epsilon,v^\epsilon,h^\epsilon,g^\epsilon)\Vert_{\mathcal{S}},\\
\epsilon^{\frac{\gamma}{4}+\frac{1}{2}}\Vert h_s\partial_x g^\epsilon+g_s \partial_y g^\epsilon\Vert_{L^2} \leq& \epsilon^{\frac{\gamma}{4}+\frac{1}{2}}C(L,h_s,g_s)(\Vert \partial_x g^\epsilon \Vert_{L^2}+\Vert \partial_y g^\epsilon \Vert_{L^2})\\
\leq &\epsilon^{\frac{\gamma}{4}-\zeta}C(L,h_s,g_s)\Vert (u^\epsilon,v^\epsilon,h^\epsilon,g^\epsilon)\Vert_{\mathcal{S}},
\end{aligned}
\end{equation}
and
\begin{equation}\nonumber
\begin{aligned}
\epsilon^{\frac{\gamma}{4}+\frac{1}{2}}\Vert h^\epsilon \partial_x g_s +g^\epsilon\partial_y g_s  \Vert_{L^2} \leq& \epsilon^{\frac{\gamma}{4}+\frac{1}{2}}\bigg(\Vert \partial_y h^\epsilon\Vert_{L^2}\sup_x \Vert \sqrt{y}\partial_x g_s\Vert_{L^2(0,\infty)}\\
&+\Vert \partial_y g^\epsilon\Vert_{L^2}\sup_x\Vert \sqrt{y}\partial_y g_s\Vert_{L^2(0,\infty)}\bigg)\\
\leq & C\epsilon^{\frac{\gamma}{4}-\zeta}\Vert (u^\epsilon,v^\epsilon,h^\epsilon,g^\epsilon)\Vert_{\mathcal{S}}.
\end{aligned}
\end{equation}
Moreover, similar arguments can be applied to estimate the terms in $F_3,F_4$. Consequently, we conclude that
\begin{equation}\label{5.038}
\begin{aligned}
\epsilon^{\frac{\gamma}{2}}&(\Vert u^\epsilon\Vert_{L^\infty}+\Vert h^\epsilon\Vert_{L^\infty})+\epsilon^{\frac{\gamma}{2}+\frac{1}{2}}(\Vert v^\epsilon\Vert_{L^\infty}+\Vert g^\epsilon\Vert_{L^\infty})\\
\leq &C(u_s,v_s,h_s,g_s)\epsilon^{-\frac{3\gamma}{4}+\frac{1}{4}-\zeta}+C(L,u_s,v_s,h_s,g_s)\epsilon^{\frac{\gamma}{4}-\zeta}\Vert (u^\epsilon,v^\epsilon,h^\epsilon,g^\epsilon)\Vert_{\mathcal{S}}\\
&+C(u_s,v_s,h_s,g_s)\epsilon^{-\frac{\gamma}{4}+\frac{1}{2}-\zeta}\left(\Vert (\overline{u}^\epsilon,\overline{v}^\epsilon,\overline{h}^\epsilon,\overline{g}^\epsilon)\Vert_{\mathcal{S}}+\Vert (\overline{u}^\epsilon,\overline{v}^\epsilon,\overline{h}^\epsilon,\overline{g}^\epsilon)\Vert_{\mathcal{S}}^2\right).
\end{aligned}
\end{equation}
Taking $\gamma +\zeta \leq \frac{1}{4}$ and $\zeta <\frac{\gamma}{4}$, we have
\begin{equation}\label{5.039}
\begin{aligned}
\Vert (u^\epsilon,v^\epsilon,h^\epsilon,g^\epsilon)\Vert_{\mathcal{S}} \lesssim 1+\epsilon^{\frac{\gamma}{4}-\zeta}\Vert (\overline{u}^\epsilon,\overline{v}^\epsilon,\overline{h}^\epsilon,\overline{g}^\epsilon)\Vert_{\mathcal{S}}+\epsilon^{\frac{3}{8}}\Vert (\overline{u}^\epsilon,\overline{v}^\epsilon,\overline{h}^\epsilon,\overline{g}^\epsilon)\Vert_{\mathcal{S}}^2.
\end{aligned}
\end{equation}
This estimate shows that the operator $(\overline{u}^\epsilon,\overline{v}^\epsilon,\overline{h}^\epsilon,\overline{g}^\epsilon) \mapsto (u^\epsilon,v^\epsilon,h^\epsilon,g^\epsilon)$ via solving the problem (\ref{5.035}) maps the closed ball $B=\{\Vert (u^\epsilon,v^\epsilon,h^\epsilon,g^\epsilon)\Vert_{\mathcal{S}}\leq 4C(u_s,v_s,h_s,g_s)=:K\}$ in $\mathcal{S}$ into itself for small enough $\epsilon$. Furthermore, we have
\begin{equation}\label{5.040}
\begin{aligned}
\Vert &( u^\epsilon_1-u^\epsilon_2,v^\epsilon_1-v^\epsilon_2,h^\epsilon_1-h^\epsilon_2,g^\epsilon_1-g^\epsilon_2)\Vert_{\mathcal{S}} \\
&\leq 2KC(L,u_s,v_s,h_s,g_s)\epsilon^{\frac{\gamma}{4}-\zeta}\Vert( \overline{u}^\epsilon_1-\overline{u}^\epsilon_2,\overline{v}^\epsilon_1-\overline{v}^\epsilon_2,\overline{h}^\epsilon_1-\overline{h}^\epsilon_2
,\overline{g}^\epsilon_1-\overline{g}^\epsilon_2)\Vert_{\mathcal{S}}.
\end{aligned}
\end{equation}
Then the existence of the solutions to (\ref{5.035}) follows via the standard contraction mapping theorem for small enough $\epsilon$. This completes the proof.
\end{proof}

\appendix
\section{Well-posedness of solutions to the MHD boundary layer system (\ref{newequation})}\label{ap1}
To prove the Proposition \ref{wellposedness}, we shall focus on \emph{a priori estimates} of solutions to the system of equations (\ref{newequation}) in Sobolev spaces. Below, we derive the weighted estimates for $D^\alpha u$ and $D^\alpha h$, $|\alpha| \leq m$.
\begin{proposition}\label{pro1}
Let $m \geq 5$ be an integer and $l\in \mathbb{R}$ with $l \geq 0$, and suppose that $(u^0_e,h^0_e)$ satisfy the hypotheses in Proposition \ref{wellposedness}. Assume that \begin{equation}\label{c1}
u_e+u^0_p (0,y)>h_e+h^0_p (0,y)\geq \vartheta_0>0
\end{equation}
uniform in $y$ for some $\vartheta_0>0$, also if
\begin{equation}\label{c2}
\begin{aligned}
&| \langle y\rangle^{l+1}\partial_y (u_e+u^0_p,h_e+h^0_p)(0,y)| \leq \frac{1}{2}\sigma_0,\\
&| \langle y\rangle^{l+1}\partial_y^2 (u_e+u^0_p,h_e+h^0_p)(0,y)| \leq \frac{1}{2}\vartheta_0^{-1},
\end{aligned}
\end{equation}
uniform in $y$, here $\sigma_0>0$ is a suitably small constant. Then there exists a small $L>0$, such that the problem (\ref{newequation}) admits a classical solution $(u,v,h,g)$ in $[0,L]\times[0, +\infty)$, which satisfies
$$(u,h)\in L^\infty(0,L;H^m_l(0,\infty)), \ \partial_y (u,h)\in L^2(0,L;H^m_l(0,\infty)),$$
and the following estimates hold for small $L>0$,
\begin{equation}\label{estimate01}
\begin{aligned}
(i) \sup\limits_{x \in [0,L] }&\Vert (u,h)\Vert_{H^m_l}\\
\leq C&\vartheta_0^{-4}\left(P(C(u_e,u_b,h_e)+\Vert (u_0,h_0)\Vert_{H^m_l})+Cx\right)^{1/2}\\
&\cdot\left\{1-C\vartheta_0^{-24}\left(P(C(u_e,u_b,h_e)+\Vert (u_0,h_0)\Vert_{H^m_l})+Cx\right)^2x\right\}^{-1/4},
\end{aligned}
\end{equation}
\begin{equation}\label{estimate02}
\begin{aligned}
(ii)\ \Vert &\langle y\rangle^{l+1}\partial_y^i (u,h)\Vert_{L^\infty}\leq  \Vert \langle y\rangle^{l+1}\partial_y^i (u_0,h_0)\Vert_{L^\infty}\\
&\quad +Cx\vartheta_0^{-4}\left(P(C(u_e,u_b,h_e)+\Vert (u_0,h_0)\Vert_{H^m_l})+Cx\right)^{1/2}\\
&\cdot\left\{1-C\vartheta_0^{-24}\left(P(C(u_e,u_b,h_e)+\Vert (u_0,h_0)\Vert_{H^m_l})+Cx\right)^2x\right\}^{-1/4},  \ i=1,2,
\end{aligned}
\end{equation}
\begin{equation}\label{estimates04}
\begin{aligned}
(iii)\ h(x,y)& \geq  h_0-Cx\vartheta_0^{-4}\left(P(C(u_e,u_b,h_e)+\Vert (u_0,h_0)\Vert_{H^m_l})+Cx\right)^{1/2}\\
&\cdot\left\{1-C\vartheta_0^{-24}\left(P(C(u_e,u_b,h_e)+\Vert (u_0,h_0)\Vert_{H^m_l})+Cx\right)^2x\right\}^{-1/4},
\end{aligned}
\end{equation}
and
\begin{equation}\label{estimate05}
\begin{aligned}
(iv)\ u-h \geq& (u_0-h_0)\\
&-Cx\vartheta_0^{-4}\left(P(C(u_e,u_b,h_e)+\Vert (u_0,h_0)\Vert_{H^m_l})+Cx\right)^{1/2}\\
&\cdot\left\{1-C\vartheta_0^{-24}\left(P(C(u_e,u_b,h_e)+\Vert (u_0,h_0)\Vert_{H^m_l})+Cx\right)^2x\right\}^{-1/4}.
\end{aligned}
\end{equation}
Moreover, for $(x,y)\in [0,L]\times [0,\infty)$, it holds that
\begin{equation}\label{conditions}
\begin{aligned}
(v)\ & h(x,y)+h_e \phi(y) \geq \frac{\vartheta_0}{2}>0,\\
&u(x,y)+u_e\phi(y)+u_b(1-\phi(y))>h(x,y)+h_e\phi(y),\\
&\Vert \langle y\rangle^{l+1}\partial_y  (u,h)\Vert_{L^\infty} \leq \sigma_0,\\
&\Vert \langle y\rangle^{l+1}\partial_y^2  (u,h)\Vert_{L^\infty} \leq \vartheta_0^{-1}.
\end{aligned}
\end{equation}
\end{proposition}

In what follows, we will establish the \emph{a priori estimates} in Proposition \ref{pro1} by two steps. It is noted that all of the energy estimates will be derived based on the \emph{a priori assumptions} (\ref{conditions}). And the \emph{a priori assumptions} (\ref{conditions}) can be verified by the energy estimates established and the fact that $L$ is suitably small.

\textbf{Step 1: Estimates for $D^\alpha (u,h), |\alpha|\leq m, D^\alpha=\partial_x^\beta \partial_y^k, \beta \leq m-1$.}

\begin{lemma}[Weighted estimates for $D^\alpha (u,h)$ with $|\alpha|\leq m, \beta\leq m-1$]\label{estimate1}
Let $m \geq 5$ be an integer and $l\in \mathbb{R}$ with $l \geq 0$, and suppose that $(u^0_e,h^0_e)$ satisfy the hypotheses in Proposition \ref{wellposedness}. Assume that $(u,v,h,g)$ is a classical solution to the problem (\ref{newequation}) in $[0,L]$ and satisfies
$$(u,h)\in L^\infty(0,L;H^m_l(0,\infty)), \ \partial_y (u,h)\in L^2(0,L;H^m_l(0,\infty)).$$
Then, there exists a constant $C>0$, depending on $m,l,\phi,$ such that for any small $0<\delta_1<1$, it holds that
\begin{equation}\label{estimate0}
\begin{aligned}
&\sum_{\alpha\in\{\alpha=(\beta,k):|\alpha| \leq m,\beta\leq m-1\}}\bigg(s(x)+\nu\int_0^x \Vert \partial_y D^\alpha u\Vert_{L^2_l}^2+\kappa\int_0^x \Vert \partial_y D^\alpha h \Vert_{L^2_l}^2 \bigg)\\
\leq &C\delta_1 \int_0^x \Vert \partial_y (u,h)\Vert_{H^m_0}^2+C\delta_1^{-1}\int_0^x E_{u,h}^2 \left(1+E_{u,h}^2\right)+\int_0^x C(u_b,u_e,h_e)+s(0),
\end{aligned}
\end{equation}
where
$$s(x)=\Vert (u+u_e\phi(y)+u_b(1-\phi(y)))^{\frac{1}{2}}\langle y\rangle^{l+k}D^\alpha (u,h) \Vert_{L^2_y}^2$$
and
$$E_{u,h}^2=\sum_{|\alpha |\leq m}\Vert (u+u_e\phi(y)+u_b(1-\phi(y)))^{\frac{1}{2}}\langle y\rangle^{l+k}D^\alpha (u,h)\Vert_{L^2_y}^2.$$
\end{lemma}

\begin{proof}
Applying the operator $D^\alpha,\ \alpha=(\beta, k)$ with $\beta\leq m-1$ on the equations in (\ref{newequation}), then we have
\begin{equation}\label{newequation1}
\left \{
\begin{array}{lll}
\big[\big(u+u_e\phi(y)+u_b(1-\phi(y))\big)\partial_x +v\partial_y\big]D^\alpha u-\big[(h+h_e\phi(y))\partial_x+g\partial_y\big]D^\alpha h\\
\quad -\nu \partial_y^2D^\alpha u
=-[D^\alpha,\big(u+u_e\phi(y)+u_b(1-\phi(y))\big)\partial_x +v\partial_y]u+D^\alpha (gh_e\phi'(y))\\
\quad  +[D^\alpha,(h+h_e\phi(y))\partial_x+g\partial_y]h-D^\alpha(v(u_e-u_b)\phi'(y))+D^\alpha r_1,\\
\big[
\big(u+u_e\phi(y)+u_b(1-\phi(y))\big)\partial_x +v\partial_y\big]D^\alpha h-\big[(h+h_e\phi(y))\partial_x+g\partial_y\big]D^\alpha u\\
\quad-\kappa \partial_y^2D^\alpha h
=-[D^\alpha,\big(u+u_e\phi(y)+u_b(1-\phi(y))\big)\partial_x +v\partial_y]h\\
\quad +D^\alpha (g(u_e-u_b)\phi'(y))+[D^\alpha,(h+h_e\phi(y))\partial_x+g\partial_y]u\\
\quad-D^\alpha(vh_e\phi'(y))+D^\alpha r_2.
\end{array}
\right.
\end{equation}
Multiplying the first and second equations in (\ref{newequation1}) by $\langle y \rangle^{2(l+k)}D^\alpha u,\ \langle y \rangle^{2(l+k)}D^\alpha h$, respectively, integrating them over $[0,\infty)$ with respect to the variable $y$, and adding them together, we obtain that
\begin{equation}\label{2.00}
\begin{aligned}
\frac{1}{2} \frac{\mathrm{d}}{\mathrm{d}x}\Vert &\big(u+u_e\phi(y)+u_b(1-\phi(y))\big)^{\frac{1}{2}}\langle y\rangle^{l+k}D^\alpha (u,h)\Vert_{L^2_y(0,\infty)}^2\\
=&\int_0^\infty \left(D^\alpha r_1 \cdot \langle y\rangle^{2(l+k)}D^\alpha u+D^\alpha r_2 \cdot \langle y\rangle^{2(l+k)} D^\alpha h\right)\mathrm{d}y\\
&+\nu \int_0^\infty \partial_y^2 D^\alpha u\cdot \langle y \rangle^{2(l+k)}D^\alpha u\mathrm{d}y+\kappa\int_0^\infty \partial_y^2 D^\alpha h\cdot \langle y \rangle^{2(l+k)}D^\alpha h\mathrm{d}y\\
&+\int_0^\infty y \langle y\rangle^{2(l+k)-2}v\left(|D^\alpha u|^2+|D^\alpha h|^2\right)\mathrm{d}y\\
&-\int_0^\infty \left(I_1\cdot \langle y\rangle^{2(l+k)}D^\alpha u+I_2\cdot \langle y \rangle^{2(l+k)}D^\alpha h\right)\mathrm{d}y,
\end{aligned}
\end{equation}
where
\begin{equation}\label{2.01}
\left \{
\begin{aligned}
I_1=& [D^\alpha,\big(u+u_e\phi(y)+u_b(1-\phi(y))\big)\partial_x +v\partial_y]u\\
&-\big[(h+h_e\phi(y))\partial_x+g\partial_y\big]D^\alpha h-[D^\alpha,(h+h_e\phi(y))\partial_x+g\partial_y]h\\
&+D^\alpha (v(u_e-u_b)\phi'-gh_e\phi'),\\
I_2=& [D^\alpha,\big(u+u_e\phi(y)+u_b(1-\phi(y))\big)\partial_x +v\partial_y]h\\
&-\big[(h+h_e\phi(y))\partial_x+g\partial_y\big]D^\alpha u-[D^\alpha,(h+h_e\phi(y))\partial_x+g\partial_y]u\\
&+D^\alpha (vh_e\phi'-g(u_e-u_b)\phi').
\end{aligned}
\right.
\end{equation}

First, based on the following {\it a priori assumption},
\begin{align}
\label{AA}
u+u_e\phi(y)+u_b(1-\phi(y)) \geq \tilde{c}>0
\end{align}
for some $\tilde{c}>0$, we have
\begin{equation}\label{2.02}
\begin{aligned}
\int_0^\infty & \left(D^\alpha r_1 \cdot \langle y\rangle^{2(l+k)}D^\alpha u+D^\alpha r_2 \cdot \langle y\rangle^{2(l+k)} D^\alpha h\right)\mathrm{d}y\\
\leq & \frac{1}{2}\Vert \big(u+u_e\phi(y)+u_b(1-\phi(y))\big)^{\frac{1}{2}}\langle y\rangle^{l+k} D^\alpha (u,h)\Vert_{L^2_y}^2\\
&+\frac{1}{2}C(\vartheta_0,\tilde{c})\Vert \langle y\rangle^{l+k} D^\alpha (r_1,r_2)\Vert_{L^2_y}^2.
\end{aligned}
\end{equation}
Next, we will estimate the other terms in (\ref{2.00}). Since the term $\nu \int_0^\infty \partial_y^2 D^\alpha u\cdot \langle y \rangle^{2(l+k)}D^\alpha u\mathrm{d}y$ and the term $\kappa\int_0^\infty \partial_y^2 D^\alpha h\cdot \langle y \rangle^{2(l+k)}D^\alpha h\mathrm{d}y$ can be estimated similarly, we only handle the first one.
\begin{equation}\label{2.03}
\begin{aligned}
\nu \int_0^\infty &\partial_y^2 D^\alpha u\cdot \langle y \rangle^{2(l+k)}D^\alpha u\mathrm{d}y\\
=&-\nu \Vert \langle y\rangle^{l+k}\partial_y D^\alpha u\Vert_{L^2_y}^2\\
&+\nu (\partial_y D^\alpha u\cdot D^\alpha u)|_{y=0}+2\nu (l+k)\int_0^\infty \langle y \rangle^{2(l+k)-2} y\partial_y D^\alpha u\cdot D^\alpha u\mathrm{d}y.
\end{aligned}
\end{equation}
The last term in (\ref{2.03}) can be estimated as
\begin{equation}\label{2.04}
\begin{aligned}
2\nu (l+k)& \int_0^\infty \langle y \rangle^{2(l+k)-2} y\partial_y D^\alpha u\cdot D^\alpha u\mathrm{d}y\\
& \leq \frac{\nu}{14}\Vert \langle y\rangle^{l+k}\partial_y D^\alpha u\Vert_{L^2_y}^2+14\nu (l+k)^2\Vert \langle y\rangle^{l+k} D^\alpha u\Vert_{L^2_y}^2,
\end{aligned}
\end{equation}
it suffices to control the boundary integral term, and it will be treated by the following two cases.

\emph{Case 1: $|\alpha|\leq m-1$}.

In this case, it is easy to deduce that for any small $0<\delta_1 <1$,
\begin{equation}\label{2.05}
\begin{aligned}
\nu (\partial_y D^\alpha u\cdot D^\alpha u)|_{y=0} \leq & \nu \Vert \partial_y^2 D^\alpha u\Vert_{L^2_y}\Vert D^\alpha u\Vert_{L^2_y}+\nu \Vert \partial_y D^\alpha u\Vert_{L^2_y}^2\\
\leq &\delta_1 \Vert \partial_y^2D^\alpha u\Vert_{L^2_y}^2+\nu \Vert \partial_y D^\alpha u\Vert_{L^2_y}^2\\
&+\frac{\nu^2}{4\delta_1\tilde{c}}\Vert \big(u+u_e\phi(y)+u_b(1-\phi(y))\big)^{\frac{1}{2}}D^\alpha u \Vert_{L^2_y}^2\\
\leq &\delta_1 \Vert \partial_y  u\Vert_{H^m_0}^2+C\delta_1^{-1} E_u^2,
\end{aligned}
\end{equation}
where
$$E_u^2=\sum_{|\alpha |\leq m}\Vert (u+u_e\phi(y)+u_b(1-\phi(y)))^{\frac{1}{2}}\langle y\rangle^{l+k}D^\alpha u\Vert_{L^2_y}^2.$$

\emph{Case 2: $|\alpha|=\beta+k=m$.}

In this case, one has $k \geq 1$ from $\beta \leq m-1$. Denote $\gamma=(\beta,k-1)$, then$|\gamma|=m-1$ and $\partial_y D^\alpha=D^\gamma \partial_y^2$. Therefore,
\begin{equation}\nonumber
\begin{aligned}
\nu \partial_y D^\alpha u=\nu D^\gamma\partial^2_y u=D^\gamma &\bigg\{\big[(u+u_e\phi(y)+u_b(1-\phi(y)))\partial_x +v\partial_y\big]u\\
&-\big[(h+h_e\phi(y))\partial_x+g\partial_y\big]h\\
&-gh_e\phi'(y)+v(u_e-u_b)\phi'(y)-r_1 \bigg\}.
\end{aligned}
\end{equation}
Note that $\phi \equiv 0$ near $\{y=0\}$, then on $y=0$, one gets that
\begin{equation}\label{2.05}
\begin{aligned}
\nu \partial_y D^\alpha u|_{y=0}=&D^\gamma \big[\big((u+u_b)\partial_x +v\partial_y\big)u-\big(h\partial_x+g\partial_y\big)h \big]|_{y=0}\\
=& \bigg[D^\gamma  \left( (u+u_b)\partial_x u-h\partial_x h\right)+D^\gamma \left(v \partial_y u-g\partial_y h\right)\bigg]|_{y=0}.
\end{aligned}
\end{equation}
By Leibnitz formula, it is direct to calculate that
$$D^\gamma ((u+u_b)\partial_x u)=\sum_{\tilde{\gamma}\leq \gamma}\begin{pmatrix} \gamma  \\ \tilde{\gamma} \end{pmatrix}(D^{\tilde{\gamma}}(u+u_b)\cdot D^{\gamma-\tilde{\gamma}+(1,0)}u),$$
then
\begin{equation}\label{2.06}
\begin{aligned}
\big|&(D^\gamma ((u+u_b)\partial_x u)\cdot D^\alpha u)|_{y=0}\big| \\
\leq &C \sum_{\tilde{\gamma}\leq \gamma}\bigg(\Vert \partial_y (D^{\tilde{\gamma}}(u+u_b)\cdot D^{\gamma-\tilde{\gamma}+(1,0)}u)\Vert_{L^2_y}\Vert D^\alpha u\Vert_{L^2_y}\\
&+\Vert D^{\tilde{\gamma}}(u+u_b)\cdot D^{\gamma-\tilde{\gamma}+(1,0)}u \Vert_{L^2_y}\Vert \partial_y D^\alpha u\Vert_{L^2_y}\bigg).
\end{aligned}
\end{equation}
Note that $|\gamma|=m-1 \geq 3$, then we have
\begin{equation}\label{2.07}
\begin{aligned}
\Vert \partial_y &(D^{\tilde{\gamma}}(u+u_b)\cdot D^{\gamma-\tilde{\gamma}+(1,0)}u)\Vert_{L^2_y} \\
\leq &\Vert \partial_y D^{\tilde{\gamma}}(u+u_b)\cdot D^{\gamma-\tilde{\gamma}+(1,0)}u\Vert_{L^2_y}+\Vert D^{\tilde{\gamma}}(u+u_b)\cdot D^{\gamma-\tilde{\gamma}+(1,0)}\partial_y u\Vert_{L^2_y}\\
\leq &\Vert \partial_y u\Vert_{H^{m-1}_0}\Vert \partial_x u\Vert_{H^{m-1}_0}+\Vert u\Vert_{H^{m-1}_0}\Vert \partial_{xy} u\Vert_{H^{m-1}_0}+|u_b|\Vert \partial_{xy} u\Vert_{H^{m-1}_0}\\
\leq & C \sum_{|\alpha |\leq m}\Vert (u+u_e\phi(y)+u_b(1-\phi(y)))^{\frac{1}{2}}\langle y\rangle^{l+k}D^\alpha u\Vert_{L^2_y}\Vert \partial_y u \Vert_{H^m_0}\\
&+C\sum_{|\alpha |\leq m}\Vert (u+u_e\phi(y)+u_b(1-\phi(y)))^{\frac{1}{2}}\langle y\rangle^{l+k}D^\alpha u\Vert_{L^2_y}^2\\
&+|u_b|\Vert \partial_{y} u\Vert_{H^{m}_0}
\end{aligned}
\end{equation}
and
\begin{equation}\label{2.08}
\begin{aligned}
\Vert &D^{\tilde{\gamma}}(u+u_b)\cdot D^{\gamma-\tilde{\gamma}+(1,0)}u \Vert_{L^2_y} \\
&\leq C\sum_{|\alpha |\leq m}\Vert (u+u_e\phi(y)+u_b(1-\phi(y)))^{\frac{1}{2}}\langle y\rangle^{l+k}D^\alpha u\Vert_{L^2_y}^2+|u_b|\Vert  u\Vert_{H^{m}_0}.
\end{aligned}
\end{equation}
Therefore we get that
\begin{equation}\label{2.09}
\begin{aligned}
\bigg|(D^\gamma& ((u+u_b)\partial_x u)\cdot D^\alpha u)|_{y=0}\bigg| \\
\leq &C\sum_{\tilde{\gamma}\leq \gamma}\bigg\{\bigg(\sum_{|\alpha|\leq m}\Vert (u+u_e\phi(y)+u_b(1-\phi(y)))^{\frac{1}{2}}\langle y\rangle^{l+k}D^\alpha u\Vert_{L^2_y}\Vert \partial_y u \Vert_{H^m_0}\\
&+\sum_{|\alpha |\leq m}\Vert (u+u_e\phi(y)+u_b(1-\phi(y)))^{\frac{1}{2}}\langle y\rangle^{l+k}D^\alpha u\Vert_{L^2_y}^2+|u_b|\Vert \partial_{y} u\Vert_{H^{m}_0}\bigg)\Vert D^\alpha u\Vert_{L^2_y}\\
&+\sum_{|\alpha |\leq m}\Vert (u+u_e\phi(y)+u_b(1-\phi(y)))^{\frac{1}{2}}\langle y\rangle^{l+k}D^\alpha u\Vert_{L^2_y}^2\Vert \partial_y D^\alpha u\Vert_{L^2_y}\bigg\}\\
\leq & \frac{\delta_1}{3}\Vert \partial_y u\Vert_{H^m_0}^2+\frac{\nu}{14}\Vert \partial_y D^\alpha u\Vert_{L^2_y}^2+C\delta_1^{-1}E_u^4+CE_u^2+C(u_b),
\end{aligned}
\end{equation}
where
$$E_u^2=\sum_{|\alpha |\leq m}\Vert (u+u_e\phi(y)+u_b(1-\phi(y)))^{\frac{1}{2}}\langle y\rangle^{l+k}D^\alpha u\Vert_{L^2_y}^2.$$

Similarly, we have
\begin{equation}\label{2.010}
\begin{aligned}
\bigg|(D^\gamma& (h\partial_x h)\cdot D^\alpha u)|_{y=0}\bigg| \\
\leq & \frac{\delta_1}{3}\Vert \partial_y (u,h) \Vert_{H^m_0}^2+\frac{\nu}{14}\Vert \partial_y D^\alpha u\Vert_{L^2_y}^2+C\delta_1^{-1}E_{u,h}^4+CE_{u,h}^2,
\end{aligned}
\end{equation}
where
$$E_{u,h}^2=\sum_{|\alpha |\leq m}\Vert (u+u_e\phi(y)+u_b(1-\phi(y)))^{\frac{1}{2}}\langle y\rangle^{l+k}D^\alpha (u,h)\Vert_{L^2_y}^2.$$

Now we turn to control the term
$$\bigg|(D^\gamma (v\partial_yu)\cdot D^\alpha u)|_{y=0}\bigg|.$$
Note that $D^\gamma=\partial_x^\beta\partial_y^{k-1}$ and $v|_{y=0}=0$, then we have
\begin{equation}\label{2.011}
\begin{aligned}
D^\gamma (v\partial_y u)=&\partial_x^\beta \left(v \partial_y^k u+\sum_{i=1}^{k-1} \begin{pmatrix} k-1  \\ i \end{pmatrix}\partial_y^i v\cdot \partial_y^{k-i}u  \right)\\
=&-\sum_{\tilde{\beta}\leq \beta, 0\leq j\leq k-2} \begin{pmatrix} k-1  \\ j+1 \end{pmatrix} \begin{pmatrix} \beta  \\  \tilde{\beta} \end{pmatrix}
\left( \partial_x^{\tilde{\beta}+1}\partial_y^j u\cdot \partial_x^{\beta-\tilde{\beta}}\partial_y^{k-j-1}u\right),
\end{aligned}
\end{equation}
where we have denoted $\begin{pmatrix} j  \\ i \end{pmatrix}=0$ for $i>j$. When $k=1$ the right hand side will disappear, we only need to consider the case $k\geq 2$. By the formula (\ref{2.011}), we have
\begin{equation}\label{2.012}
\begin{aligned}
\bigg| &\left( D^\gamma (v\partial_y u)\cdot D^\alpha u\right)|_{y=0}\bigg| \\
\leq &C \sum_{\tilde{\beta}\leq \beta, 0\leq j\leq k-2} \bigg(\Vert \partial_y (\partial_x^{\tilde{\beta}+1}\partial_y^ju\cdot\partial_x^{\beta-\tilde{\beta}}\partial_y^{k-j-1}u)\Vert_{L^2_y}\Vert D^\alpha u\Vert_{L^2_y}\\
&+\Vert \partial_x^{\tilde{\beta}+1}\partial_y^ju\cdot \partial_x^{\beta-\tilde{\beta}}\partial_y^{k-j-1}u\Vert_{L^2_y}\Vert \partial_y D^\alpha u\Vert_{L^2_y}\bigg).
\end{aligned}
\end{equation}
Since $0 \leq j \leq k-2$, we have
\begin{equation}\label{2.013}
\begin{aligned}
\Vert \partial_y &(\partial_x^{\tilde{\beta}+1}\partial_y^j u\cdot \partial_x^{\beta-\tilde{\beta}}\partial_y^{k-j-1}u)\Vert_{L^2_y} \\
\leq&\Vert \partial_x^{\tilde{\beta}+1}\partial_y^{j+1}u\cdot \partial_x^{\beta-\tilde{\beta}}\partial_y^{k-j-1}u\Vert_{L^2_y}+\Vert \partial_x^{\tilde{\beta}+1}\partial_y^ju\cdot \partial_x^{\beta-\tilde{\beta}}\partial_y ^{k-j}u\Vert_{L^2_y}\\
\leq &C\Vert \partial_y u\Vert_{H^{m-1}_l}^2+C\Vert \partial_x u\Vert_{H^{m-1}_l}\Vert \partial_y u\Vert_{H^{m-1}_l}\\
\leq &C\Vert (u+u_e\phi(y)+u_b(1-\phi(y)))^{\frac{1}{2}}\langle y\rangle^{l+k}D^\alpha u \Vert_{L^2_y}^2
\end{aligned}
\end{equation}
and
\begin{equation}\label{2.014}
\begin{aligned}
\Vert & \partial_x^{\tilde{\beta}+1}\partial_y^ju\cdot \partial_x^{\beta-\tilde{\beta}}\partial_y^{k-j-1}u\Vert_{L^2_y}\\
 &\leq C\Vert (u+u_e\phi(y)+u_b(1-\phi(y)))^{\frac{1}{2}}\langle y\rangle^{l+k}D^\alpha u \Vert_{L^2_y}^2,
\end{aligned}
\end{equation}
provided that $\beta+k=|\alpha|=m$. Therefore, we get that
\begin{equation}\label{2.015}
\begin{aligned}
\bigg| &\left( D^\gamma (v\partial_y u)\cdot D^\alpha u\right)|_{y=0}\bigg| \\
\leq &C \sum_{\tilde{\beta}\leq \beta, 0\leq j\leq k-2}  \bigg(E_u^2 \Vert D^\alpha u\Vert_{L^2_y}+E_u^2\Vert \partial_y D^\alpha u\Vert_{L^2_y}\bigg)\\
\leq & \frac{\nu}{14}\Vert \partial_y D^\alpha u\Vert_{L^2_y}^2+CE_u^4+C E_u^2,
\end{aligned}
\end{equation}
where $E_u$ is defined as before.

Applying the similar arguments to yield that
\begin{equation}\label{2.016}
\begin{aligned}
\bigg| \left( D^\gamma (g\partial_y h)\cdot D^\alpha u\right)|_{y=0}\bigg|
\leq  \frac{\nu}{14}\Vert \partial_y D^\alpha u\Vert_{L^2_y}^2+CE_{u,h}^4+C E_{u,h}^2.
\end{aligned}
\end{equation}
Consequently, for $|\alpha|=\beta+k\leq m$ with $\beta \leq m-1$, we obtain
\begin{equation}\label{2.017}
\begin{aligned}
\bigg| \left(\nu\partial_y D^\alpha u\cdot D^\alpha u\right)|_{y=0}\bigg| \leq & \delta_1\Vert \partial_y (u,h)\Vert_{H^m_0}^2+\frac{3\nu}{7}\Vert \partial_y D^\alpha u\Vert_{L^2_y}^2\\
& +C\delta_1^{-1} E_{u,h}^2(1+E_{u,h}^2)+C(u_b).
\end{aligned}
\end{equation}
Putting the above estimates into (\ref{2.03}), we arrive at
\begin{equation}\label{2.018}
\begin{aligned}
\nu \int_0^\infty \partial_y^2& D^\alpha u\cdot \langle y \rangle^{2(l+k)}D^\alpha u\mathrm{d}y\leq -\frac{\nu}{2}\Vert \langle y\rangle^{l+k}\partial_yD^\alpha u\Vert_{L^2_y}^2\\
&+\delta_1\Vert \partial_y (u,h)\Vert_{H^m_0}^2+C\delta_1^{-1} E_{u,h}^2(1+E_{u,h}^2)+C(u_b).
\end{aligned}
\end{equation}
By similar arguments, one achieves that
\begin{equation}\label{2.019}
\begin{aligned}
\kappa \int_0^\infty \partial_y^2& D^\alpha h\cdot \langle y \rangle^{2(l+k)}D^\alpha h\mathrm{d}y\\
\leq& -\frac{\kappa}{2}\Vert \langle y\rangle^{l+k}\partial_yD^\alpha h\Vert_{L^2_y}^2+\delta_1\Vert \partial_y (u,h)\Vert_{H^m_0}^2\\
&+C\delta_1^{-1} E_{u,h}^2(1+E_{u,h}^2).
\end{aligned}
\end{equation}

It remains to control the following terms
$$-\int_0^\infty (I_1\cdot \langle y\rangle^{l+k}D^\alpha u+I_2\cdot \langle y\rangle^{l+k}D^\alpha h)\mathrm{d}y$$
and
$$\int_0^\infty (l+k)y\langle y\rangle^{2(l+k)-2}\cdot v\left(|D^\alpha u|^2+|D^\alpha h|^2\right)\mathrm{d}y.$$
First of all, it is easy to see that

\begin{equation}\label{2.020}
\begin{aligned}
\int_0^\infty &(l+k)y\langle y\rangle^{2(l+k)-2}\cdot v\left(|D^\alpha u|^2+|D^\alpha h|^2\right)\mathrm{d}y \\
\leq & C \left\Vert \frac{v}{\langle y\rangle}\right\Vert_{L^\infty_y} \Vert (u+u_e\phi(y)+u_b(1-\phi(y)))^{\frac{1}{2}}\langle y\rangle^{l+k}D^\alpha (u,h) \Vert_{L^2_y}^2\\
\leq &C \Vert u_x \Vert_{L^\infty_y} \Vert (u+u_e\phi(y)+u_b(1-\phi(y)))^{\frac{1}{2}}\langle y\rangle^{l+k}D^\alpha (u,h) \Vert_{L^2_y}^2\\
\leq &C \Vert u\Vert_{H^3_0} \Vert (u+u_e\phi(y)+u_b(1-\phi(y)))^{\frac{1}{2}}\langle y\rangle^{l+k}D^\alpha (u,h) \Vert_{L^2_y}^2.
\end{aligned}
\end{equation}

For $I_1$ and $I_2$, we have
\begin{equation}\nonumber
\begin{aligned}
I_1
=&-[(h+h_e\phi(y))\partial_x+g\partial_y]D^\alpha h\\
&+[D^\alpha,[(u+u_e\phi(y)+u_b(1-\phi(y)))\partial_x +v\partial_y ]-[D^\alpha,(h+h_e\phi)\partial_x+g\partial_y ]h\\
&+D^\alpha (-gh_e\phi'(y)+v(u_e-u_b)\phi'(y))\\
=:&I^1_1+I^2_1+I^3_1,
\end{aligned}
\end{equation}
and
\begin{equation}\nonumber
\begin{aligned}
I_2=&-[(h+h_e\phi(y))\partial_x+g\partial_y]D^\alpha u\\
&+[D^\alpha,[(u+u_e\phi(y)+u_b(1-\phi(y)))\partial_x +v\partial_y ]h-[D^\alpha,(h+h_e\phi)\partial_x+g\partial_y ]u\\
&+D^\alpha (-g(u_e-u_b)\phi'(y)+vh_e\phi'(y))\\
=:&I^1_2+I^2_2+I^3_2.
\end{aligned}
\end{equation}
Therefore, we only need to control the following three parts
\begin{equation}\label{2.021}
\begin{aligned}
-&\int_0^\infty(I_1\cdot \langle y\rangle^{2(l+k)}D^\alpha u+I_2\cdot \langle y\rangle^{2(l+k)}D^\alpha h)\mathrm{d}y\\
&=-\sum_{i=1}^3 \int_0^\infty (I_1^i\cdot \langle y\rangle^{2(l+k)}D^\alpha u+I^i_2\cdot \langle y\rangle^{2(l+k)}D^\alpha h)\mathrm{d}y\\
&=:J_1+J_2+J_3.
\end{aligned}
\end{equation}
By the definition of $\phi(y)$ in (\ref{cutoff}), then we have
$$\Vert \langle y\rangle^{i-1}\phi^{(i)}(y)\Vert_{L^\infty_y}, \ \Vert \langle y\rangle^\lambda \phi^{(j)}(y)\Vert_{L^\infty_y}\leq C$$
for $i=0,1,j\geq 2, \lambda \in \mathbb{R}$.

\textbf{\emph{Estimate of $J_1.$}}

\begin{equation}\nonumber
\begin{aligned}
J_1=&\int_0^\infty [(h+h_e\phi(y))\partial_x+g\partial_y]D^\alpha h\cdot\langle y\rangle^{2(l+k)}D^\alpha u\mathrm{d}y\\
&+\int_0^\infty [(h+h_e\phi(y))\partial_x+g\partial_y]D^\alpha u\cdot \langle y\rangle^{2(l+k)}D^\alpha h\mathrm{d}y.
\end{aligned}
\end{equation}
By integrating by parts, we have
\begin{equation}\nonumber
\begin{aligned}
J_1=&\frac{\mathrm{d}}{\mathrm{d}x}\int_0^\infty (h+h_e\phi)D^\alpha h\cdot \langle y\rangle^{2(l+k)}D^\alpha u \mathrm{d}y\\
&-\int_0^\infty (l+k)gD^\alpha h\cdot \langle y \rangle^{2(l+k)-2}2yD^\alpha u\mathrm{d}y.
\end{aligned}
\end{equation}
The second term can be estimated as
\begin{equation}\nonumber
\begin{aligned}
-\int_0^\infty& gD^\alpha h\cdot \langle y \rangle^{2(l+k)-2}2yD^\alpha u\mathrm{d}y\\
\leq &C\left\Vert \frac{g}{\langle y\rangle}\right\Vert_{L^\infty_y} \Vert(u+u_e\phi(y)+u_b(1-\phi(y)))^{\frac{1}{2}}\langle y\rangle^{l+k}D^\alpha (u,h) \Vert_{L^2_y}^2\\
\leq &C \Vert h_x \Vert_{L^\infty_y} \Vert (u+u_e\phi(y)+u_b(1-\phi(y)))^{\frac{1}{2}}\langle y\rangle^{l+k}D^\alpha (u,h) \Vert_{L^2_y}^2\\
\leq &C \Vert h \Vert_{H^3_0} \Vert (u+u_e\phi(y)+u_b(1-\phi(y)))^{\frac{1}{2}}\langle y\rangle^{l+k}D^\alpha (u,h) \Vert_{L^2_y}^2.
\end{aligned}
\end{equation}

\textbf{\emph{Estimate of $J_2$.}}

It is easy to find that
\begin{equation}\nonumber
\begin{aligned}
J_2 \leq &C\Vert \langle y\rangle^{l+k} I^2_1\Vert_{L^2_y}\Vert (u+u_e\phi(y)+u_b(1-\phi(y)))^{\frac{1}{2}}\langle y\rangle^{l+k}D^\alpha u \Vert_{L^2_y}\\
&+C\Vert \langle y\rangle^{l+k} I^2_2\Vert_{L^2_y}\Vert (u+u_e\phi(y)+u_b(1-\phi(y)))^{\frac{1}{2}}\langle y\rangle^{l+k}D^\alpha h \Vert_{L^2_y},
\end{aligned}
\end{equation}
thus it remains to estimate $\Vert \langle y\rangle^{l+k} I^2_i\Vert_{L^2_y}\ (i=1,2)$. Keep this goal in mind, we will only handle the term $\Vert \langle y\rangle^{l+k} I^2_1\Vert_{L^2_y}$, and the term of $\Vert \langle y\rangle^{l+k} I^2_2\Vert_{L^2_y}$ can be done by the similar arguments.

Recall that
\begin{equation}\nonumber
\begin{aligned}
I_1^2=&[D^\alpha,u\partial_x+v\partial_y]u-[D^\alpha,h\partial_x+g\partial_y]h\\
&+[D^\alpha,(u_e\phi+u_b(1-\phi))\partial_x]u-[D^\alpha,h_e\phi\partial_x]h\\
=:&I^2_{1,1}+I^2_{1,2}.
\end{aligned}
\end{equation}
So, we will estimate $\Vert \langle y\rangle^{l+k} I^2_{1,i}\Vert_{L^2_y} (i=1,2)$ respectively in the subsequent parts.

For $I_{1,1}^2$,
\begin{equation}\label{2.022}
\begin{aligned}
I^2_{1,1}=\sum_{0<\tilde{\alpha}\leq \alpha}&\begin{pmatrix} \alpha  \\ \tilde{\alpha} \end{pmatrix}\bigg\{(D^{\tilde{\alpha}}u\partial_x+D^{\tilde{\alpha}}v\partial_y)D^{\alpha-\tilde{\alpha}}u\\
&-(D^{\tilde{\alpha}}h\partial_x+D^{\tilde{\alpha}}g\partial_y)D^{\alpha-\tilde{\alpha}}h\bigg\}.
\end{aligned}
\end{equation}

Denote $\tilde{\alpha}=(\tilde{\beta},\tilde{k})$, the above terms in (\ref{2.022}) will be classified into two cases.

\emph{Case A: $\tilde{k}=0$.}

In this case, $D^{\tilde{\alpha}}=\partial_x^{\tilde{\beta}}, \tilde{\beta}\geq1$. Then, we have
\begin{equation}\nonumber
\begin{aligned}
\Vert \langle y\rangle^{l+k}D^{\tilde{\alpha}}u\cdot D^{\alpha-\tilde{\alpha}} \partial_x u\Vert_{L^2_y}&=\Vert\langle y\rangle^{l+k} \partial_x^{\tilde{\beta}-1}(\partial_x u)\cdot D^{\alpha-\tilde{\alpha}}(\partial_x u)\Vert_{L^2_y}\\
&\leq C\Vert  u\Vert_{H^{m}_l}^2\\
&\lesssim E_u^2,
\end{aligned}
\end{equation}
provided that $m-1 \geq 3$.

Similar arguments yield that
\begin{equation}\nonumber
\begin{aligned}
\Vert \langle y\rangle^{(l+k)}D^{\tilde{\alpha}}h\cdot D^{\alpha-\tilde{\alpha}} \partial_x h\Vert_{L^2_y}\lesssim E_h^2,
\end{aligned}
\end{equation}
where
$$E_h^2=\sum_{|\alpha |\leq m}\Vert (u+u_e\phi(y)+u_b(1-\phi(y)))^{\frac{1}{2}}\langle y\rangle^{l+k}D^\alpha h\Vert_{L^2_y}^2.$$

On the other hand, since $v=-\partial_y^{-1}\partial_x u$, one has
$$D^{\tilde{\alpha}}v\cdot \partial_y D^{\alpha-\tilde{\alpha}}u=-\partial_x^{\tilde{\beta}}\partial_y^{-1}\partial_xu\cdot \partial_x^{\beta-\tilde{\beta}}\partial_y^{k+1}u.$$
Thus, for $|\alpha|=\beta+k\leq m-1$ with $m-1\geq 3$, we get
\begin{equation}\nonumber
\begin{aligned}
\Vert \langle y\rangle^{l+k}D^{\tilde{\alpha}}v\cdot D^{\alpha-\tilde{\alpha}} \partial_y u \Vert_{L^2_y}=&\Vert \langle y\rangle^{l+k}\partial_x^{\tilde{\beta}}\partial_y^{-1}(\partial_x u)\cdot \partial_x^{\beta-\tilde{\beta}}\partial_y^k(\partial_y u) \Vert_{L^2_y}\\
\leq &C \Vert \partial_x u\Vert_{H^{m-1}_0}\Vert \partial_y u\Vert_{H^{m-1}_l}\\
\lesssim &E_u^2.
\end{aligned}
\end{equation}
When $|\alpha|=m$, it yields that $k\geq 1$ since $\beta\leq m-1$, therefore,
\begin{equation}\nonumber
\begin{aligned}
\Vert &\langle y\rangle^{l+k}D^{\tilde{\alpha}}v\cdot D^{\alpha-\tilde{\alpha}} \partial_y u \Vert_{L^2_y}\\
=&\Vert \langle y\rangle^{l+k}\partial_x^{\tilde{\beta}-1}\partial_y^{-1}(\partial_x^2 u)\cdot \partial_x^{\beta-\tilde{\beta}}\partial_y^{k-1}(\partial_y^2 u) \Vert_{L^2_y}\\
\leq &C \Vert \partial_x^2 u\Vert_{H^{m-2}_0}\Vert \partial_y^2 u\Vert_{H^{m-2}_l}\\
\lesssim &E_u^2.
\end{aligned}
\end{equation}
provided that $m-2 \geq 3$. Hence, for $|\alpha|\leq m, \beta \leq m-1$, we have
\begin{equation}\nonumber
\begin{aligned}
\Vert \langle y\rangle^{l+k}D^{\tilde{\alpha}}v\cdot D^{\alpha-\tilde{\alpha}} \partial_y u \Vert_{L^2_y}
\lesssim E_u^2.
\end{aligned}
\end{equation}
Similarly, one obtains
\begin{equation}\nonumber
\begin{aligned}
\Vert \langle y\rangle^{l+k}D^{\tilde{\alpha}}g\cdot D^{\alpha-\tilde{\alpha}} \partial_y h \Vert_{L^2_y}
\lesssim E_h^2.
\end{aligned}
\end{equation}

Consequently, for $\tilde{\alpha}=(\tilde{\beta},\tilde{k})$ with $\tilde{k}=0$, it holds that
\begin{equation}\label{2.023}
\begin{aligned}
\big\Vert \langle y\rangle^{l+k} \{(D^{\tilde{\alpha}}u\partial_x&+D^{\tilde{\alpha}}v\partial_y)D^{\alpha-\tilde{\alpha}}u\\
&-(D^{\tilde{\alpha}}h\partial_x+D^{\tilde{\alpha}}g\partial_y)D^{\alpha-\tilde{\alpha}}h\}\big\Vert_{L^2_y}\\
&\lesssim E_{u,h}^2.
\end{aligned}
\end{equation}

\emph{Case B: $\tilde{k} \geq 1$.}

For this case, we have
\begin{equation}\nonumber
\begin{aligned}
(D^{\tilde{\alpha}}u\partial_x&+D^{\tilde{\alpha}}v\partial_y)(D^{\alpha-\tilde{\alpha}}u)
-(D^{\tilde{\alpha}}h\partial_x+D^{\tilde{\alpha}}g\partial_y)(D^{\alpha-\tilde{\alpha}}h)\\
=&(D^{\tilde{\alpha}}u\partial_x-D^{\tilde{\alpha}-(0,1)}(\partial_x u)\partial_y)(D^{\alpha-\tilde{\alpha}}u)\\
&-(D^{\tilde{\alpha}}h\partial_x-D^{\tilde{\alpha}-(0,1)}(\partial_x h)\partial_y)(D^{\alpha-\tilde{\alpha}}h).
\end{aligned}
\end{equation}
Then, each term of the right-hand side of the above equality is estimated as follows.
\begin{equation}\nonumber
\begin{aligned}
\Vert D^{\tilde{\alpha}}u\partial_x D^{\alpha-\tilde{\alpha}}u\Vert_{L^2_{l+k}}=&\Vert D^{\tilde{\alpha}-(0,1)}\partial_y u\partial_x D^{\alpha-\tilde{\alpha}}u\Vert_{L^2_{l+1+(k-1)}}\\
\leq & C\Vert \partial_y u\Vert_{H^{m-1}_{l+1}}\Vert \partial_x u\Vert_{H^{m-1}_0}\\
\lesssim & E_u^2,
\end{aligned}
\end{equation}
and
\begin{equation}\nonumber
\begin{aligned}
\Vert D^{\tilde{\alpha}-(0,1)}(\partial_x u)\partial_yD^{\alpha-\tilde{\alpha}}u\Vert_{L^2_{l+k}}=&\Vert D^{\tilde{\alpha}-(0,1)}(\partial_x u)\partial_yD^{\alpha-\tilde{\alpha}}u\Vert_{L^2_{l+1+(k-1)}}\\
\leq & C\Vert \partial_y u\Vert_{H^{m-1}_{l+1}}\Vert \partial_x u\Vert_{H^{m-1}_0}\\
\lesssim & E_u^2.
\end{aligned}
\end{equation}
Similarly,
$$\Vert(D^{\tilde{\alpha}}h\partial_x-D^{\tilde{\alpha}-(0,1)}(\partial_x h)\partial_y)(D^{\alpha-\tilde{\alpha}}h)\Vert_{L^2_{l+k}} \lesssim E_h^2.$$

Collecting the above estimates leads to
\begin{equation}\label{2.024}
\begin{aligned}
\Vert \langle y\rangle^{l+k} I^2_{1,1}\Vert_{L^2_y} \leq C E_{u,h}^2.
\end{aligned}
\end{equation}

Next, we are in position to estimate the term $I^2_{1,2}$. Rewrite the term $I^2_{1,2}$ as follows:
\begin{equation}\nonumber
\begin{aligned}
I^2_{1,2}=\sum_{0 <\tilde{\alpha}\leq \alpha}\begin{pmatrix} \alpha  \\ \tilde{\alpha} \end{pmatrix}\bigg\{D^{\tilde{\alpha}}(u_e\phi+u_b(1-\phi))D^{\alpha-\tilde{\alpha}}(\partial_x u)-D^{\tilde{\alpha}}(h_e\phi)D^{\alpha-\tilde{\alpha}}(\partial_x h)\bigg\}.
\end{aligned}
\end{equation}
Therefore, one has
\begin{equation}\label{2.025}
\begin{aligned}
\Vert \langle y\rangle^{l+k} I^2_{1,2}\Vert_{L^2_y} \leq C(u_e,u_b,h_e) E_{u,h}.
\end{aligned}
\end{equation}

Consequently, from the above estimates, we achieve that
\begin{equation}\label{2.026}
\begin{aligned}
\Vert \langle y\rangle^{l+k} I^2_{1}\Vert_{L^2_y} \leq C(C(u_e,u_b,h_e)+E_{u,h})E_{u,h}.
\end{aligned}
\end{equation}
By similar arguments, we also obtain
\begin{equation}\label{2.027}
\begin{aligned}
\Vert \langle y\rangle^{l+k} I^2_{2}\Vert_{L^2_y} \leq C(C(u_e,u_b,h_e)+E_{u,h})E_{u,h}.
\end{aligned}
\end{equation}
It is noted that the constant $C(u_e,u_b,h_e)>0$ depends only on the given data $u_e,u_b,h_e$ in the above estimates.

Finally, one gets that
\begin{equation}\label{2.028}
\begin{aligned}
J_2 \leq &C(C(u_e,u_b,h_e)+E_{u,h})E_{u,h}\\
&\cdot \Vert (u+u_e\phi(y)+u_b(1-\phi(y)))^{\frac{1}{2}}\langle y\rangle^{l+k}D^\alpha (u,h) \Vert_{L^2_y}\\
\leq &C(C(u_e,u_b,h_e)+E_{u,h})E_{u,h}^2.
\end{aligned}
\end{equation}

\textbf{\emph{Estimate of $J_3$.}}

For $J_3$, we have
\begin{equation}\label{2.029}
\begin{aligned}
J_3 \leq &\Vert \langle y\rangle^{l+k}I^3_1\Vert_{L^2_y}\Vert \langle y\rangle^{l+k}D^\alpha u\Vert_{L^2_y}+\Vert \langle y\rangle^{l+k}I^3_2\Vert_{L^2_y}\Vert \langle y\rangle^{l+k}D^\alpha h\Vert_{L^2_y}.
\end{aligned}
\end{equation}
It is left to estimate the terms $\Vert \langle y\rangle^{l+k}I^3_i\Vert_{L^2_y}\ (i=1,2)$. And it suffices to give the weighted estimate on $I^3_1$, the estimate on the other term $I^3_2$ can be derived by the similar arguments.

Since
\begin{equation}\label{2.030}
\begin{aligned}
D^\alpha&(gh_e\phi'-v(u_e-u_b)\phi')\\
&=\sum_{\tilde{\alpha}\leq \alpha}\left(h_eD^{\tilde{\alpha}}g D^{\alpha-\tilde{\alpha}}\phi' -(u_e-u_b)D^{\tilde{\alpha}}vD^{\alpha-\tilde{\alpha}}\phi'\right),
\end{aligned}
\end{equation}
notice that $\alpha=(\beta, k), |\alpha|\leq m, \beta\leq m-1$, we have
\begin{equation}\label{2.031}
\begin{aligned}
\Vert \langle y\rangle^{l+k}D^\alpha(gh_e\phi'-v(u_e-u_b)\phi')\Vert_{L^2_y}\leq CC(u_e,u_b,h_e)E_{u,h}.
\end{aligned}
\end{equation}
Furthermore, one gets that
\begin{equation}\label{2.031}
\begin{aligned}
J_3 \lesssim E_{u,h}^2.
\end{aligned}
\end{equation}

Collecting the above estimates, integrating it with respect to $x$ variable and summing over $|\alpha|\leq m$, we have

\begin{equation}\label{estimate}
\begin{aligned}
&\sum_{\alpha\in\{\alpha=(\beta,k):|\alpha| \leq m,\beta\leq m-1\}}\bigg(s(x)+\nu\int_0^x \Vert \partial_y D^\alpha u\Vert_{L^2_l}^2+\kappa\int_0^x \Vert \partial_y D^\alpha h \Vert_{L^2_l}^2 \bigg)\\
\leq &C\delta_1 \int_0^x \Vert \partial_y (u,h)\Vert_{H^m_0}^2\\
&+C\delta_1^{-1}\int_0^x E_{u,h}^2 \left(1+E_{u,h}^2\right)+\int_0^x C(u_b,u_e,h_e)\\
&+\int_0^\infty (h+h_e\phi(y))\langle y\rangle^{2(l+k)}D^\alpha u D^\alpha h +s(0),
\end{aligned}
\end{equation}
where
$$s(x)=\Vert (u+u_e\phi(y)+u_b(1-\phi(y)))^{\frac{1}{2}}\langle y\rangle^{l+k}D^\alpha (u,h) \Vert_{L^2_y}^2.$$

Recall that
$$u(x,y)+u_e\phi(y)+u_b(1-\phi(y))>h(x,y)+h_e\phi(y)>0,$$
it follows that
$$\int_0^\infty (h+h_e\phi(y))\langle y\rangle^{2(l+k)}D^\alpha u D^\alpha h \leq \frac{1}{2} s(x),$$
which leads to
\begin{equation}\label{estimate1end}
\begin{aligned}
&\sum_{\alpha\in\{\alpha=(\beta,k):|\alpha|\leq m, \beta\leq m-1\}}\bigg(s(x)+\nu\int_0^x \Vert \partial_y D^\alpha u\Vert_{L^2_l}^2+\kappa\int_0^x \Vert \partial_y D^\alpha h \Vert_{L^2_l}^2 \bigg)\\
\leq &C\delta_1 \int_0^x \Vert \partial_y (u,h)\Vert_{H^m_0}^2+C\delta_1^{-1}\int_0^x E_{u,h}^2 \left(1+E_{u,h}^2\right)+\int_0^x C(u_b,u_e,h_e)+s(0).
\end{aligned}
\end{equation}
The proof is completed.
\end{proof}

\textbf{Step 2: Estimate for $\partial_x^\beta(u,h), \beta =m.$}

It is well-known that the essential difficulty to solve the MHD Prandtl boundary layer equations in finite regularity function spaces is to deal with the terms of loss of derivative in the tangential variable $x$. Precisely, the two functions $v=-\partial_y^{-1}\partial_x u$ and $g=-\partial_y^{-1}\partial_x h$ will cause the loss of $x$-derivative, which prevents us from applying the standard energy methods to reach the closure of energy estimates. Let us explain the main idea used to overcome the difficulty here.

First, by using the divergence free conditions, the equation of $h$ (\ref{newequation}) can be rewritten as
\begin{align*}
\partial_y \left[v(h+h_e\phi)-g(u+u_e\phi+u_b(1-\phi))\right]-\kappa \partial_y^2 h=\kappa \phi''(y)h_e,
\end{align*}
Integrating the above equation from $[0, y]$ and using the boundary conditions $v|_{y=0}=g|_{y=0}=\partial_yh|_{y=0}=\phi'|_{y=0}=0$, we have
\begin{align*}
v(h+h_e\phi)-g(u+u_e\phi+u_b(1-\phi))-\kappa \partial_y h=\kappa \phi'(y)h_e.
\end{align*}
Since $\partial_x h+\partial_yg=0$, then there exists a stream function $\psi$, such that
\begin{align*}
h=\partial_y \psi,\  g=-\partial_x \psi,\  \psi|_{y=0}=0.
\end{align*}
Then, the equation of $\psi$ reads as
\begin{align}
\label{SE}
\big[\big(u+u_e\phi(y)+u_b(1-\phi(y))\big)\partial_x +v\partial_y\big]\psi+vh_e\phi-\kappa \partial_y^2\psi=\kappa h_e \phi'(y).
\end{align}
Applying the operator $\partial_x^\beta, \beta=m$ on (\ref{SE}) to yield that
\begin{equation}\label{phi}
\begin{aligned}
\big[\big(u+u_e\phi(y)+u_b(1-\phi(y))\big)\partial_x +v\partial_y\big]\partial_x^\beta \psi+(h+h_e\phi(y))\partial_x^\beta v-\kappa\partial_y^2 \partial_x^\beta \psi\\
=-[\partial_x^\beta,\big(u+u_e\phi(y)+u_b(1-\phi(y))\big)\partial_x]\psi-\sum\limits_{0 <\tilde{\beta}<\beta}C_\beta^{\tilde{\beta}}(\partial_x^{\tilde{\beta}}v \partial_x^{\beta-\tilde{\beta}}\partial_y \psi)=:R^\beta_\psi.
\end{aligned}
\end{equation}

And applying $\partial_x^\beta, \beta=m$ on the equations in (\ref{newequation}), we have
\begin{equation}\label{1}
\left \{
\begin{aligned}
&\big[\big(u+u_e\phi(y)+u_b(1-\phi(y))\big)\partial_x +v\partial_y\big]\partial_x^\beta u-\big[(h+h_e\phi(y))\partial_x+g\partial_y\big]\partial_x^\beta h\\
&\quad +(\partial_y u+(u_e-u_b)\phi'(y))\partial_x^\beta v-(\partial_y h+h_e\phi'(y))\partial_x^\beta g-\nu \partial_y^2\partial_x^\beta u\\
&=-[\partial_x^\beta,\big(u+u_e\phi(y)+u_b(1-\phi(y))\big)\partial_x]u +[\partial_x^\beta,(h+h_e\phi(y))\partial_x]h\\
&\quad-\sum\limits_{0 <\tilde{\beta}<\beta}C_\beta^{\tilde{\beta}}(\partial_x^{\tilde{\beta}}v \partial_x^{\beta-\tilde{\beta}}\partial_y u-\partial_x^{\tilde{\beta}}g\partial_y^{\beta-\tilde{\beta}}h)=:R^\beta_u,\\
&\big[
\big(u+u_e\phi(y)+u_b(1-\phi(y))\big)\partial_x +v\partial_y\big]\partial_x^\beta h-\big[(h+h_e\phi(y))\partial_x+g\partial_y\big]\partial_x^\beta u\\
&\quad+(\partial_y h+h_e\phi'(y))\partial_x^\beta v-(\partial_y u+(u_e-u_b)\phi'(y))\partial_x^\beta g-\kappa \partial_y^2\partial_x^\beta h\\
&=-[\partial_x^\beta,\big(u+u_e\phi(y)+u_b(1-\phi(y))\big)\partial_x]h +[\partial_x^\beta,(h+h_e\phi(y))\partial_x]u\\
&\quad-\sum\limits_{0 <\tilde{\beta}<\beta}C_\beta^{\tilde{\beta}}(\partial_x^{\tilde{\beta}}v \partial_x^{\beta-\tilde{\beta}}\partial_y h-\partial_x^{\tilde{\beta}}g\partial_y^{\beta-\tilde{\beta}}u)=:R^\beta_h\\
&(u,v,\partial_y h,g)|_{y=0}=(0,0,0,0),\\
&(u,h) \to (0,0) \ \ \mathrm{as} \ \ y\to \infty,\\
&(u,h)|_{x=0}=(\overline{u}_0(y)+(u_e-u_b)(1-\phi(y)), \overline{h}_0(y)+h_e(1-\phi(y)).
\end{aligned}
\right.
\end{equation}
Recall that for $(x,y)\in [0,L]\times [0,\infty)$,
$$h+h_e\phi \geq \frac{\vartheta_0}{2}>0,$$
then we define
$$\eta_1=\frac{\partial_y u+(u_e-u_b)\phi'}{h+h_e\phi},\ \eta_2=\frac{\partial_y h+h_e\phi'}{h+h_e\phi},$$
and introduce
$$u_\beta=\partial_x^\beta u-\eta_1\partial_x^\beta\psi, \ h_\beta=\partial_x^\beta h-\eta_2 \partial_x^\beta\psi.$$
Then, the equations of $(u_\beta,h_\beta)$ can be derived from (\ref{phi}) and (\ref{1}), which are described as follows.
\begin{equation}\label{2}
\left \{
\begin{array}{lll}
\big[\big(u+u_e\phi(y)+u_b(1-\phi(y))\big)\partial_x +v\partial_y\big]u_\beta-\big[(h+h_e\phi(y))\partial_x+g\partial_y\big]h_\beta\\
\quad -\nu \partial_y^2u_\beta+(\kappa-\nu)\eta_1\partial_yh_\beta
=R^\beta_1,\\
\big[
\big(u+u_e\phi(y)+u_b(1-\phi(y))\big)\partial_x +v\partial_y\big]h^\beta-\big[(h+h_e\phi(y))\partial_x+g\partial_y\big]u_\beta\\
\quad-\kappa \partial_y^2 h_\beta
=R^\beta_2,
\end{array}
\right.
\end{equation}
where
\begin{equation}\label{3}
\left \{
\begin{aligned}
R^\beta_1=&R^\beta_u-\eta_1 R^\beta_\psi-\partial_x^\beta\psi \zeta_1+2\nu \partial_y \eta_1 \partial_x^\beta h\\
&+g\eta_2\partial_x^\beta h+(\kappa-\nu)\eta_1\eta_2\partial_x^\beta  h,\\
R^\beta_2=&R^\beta_h-\eta_2 R^\beta_\psi-\partial_x^\beta\psi \zeta_2+2\kappa  \eta_2\partial_y \partial_x^\beta h+g\eta_1 \partial_x^\beta h,
\end{aligned}
\right.
\end{equation}
in which
\begin{equation}\label{3.00}
\left \{
\begin{array}{lll}
\zeta_1=&\big[\big(u+u_e\phi(y)+u_b(1-\phi(y))\big)\partial_x +v\partial_y\big]\eta_1\\
&-\big[(h+h_e\phi(y))\partial_x+g\partial_y\big]\eta_2-\nu\partial_y^2 \eta_1+(\kappa-\nu)\eta_1\partial_y\eta_2,\\
\zeta_2=&\big[\big(u+u_e\phi(y)+u_b(1-\phi(y))\big)\partial_x +v\partial_y\big]\eta_2\\
&-\big[(h+h_e\phi(y))\partial_x+g\partial_y\big]\eta_1-\kappa\partial_y^2 \eta_2.
\end{array}
\right.
\end{equation}
The direct calculations yield that the boundary conditions and ``initial data" (the data at $x=0$) of $(u_\beta, h_\beta)$ can be written as follows.
\begin{equation}\label{3.01}
\left \{
\begin{aligned}
u_\beta|_{x=0}=&\partial_x^\beta u(0,y)-\frac{\partial_y u_0(y)+(u_e-u_b)\phi'}{h_0(y)+h_e\phi}\int_0^y\partial_x^\beta h(0,z)\mathrm{d}z\\
\triangleq &u_{\beta0}(y),\\
h_\beta|_{x=0}=&\partial_x^\beta h(0,y)-\frac{\partial_y h_0(y)+h_e\phi'}{h_0(y)+h_e\phi}\int_0^y\partial_x^\beta h(0,z)\mathrm{d}z\\
\triangleq &h_{\beta0}(y),\\
(u_\beta,\partial_y h_\beta)&|_{y=0}=(0,0).
\end{aligned}
\right.
\end{equation}

Therefore, we obtain the initial boundary value problem of $(u_{\beta},h_{\beta})$ as follows.
\begin{equation}\label{3.02}
\left \{
\begin{aligned}
\big[\big(u&+u_e\phi(y)+u_b(1-\phi(y))\big)\partial_x +v\partial_y\big]u_\beta-\big[(h+h_e\phi(y))\partial_x+g\partial_y\big]h_\beta\\
&-\nu \partial_y^2u_\beta+(\kappa-\nu)\eta_1\partial_yh_\beta
=R^\beta_1,\\
\big[
\big(u&+u_e\phi(y)+u_b(1-\phi(y))\big)\partial_x +v\partial_y\big]h^\beta-\big[(h+h_e\phi(y))\partial_x+g\partial_y\big]u_\beta\\
&-\kappa \partial_y^2 h_\beta
=R^\beta_2,\\
(u_{\beta},&h_{\beta})|_{x=0}=(u_{\beta 0},h_{\beta 0})(y),\\
(u_\beta,&\partial_y h_\beta)|_{y=0}=(0,0).
\end{aligned}
\right.
\end{equation}
Since $\psi=\partial_y^{-1}h$, we have
\begin{equation}\label{3.03}
\Vert \langle y\rangle^{-1}\partial_x^\beta \psi\Vert_{L^2_y} \lesssim \Vert \partial_x^\beta h\Vert_{L^2_y}.
\end{equation}
due to the Hardy-type inequality.  From the definitions of $\eta_i\ (i=1,2)$, it follows from the Hardy-type inequality and Sobolev embedding that for $\lambda \in \mathbb{R}$ and $i=1,2$,
\begin{equation}\label{3.04}
\begin{aligned}
\Vert &\langle y\rangle^\lambda \eta_i \Vert_{L^\infty_y} \lesssim \vartheta_0^{-1}\bigg(C(u_e,h_e,u_b)\\
&+\sum_{|\alpha|\leq 3}\Vert(u+u_e\phi(y)+u_b(1-\phi(y))^{\frac{1}{2}}\langle y\rangle^{\lambda-1}D^\alpha (u,h)\Vert_{L^2_y}\bigg),\\
\Vert &\langle y\rangle^\lambda \partial_y \eta_i \Vert_{L^\infty_y} \lesssim \vartheta_0^{-2}\bigg(C(u_e,h_e,u_b)\\
&+\sum_{|\alpha|\leq 4}\Vert(u+u_e\phi(y)+u_b(1-\phi(y))^{\frac{1}{2}}\langle y\rangle^{\lambda-1}D^\alpha (u,h)\Vert_{L^2_y}\bigg)^2,\\
\end{aligned}
\end{equation}
and
\begin{equation}\label{3.05}
\begin{aligned}
\Vert &\langle y\rangle^\lambda \zeta_i \Vert_{L^\infty_y} \lesssim \vartheta_0^{-3}\bigg(C(u_e,h_e,u_b)\\
&+\sum_{|\alpha|\leq 5}\Vert(u+u_e\phi(y)+u_b(1-\phi(y))^{\frac{1}{2}}\langle y\rangle^{\lambda-1}D^\alpha (u,h)\Vert_{L^2_y}\bigg)^3.
\end{aligned}
\end{equation}
Therefore,  for $\beta=m\geq 5, l\geq 0$, the following inequality holds true.
\begin{equation}\label{3.06}
\begin{aligned}
\Vert R^\beta_1\Vert_{L^2_l} \leq &\Vert R^\beta_u\Vert_{L^2_l}+\Vert \langle y \rangle^{l+1}\eta_1\Vert_{L^\infty_y}\Vert \langle y\rangle^{-1} R^\beta_\psi\Vert_{L^2_y}\\
&+\Vert \langle y\rangle^{l+1}\zeta_1\Vert_{L^\infty_y}\Vert \langle y\rangle^{-1}\partial_x^\beta\psi\Vert_{L^2_y}\\
 &+\bigg(\Vert 2\nu \partial_y \eta_1 +(\kappa-\nu)\eta_1\eta_2\Vert_{L^\infty_y}\\
 &+\Vert \langle y\rangle^{-1}g\Vert_{L^\infty_y}\Vert \langle y\rangle\eta_2\Vert_{L^\infty_y}\bigg)\Vert \langle y\rangle^l \partial_x^\beta  h\Vert_{L^2_y}\\
 \leq  & C\vartheta_0^{-3}\bigg(C(u_e,h_e.u_b)+E_{u,h}\bigg)^3E_{u,h}.
\end{aligned}
\end{equation}
Similarly,
\begin{equation}\label{3.07}
\begin{aligned}
\Vert R^\beta_2\Vert_{L^2_l} \leq C\vartheta_0^{-3}\bigg(C(u_e,h_e.u_b)+E_{u,h}\bigg)^3E_{u,h}.
\end{aligned}
\end{equation}
Now, it is ready to give the weighted estimates of $(u_\beta,h_\beta)$.
\begin{lemma}[Weighted estimates of $(u_\beta,h_\beta)$]\label{estimatestep2}
Under the assumptions in Proposition \ref{pro1}, for any $x\in [0,L]$, there holds that
\begin{equation}\label{estimatetan}
\begin{aligned}
&s_\beta(x)+\nu\int_0^x \Vert \partial_y  u_\beta \Vert_{L^2_l}^2+\kappa\int_0^x \Vert \partial_y h_\beta \Vert_{L^2_l}^2\\
\leq &C\vartheta_0^{-2}\int_0^x \left(C(u_e,h_e,u_b)+E_{u,h}\right)^2\cdot s_\beta(x)\\
&+C\vartheta_0^{-4}\int_0^x \left(C(u_e,h_e,u_b)+E_{u,h}\right)^4\cdot E_{u,h}^2+\int_0^x C(u_b,u_e,h_e)+s_\beta(0),
\end{aligned}
\end{equation}
where
$$s_\beta(x)=\Vert (u+u_e\phi(y)+u_b(1-\phi(y)))^{\frac{1}{2}}\langle y\rangle^{l}(u_\beta,h_\beta) \Vert_{L^2_y}^2$$
and
$$E_{u,h}^2=\sum_{|\alpha |\leq m}\Vert (u+u_e\phi(y)+u_b(1-\phi(y)))^{\frac{1}{2}}\langle y\rangle^{l+k}D^\alpha (u,h)\Vert_{L^2_y}^2.$$
\end{lemma}
\begin{proof}
Multiplying the first and second equations in (\ref{3.02}) by $\langle y\rangle^{2l}u_\beta$ and $\langle y\rangle^{2l}h_\beta$, respectively, and integrating by parts over $y \in [0,\infty)$, we have
\begin{equation}\label{3.08}
\begin{aligned}
\frac{1}{2}\frac{\mathrm{d}}{\mathrm{d}x}&s_\beta(x)+\nu\Vert \langle y\rangle^l \partial_y u_\beta\Vert_{L^2_y}^2+\kappa \Vert \langle y\rangle^l\partial_y h_\beta\Vert_{L^2_y}^2\\
=&l\int_0^\infty \langle y\rangle^{2l-2}yv\cdot (|u_\beta|^2+|h_\beta|^2)\mathrm{d}y\\
&+\int_0^\infty \partial_x\left(\langle y\rangle^{2l} (h+h_e\phi)u_\beta h_\beta\right)\mathrm{d}y-\int_0^\infty 2l\langle y\rangle^{2l-2}ygu_\beta h_\beta\mathrm{d}y\\
&+(\nu -\kappa)\int_0^\infty \langle y\rangle^{2l}(\eta_1\partial_y h_\beta\cdot u_\beta)\mathrm{d}y\\
&+\int_0^\infty \langle y\rangle^{2l}(u_\beta R^\beta_1+h_\beta R^\beta_2)\mathrm{d}y\\
&-l\int_0^\infty \langle y\rangle^{2l-2}2y(\nu u_\beta \partial_y u_\beta+\kappa h_\beta\partial_yh_\beta)\mathrm{d}y.
\end{aligned}
\end{equation}
First,
\begin{equation}\label{3.09}
\begin{aligned}
l\int_0^\infty & \langle y\rangle^{2l-2}yv\cdot (|u_\beta|^2+|h_\beta|^2)\mathrm{d}y-\int_0^\infty 2l\langle y\rangle^{2l-2}ygu_\beta h_\beta\mathrm{d}y\\
\leq & C \left(\left\Vert \frac{v}{1+y}\right\Vert_{L^\infty_y}+\left\Vert \frac{g}{1+y}\right\Vert_{L^\infty_y} \right)\Vert \langle y\rangle^l (u_\beta,h_\beta)\Vert_{L^2_y}^2\\
\leq &C \left(\Vert u_x\Vert_{L^\infty_y}+\Vert h_x\Vert_{L^\infty_y}\right)s_\beta(x)\lesssim E_{u,h}s_\beta(x).
\end{aligned}
\end{equation}
Notice that $u_\beta|_{y=0}=0$, then integrating by parts yields
\begin{equation}\label{3.010}
\begin{aligned}
(\nu -&\kappa)\int_0^\infty \langle y\rangle^{2l}(\eta_1\partial_y h_\beta\cdot u_\beta)\mathrm{d}y\\
=&-\nu \int_0^\infty h_\beta \partial_y \left(\langle y\rangle^{2l}\eta_1 u_\beta\right)\mathrm{d}y-\kappa \int_0^\infty \langle y\rangle^{2l}(\eta_1\partial_y h_\beta\cdot u_\beta)\mathrm{d}y\\
\leq &\frac{\nu}{4}\Vert \partial_y u_\beta\Vert_{L^2_l}^2+\frac{\kappa}{4}\Vert \partial_y h_\beta\Vert_{L^2_l}^2+C(1+\Vert \eta_1\Vert_{L^\infty}^2+\Vert \partial_y \eta_1\Vert_{L^\infty})s_\beta(x)\\
\leq &\frac{\nu}{4}\Vert \partial_y u_\beta\Vert_{L^2_l}^2+\frac{\kappa}{4}\Vert \partial_y h_\beta\Vert_{L^2_l}^2+C\vartheta_0^{-2}(C(u_e,h_e,u_b)+E_{u,h})^2s_\beta.
\end{aligned}
\end{equation}
By Cauchy-Schwartz inequality, (\ref{3.06}) and (\ref{3.07}), it follows that
\begin{equation}\label{3.011}
\begin{aligned}
\int_0^\infty \langle y\rangle^{2l}(u_\beta R^\beta_1+h_\beta R^\beta_2)\mathrm{d}y\leq &\Vert R^\beta_1\Vert_{L^2_l}\Vert u_\beta\Vert_{L^2_l}+\Vert R^\beta_2\Vert_{L^2_l}\Vert h_\beta\Vert_{L^2_l}\\
\leq &C\vartheta_0^{-2}\left(C(u_e,h_e,u_b)+E_{u,h}\right)^2\cdot s_\beta\\
&+C\vartheta_0^{-4}\left(C(u_e,h_e,u_b)+E_{u,h}\right)^4E_{u,h}^2.
\end{aligned}
\end{equation}
Similarly,
\begin{equation}\label{3.012}
\begin{aligned}
\bigg|-l\int_0^\infty &\langle y\rangle^{2l-2}2y(\nu u_\beta \partial_y u_\beta+\kappa h_\beta\partial_yh_\beta)\mathrm{d}y\bigg|\\
&\leq \frac{\nu}{4}\Vert \partial_y u_\beta\Vert_{L^2_l}^2+\frac{\kappa}{4}\Vert \partial_y h_\beta\Vert_{L^2_l}^2+Cs_\beta(x).
\end{aligned}
\end{equation}
Plugging (\ref{3.09})-(\ref{3.012}) into (\ref{3.08}), we have
\begin{equation}\label{3.013}
\begin{aligned}
\frac{\mathrm{d}}{\mathrm{d}x}&s_\beta(x)+\nu\Vert \langle y\rangle^l \partial_y u_\beta\Vert_{L^2_y}^2+\kappa \Vert \langle y\rangle^l\partial_y h_\beta\Vert_{L^2_y}^2\\
\leq &\int_0^\infty \partial_x\left(\langle y\rangle^{2l} (h+h_e\phi)u_\beta h_\beta\right)\mathrm{d}y\\
&+C\vartheta_0^{-2}\left(C(u_e,h_e,u_b)+E_{u,h}\right)^2\cdot s_\beta\\
&+C\vartheta_0^{-4}\left(C(u_e,h_e,u_b)+E_{u,h}\right)^4E_{u,h}^2.
\end{aligned}
\end{equation}
From the \emph{a priori assumption}, which will be verified latter,
$$u(x,y)+u_e\phi(y)+u_b(1-\phi(y))>h(x,y)+h_e\phi(y) >0,$$
one has
$$\int_0^\infty (h+h_e\phi(y))\langle y\rangle^{2l}u_\beta h_\beta \leq \frac{1}{2} s_\beta(x).$$
Then, integrating (\ref{3.013}) with respect to $x$ variable yields that
\begin{equation}\label{3.014}
\begin{aligned}
&s_\beta(x)+\nu\int_0^x \Vert \partial_y  u_\beta \Vert_{L^2_l}^2+\kappa\int_0^x \Vert \partial_y h_\beta \Vert_{L^2_l}^2 \\
\leq &C\vartheta_0^{-2}\int_0^x \left(C(u_e,h_e,u_b)+E_{u,h}\right)^2\cdot \sum_{\beta=m}s_\beta(x)\\
&+C\vartheta_0^{-4}\int_0^x \left(C(u_e,h_e,u_b)+E_{u,h}\right)^4\cdot E_{u,h}^2+\int_0^x C(u_b,u_e,h_e)+s_\beta(0).
\end{aligned}
\end{equation}
This completes the proof.
\end{proof}

Then, it is left to show the equivalence in the weighed Sobolev space $L^2_l$ between $\partial_x^\beta(u,h)$ and $(u_\beta,h_\beta)$. Which  is summarized in the following lemma.
\begin{lemma}\label{equivalence}
Under the assumptions in Proposition \ref{pro1}, we have
\begin{equation}\label{3.015}
\begin{aligned}
M(x)^{-1}&\Vert(u+u_e\phi(y)+u_b(1-\phi(y)))^{\frac{1}{2}}\langle y\rangle^{l} \partial_x^\beta (u,h)\Vert_{L^2_y} \\
\leq &\Vert (u+u_e\phi(y)+u_b(1-\phi(y)))^{\frac{1}{2}}\langle y\rangle^{l}(u_\beta,h_\beta) \Vert_{L^2_y}\\
 \leq &M(x) \Vert (u+u_e\phi(y)+u_b(1-\phi(y)))^{\frac{1}{2}}\langle y\rangle^{l} \partial_x^\beta (u,h)\Vert_{L^2_y},
\end{aligned}
\end{equation}
and
\begin{equation}\label{3.0151}
\begin{aligned}
&\Vert\partial_y \partial_x^\beta (u,h) \Vert_{L^2_l} \leq  \Vert \partial_y (u_\beta,h_\beta)\Vert_{L^2_l}+M(x)\Vert h_\beta\Vert_{L^2_l},
\end{aligned}
\end{equation}
where
\begin{equation}\label{3.016}
\begin{aligned}
M(x):=\vartheta_0^{-1}\left(C(u_e,h_e,u_b)+\Vert \langle y\rangle^{l+1}\partial_y (u,h)\Vert_{L^\infty}+\Vert \langle y\rangle^{l+1}\partial_y^2 (u,h)\Vert_{L^\infty}\right).
\end{aligned}
\end{equation}
\end{lemma}
\begin{proof}
This lemma can be proved directly by the definitions of $u_\beta,h_\beta$, the main idea to prove Lemma \ref{equivalence} is similar as what was proposed in \cite{CLiu2}. We omitted it here for simplicity of presentation.
\end{proof}

\textbf{Step 3: Completeness of a priori estimates.}
Now we will prove the Proposition \ref{pro1}. First of all, based on the {\it a priori assumption}
$$\Vert \langle y\rangle^{l+1}\partial_y  (u,h)\Vert_{L^\infty} \leq \sigma_0, $$
$$\Vert \langle y\rangle^{l+1}\partial_y^2  (u,h)\Vert_{L^\infty} \leq \vartheta_0^{-1}, $$
then if we choose $\vartheta_0>0$ small enough, we have
\begin{equation}\nonumber
\begin{aligned}
&\Vert \langle y\rangle^{l+1}\eta_i\Vert_{L^\infty} \leq 2\vartheta_0^{-1}\sigma_0\leq 2\vartheta_0^{-2},\\
&M(x)\leq \vartheta_0^{-1}(C(u_e,h_e,u_b)+\vartheta_0^{-1}+\sigma_0)\leq 5 \vartheta_0^{-2}.
\end{aligned}
\end{equation}
Therefore, it follows that
\begin{equation}\label{3.017}
\begin{aligned}
E_{u,h}^2=&\sum_{\alpha\in\{\alpha=(\beta,k):|\alpha|\leq m, \beta\leq m-1\}} s(x)\\
&+\Vert (u+u_e\phi(y)+u_b(1-\phi(y)))^{\frac{1}{2}}\langle y\rangle^{l}\partial_x^m (u,h) \Vert_{L^2_y}^2\\
\leq &\sum_{\alpha\in\{\alpha=(\beta,k):|\alpha|\leq m, \beta\leq m-1\}} s(x)+25\vartheta_0^{-4}s_m(x),
\end{aligned}
\end{equation}
and
\begin{equation}\label{3.018}
\begin{aligned}
\Vert \partial_y (u,h)\Vert_{H^m_l}^2 \leq  &\sum_{\alpha\in\{\alpha=(\beta,k):|\alpha|\leq m, \beta\leq m-1\}} \Vert \partial_y D^\alpha (u,h)\Vert_{L^2_l}^2\\
&+2\Vert \partial_y (u_m,h_m)\Vert_{L^2_l}^2+50\vartheta_0^{-4}\Vert h_m\Vert_{L^2_l}^2.
\end{aligned}
\end{equation}

With the above estimates in hand, we can derive the desired a priori estimates of $(u,h)$ for the boundary layer problem (\ref{newequation}). Thanks to Lemma \ref{estimate1} and \ref{estimatestep2}, it follows from (\ref{3.017})-(\ref{3.018}) that for $m\geq 5$, small enough $\delta_1>0$ and any $x\in[0,L]$,
\begin{equation}\label{3.019}
\begin{aligned}
&\sum_{\alpha\in\{\alpha=(\beta,k):|\alpha|\leq m, \beta\leq m-1\}} s(x)+25\vartheta_0^{-4} s_m(x)\\
&+\int_0^x\big(\sum_{\alpha\in\{\alpha=(\beta,k):|\alpha|\leq m, \beta\leq m-1\}} \Vert \partial_y D^\alpha (u,h)\Vert_{L^2_l}^2+25\vartheta_0^{-4} \Vert \partial_y (u_m,h_m)\Vert_{L^2_l}^2\big)\\
\leq &\sum_{\alpha\in\{\alpha=(\beta,k):|\alpha|\leq m, \beta\leq m-1\}} s(0)+25\vartheta_0^{-4} s_m(0)+\int_0^x C(u_e,u_b,h_e,\vartheta_0,\sigma_0) \\
&+C\vartheta_0^{-8} \int_0^x\bigg(\sum_{\alpha\in\{\alpha=(\beta,k):|\alpha|\leq m, \beta\leq m-1\}} s(x)+25\vartheta_0^{-4} s_m(x)\bigg)^3.
\end{aligned}
\end{equation}
Denote
$$F_0:=\sum_{\alpha\in\{\alpha=(\beta,k):|\alpha|\leq m, \beta\leq m-1\}} s(0)+25\vartheta_0^{-4} s_m(0)$$
and
$$F:=C(u_e,u_b,h_e,\vartheta_0,\sigma_0),$$
then it follows from the Gronwall inequality that
\begin{equation}\label{3.020}
\begin{aligned}
&\sum_{\alpha\in\{\alpha=(\beta,k):|\alpha|\leq m, \beta\leq m-1\}} s(x)+25\vartheta_0^{-4} s_m(x) \\
&\leq \left(F_0+\int_0^x F\right)\left\{1-2C\vartheta_0^{-8}\left(F_0+\int_0^x F\right)^2x\right\}^{-1/2},
\end{aligned}
\end{equation}
which gives that
\begin{equation}\label{3.020}
\begin{aligned}
\sup\limits_{x \in [0,L] }E_{u.h}\leq \left(F_0+\int_0^x F\right)^{1/2}\left\{1-2C\vartheta_0^{-8}\left(F_0+\int_0^x F\right)^2x\right\}^{-1/4}.
\end{aligned}
\end{equation}
Moreover, under the assumptions in Proposition \ref{pro1}, we obtain
\begin{equation}\label{3.021}
\begin{aligned}
\sup\limits_{x \in [0,L] }&\Vert (u,h)\Vert_{H^m_l}\\
\leq C&\left(F_0+\int_0^x F\right)^{1/2}\left\{1-2C\vartheta_0^{-8}\left(F_0+\int_0^x F\right)^2x\right\}^{-1/4}.
\end{aligned}
\end{equation}
For $i=1,2$, by Newton-Leibnitz formula, we have
$$\langle y\rangle^{l+1}\partial_y^i (u,h)=\langle y\rangle^{l+1}\partial_y^i (u_0,h_0)+\int_0^x \langle y\rangle^{l+1}\partial_x \partial_y (u,h),$$
$$u-h=u_0-h_0+\int_0^x \partial_x (u-h),$$
and
$$(u,h)=(u_0,h_0)+\int_0^x \partial_x (u,h).$$
Therefore, we deduce from the Sobolev embedding and (\ref{3.021}) that
\begin{equation}\label{3.022}
\begin{aligned}
\Vert &\langle y\rangle^{l+1}\partial_y^i (u,h)\Vert_{L^\infty}\\
\leq &\Vert \langle y\rangle^{l+1}\partial_y^i (u_0,h_0)\Vert_{L^\infty}+\int_0^x \Vert\langle y\rangle^{l+1}\partial_x \partial_y (u,h)\Vert_{L^\infty}\\
\leq &\Vert \langle y\rangle^{l+1}\partial_y^i (u_0,h_0)\Vert_{L^\infty}+Cx\sup\limits_{x\in [0,L]}\Vert (u,h)\Vert_{H^5_l}\\
\leq  &\Vert \langle y\rangle^{l+1}\partial_y^i (u_0,h_0)\Vert_{L^\infty}\\
&+Cx\left(F_0+\int_0^x F\right)^{1/2}\left\{1-2C\vartheta_0^{-8}\left(F_0+\int_0^x F\right)^2x\right\}^{-1/4}.
\end{aligned}
\end{equation}
Similarly,
\begin{equation}\label{3.023}
\begin{aligned}
h(x,y) \geq &h_0(x,y)-\int_0^x \Vert \partial_x h\Vert_{L^\infty}\\
\geq &h_0-Cx\cdot \sup\limits_{x\in[0,L]}\Vert h\Vert_{H^3_l}\\
\geq &h_0-Cx\left(F_0+\int_0^x F\right)^{1/2}\left\{1-2C\vartheta_0^{-8}\left(F_0+\int_0^x F\right)^2x\right\}^{-1/4}
\end{aligned}
\end{equation}
and
\begin{equation}\label{3.025}
\begin{aligned}
u-h \geq& (u_0-h_0)\\
&-2Cx\left(F_0+\int_0^x F\right)^{1/2}\left\{1-2C\vartheta_0^{-8}\left(F_0+\int_0^x F\right)^2x\right\}^{-1/4}.
\end{aligned}
\end{equation}
\begin{proof}[Proof of the Proposition \ref{pro1}]
Notice that
$$\int_0^x F\leq C(u_e,u_b,h_e,\vartheta_0,\sigma_0)x,$$
and $F_0$ is a polynomial of $\Vert (u_0,h_0)\Vert_{H^m_l}$. Precisely,
$$F_0 \leq C\vartheta_0^{-8}P(C(u_e,u_b,h_e)+\Vert (u_0,h_0)\Vert_{H^m_l}).$$
Putting the above estimates into (\ref{3.021})-(\ref{3.025}), the Proposition \ref{pro1} follows.
\end{proof}
We are in a position to prove the Proposition \ref{wellposedness}.
\begin{proof}[Proof of Proposition \ref{wellposedness}]
The existence of solutions to the problem (\ref{newequation}) can be achieved by the classical Picard iteration scheme and the fixed point theorem. And the uniform a priori energy estimates of solutions to the linear problem are derived similarly as those in (\ref{3.021})-(\ref{3.025}) for the nonlinear problem, which also guarantee the lifespan of the solution sequences to the linear problem will not shrink to zero. We refer to \cite{CLiu2,Masmoudi} for more details for instance.  Here we omit it here for simplicity of presentation.
\end{proof}

\section{Proof of Lemma \ref{standardenergy}}\label{ap2}
Now we will prove the Lemma \ref{standardenergy} in this Appendix.
\begin{proof}[Proof of Lemma \ref{standardenergy}]
Multiplying the equations in (\ref{5.02}) by $u^\epsilon,\epsilon v^\epsilon,h^\epsilon,\epsilon g^\epsilon$ respectively, integrating by parts on $[0,L]\times [0,+\infty)$ and adding all resulting equalities together, we obtain that
\begin{equation}\label{5.05}
\begin{aligned}
&\int_{x=L}u_s\frac{|u^\epsilon|^2+\epsilon|v^\epsilon|^2}{2}+\int_{x=L}u_s\frac{|h^\epsilon|^2+\epsilon|g^\epsilon|^2}{2}\\
&+\nu\iint \left[2\epsilon |u^\epsilon_x|^2+(u^\epsilon_y+\epsilon v^\epsilon_x)^2+2\epsilon |v^\epsilon_y|^2\right]\\
&+\kappa \iint \left[2\epsilon |h^\epsilon_x|^2+(h^\epsilon_y+\epsilon g^\epsilon_x)^2+2\epsilon |g^\epsilon_y|^2\right]\\
=&-\bigg(\iint \partial_x u_s |u^\epsilon|^2+\iint v^\epsilon\partial_y u_s u^\epsilon-\iint h_s \partial_x h^\epsilon u^\epsilon\\
&-\iint g_s\partial_y h^\epsilon u^\epsilon-\iint g^\epsilon\partial_y h_s u^\epsilon+\iint u^\epsilon \partial_x v_s \epsilon v^\epsilon\\
&+\epsilon\iint\partial_y v_s |v^\epsilon|^2-\iint h_s \partial_x g^\epsilon \epsilon v^\epsilon-\iint h^\epsilon \partial_x g_s \epsilon v^\epsilon\\
&+\iint g_s \partial_y g^\epsilon \epsilon v^\epsilon+\iint v^\epsilon \partial_yh_s h^\epsilon-\iint h_s\partial_x u^\epsilon h^\epsilon\\
&-\iint \partial_x u_s |h^\epsilon|^2-\iint g_s \partial_y u^\epsilon h^\epsilon-\iint g^\epsilon \partial_y u_s h^\epsilon\\
&+\iint u^\epsilon \partial_x g_s \epsilon g^\epsilon-\iint h_s \partial_x v^\epsilon \epsilon g^\epsilon-\iint h^\epsilon \partial_x v_s \epsilon g^\epsilon\\
&-\iint g_s \partial_y v^\epsilon \epsilon g^\epsilon-\iint \epsilon |g^\epsilon|^2\partial_yv_s\bigg)\\
&+\iint (f_1 u^\epsilon+f_3h^\epsilon)+\epsilon\iint (f_2 v^\epsilon+f_4g^\epsilon).
\end{aligned}
\end{equation}
Notice that
\begin{equation}\label{5.06}
\begin{aligned}
&\iint \left[2\epsilon |u^\epsilon_x|^2+(u^\epsilon_y+\epsilon v^\epsilon_x)^2+2\epsilon |v^\epsilon_y|^2\right] \\
&\geq \frac{1}{2}\Vert \nabla_\epsilon u^\epsilon\Vert_{L^2}^2-2\epsilon\Vert \nabla_\epsilon v^\epsilon\Vert_{L^2}^2,
\end{aligned}
\end{equation}
and
\begin{equation}\label{5.07}
\begin{aligned}
&\iint \left[2\epsilon |h^\epsilon_x|^2+(h^\epsilon_y+\epsilon g^\epsilon_x)^2+2\epsilon |g^\epsilon_y|^2\right] \\
&\geq \frac{1}{2}\Vert \nabla_\epsilon h^\epsilon\Vert_{L^2}^2-2\epsilon\Vert \nabla_\epsilon g^\epsilon\Vert_{L^2}^2.
\end{aligned}
\end{equation}

Since $(u^\epsilon,v^\epsilon,h^\epsilon,g^\epsilon)|_{x=0}=(0,0,0,0)$, the following Poincar\'{e} type inequalities hold.
$$\Vert u^\epsilon \Vert_{L^2} \leq L \Vert u^\epsilon_x\Vert_{L^2}=L \Vert v^\epsilon_y\Vert_{L^2}, \ \Vert h^\epsilon \Vert_{L^2} \leq L \Vert h^\epsilon_x\Vert_{L^2}=L \Vert g^\epsilon_y\Vert_{L^2},$$
$$\Vert v^\epsilon \Vert_{L^2} \leq L \Vert v^\epsilon_x\Vert_{L^2}, \ \Vert g^\epsilon \Vert_{L^2} \leq L \Vert g^\epsilon_x\Vert_{L^2}.$$
By using $(v^\epsilon,g^\epsilon)|_{y=0}=(0,0)$ again, one has
$$v^\epsilon=\int_0^y v^\epsilon_y \leq \sqrt{y}\left(\int_0^y |v^\epsilon_y|^2\right)^{1/2}, \ \ g^\epsilon=\int_0^y g^\epsilon_y \leq \sqrt{y}\left(\int_0^y |g^\epsilon_y|^2\right)^{1/2}.$$
Therefore, we obtain that
\begin{equation}\label{5.08}
\begin{aligned}
&\iint (f_1 u^\epsilon+f_3h^\epsilon)+\epsilon\iint (f_2 v^\epsilon+f_4g^\epsilon)\\
&\leq  \Vert u^\epsilon \Vert_{L^2}\Vert f_1\Vert_{L^2}+\Vert h^\epsilon \Vert_{L^2}\Vert f_3\Vert_{L^2}+\epsilon\left(\Vert v^\epsilon \Vert_{L^2}\Vert f_2\Vert_{L^2}+\Vert g^\epsilon \Vert_{L^2}\Vert f_4\Vert_{L^2}\right)\\
&\leq L^2\Vert \nabla_\epsilon (v^\epsilon,g^\epsilon)\Vert_{L^2}^2+\Vert (f_1,f_3)\Vert_{L^2}^2+\epsilon \Vert(f_2,f_4)\Vert_{L^2}^2.
\end{aligned}
\end{equation}
Moreover, one has the following estimates.
\begin{equation}\label{5.09}
\begin{aligned}
\iint \partial_x u_s |u^\epsilon|^2 &\leq \Vert\partial_x u_s\Vert_{L^\infty} \Vert u^\epsilon\Vert_{L^2}^2\\
&\leq L^2 \Vert\partial_x u_s\Vert_{L^\infty}\Vert u^\epsilon_x \Vert_{L^2}^2= L^2\Vert\partial_x u_s\Vert_{L^\infty}\Vert v^\epsilon_y \Vert_{L^2}^2,
\end{aligned}
\end{equation}
and
\begin{equation}\label{5.010}
\begin{aligned}
\iint v^\epsilon \partial_y u_s u^\epsilon &\leq L \Vert v^\epsilon_y\Vert_{L^2}\Vert y\partial_y u_s\Vert_{L^\infty}\Vert u^\epsilon_x\Vert_{L^2}\\
&= L \Vert y\partial_y u_s\Vert_{L^\infty}\Vert v^\epsilon_y\Vert_{L^2}^2,
\end{aligned}
\end{equation}
where $u_x^\epsilon+v^\epsilon_y=0$ is used in the last equality. Similarly,
\begin{equation}\label{5.011}
\begin{aligned}
-\iint h_s \partial_x h^\epsilon u^\epsilon &\leq \Vert h_s \Vert_{L^\infty}\Vert u^\epsilon\Vert_{L^2}\Vert \partial_x h^\epsilon\Vert_{L^2}\\
&\leq L\Vert h_s \Vert_{L^\infty}\Vert u^\epsilon_x\Vert_{L^2}\Vert g^\epsilon_y\Vert_{L^2}=L\Vert h_s \Vert_{L^\infty}\Vert v^\epsilon_y\Vert_{L^2}\Vert g^\epsilon_y\Vert_{L^2},
\end{aligned}
\end{equation}
\begin{equation}\label{5.012}
\begin{aligned}
-\iint g_s \partial_y h^\epsilon u^\epsilon &\leq \Vert g_s \Vert_{L^\infty}\Vert u^\epsilon\Vert_{L^2}\Vert \partial_y h^\epsilon\Vert_{L^2}\\
&\leq L\Vert g_s \Vert_{L^\infty}\Vert v^\epsilon_y\Vert_{L^2}\Vert h^\epsilon_y\Vert_{L^2},
\end{aligned}
\end{equation}
\begin{equation}\label{5.013}
\begin{aligned}
-\iint g^\epsilon \partial_y h_s u^\epsilon &\leq L \Vert y\partial_y h_s \Vert_{L^\infty}\Vert u^\epsilon_x \Vert_{L^2}\Vert g^\epsilon_y\Vert_{L^2}\\
&=L \Vert y\partial_y h_s \Vert_{L^\infty}\Vert v^\epsilon_y  \Vert_{L^2}\Vert g^\epsilon_y\Vert_{L^2},
\end{aligned}
\end{equation}
\begin{equation}\label{5.014}
\begin{aligned}
\iint u^\epsilon \partial_x v_s \epsilon v^\epsilon &\leq L\epsilon \Vert v_{sx} \Vert_{L^\infty}\Vert v^\epsilon_x \Vert_{L^2}\Vert v^\epsilon_y\Vert_{L^2},
\end{aligned}
\end{equation}
\begin{equation}\label{5.015}
\begin{aligned}
\iint g_s \partial_y g^\epsilon \epsilon v^\epsilon &\leq \epsilon \Vert g_s \Vert_{L^\infty}\Vert g^\epsilon_y \Vert_{L^2}\Vert v^\epsilon\Vert_{L^2}\\
&\leq \epsilon L\Vert g_s \Vert_{L^\infty}\Vert g^\epsilon_y \Vert_{L^2}\Vert v^\epsilon_x\Vert_{L^2},
\end{aligned}
\end{equation}
\begin{equation}\label{5.016}
\begin{aligned}
\iint v^\epsilon \partial_y h_s h^\epsilon &\leq  L \Vert y \partial_y h_{s} \Vert_{L^\infty}\Vert g^\epsilon_y \Vert_{L^2}\Vert v^\epsilon_y\Vert_{L^2},
\end{aligned}
\end{equation}
\begin{equation}\label{5.017}
\begin{aligned}
-\iint h_s \partial_x u^\epsilon h^\epsilon &\leq   \Vert h_{s} \Vert_{L^\infty}\Vert u^\epsilon_x \Vert_{L^2}\Vert h^\epsilon\Vert_{L^2}\\
&\leq L  \Vert h_{s} \Vert_{L^\infty}\Vert u^\epsilon_x \Vert_{L^2}\Vert h^\epsilon_x\Vert_{L^2}= L \Vert h_{s} \Vert_{L^\infty}\Vert v^\epsilon_y \Vert_{L^2}\Vert g^\epsilon_y\Vert_{L^2},
\end{aligned}
\end{equation}
\begin{equation}\label{5.018}
\begin{aligned}
-\iint \partial_x u_s |h^\epsilon|^2 &\leq   \Vert u_{sx}\Vert_{L^\infty}\Vert h^\epsilon\Vert_{L^2}^2\\
&\leq L^2 \Vert u_{sx}\Vert_{L^\infty}\Vert h^\epsilon_x\Vert_{L^2}^2=L^2 \Vert u_{sx}\Vert_{L^\infty}\Vert g^\epsilon_y\Vert_{L^2}^2,
\end{aligned}
\end{equation}
\begin{equation}\label{5.019}
\begin{aligned}
-\iint g_s \partial_y u^\epsilon h^\epsilon &\leq  L \Vert g_{s} \Vert_{L^\infty}\Vert h^\epsilon_x\Vert_{L^2}\Vert u^\epsilon_y \Vert_{L^2}= L \Vert g_{s} \Vert_{L^\infty}\Vert g^\epsilon_y\Vert_{L^2}\Vert u^\epsilon_y \Vert_{L^2},
\end{aligned}
\end{equation}
\begin{equation}\label{5.020}
\begin{aligned}
-\iint g^\epsilon \partial_y u_s h^\epsilon &\leq L\Vert g^\epsilon_y\Vert_{L^2} \Vert y\partial_y u_s\Vert_{L^\infty}\Vert h^\epsilon_x\Vert_{L^2}\leq
 L \Vert yu_{sy} \Vert_{L^\infty}\Vert g^\epsilon_y\Vert_{L^2}^2,
\end{aligned}
\end{equation}
\begin{equation}\label{5.021}
\begin{aligned}
\iint u^\epsilon \partial_x g_s \epsilon g^\epsilon &\leq  \epsilon \Vert g_{sx} \Vert_{L^\infty}\Vert u^\epsilon\Vert_{L^2}\Vert g^\epsilon\Vert_{L^2}\\
&\leq \epsilon L^2 \Vert g_{sx} \Vert_{L^\infty}\Vert u^\epsilon_x \Vert_{L^2}\Vert g^\epsilon_x\Vert_{L^2}= \epsilon L^2 \Vert g_{sx} \Vert_{L^\infty}\Vert v^\epsilon_y \Vert_{L^2}\Vert g^\epsilon_x\Vert_{L^2},
\end{aligned}
\end{equation}
\begin{equation}\label{5.022}
\begin{aligned}
\iint h_s  \partial_x v^\epsilon \epsilon g^\epsilon &\leq  \epsilon \Vert h_s \Vert_{L^\infty}\Vert g^\epsilon\Vert_{L^2}\Vert v^\epsilon_x\Vert_{L^2}\leq \epsilon L \Vert h_s \Vert_{L^\infty}\Vert v^\epsilon_x \Vert_{L^2}\Vert g^\epsilon_x\Vert_{L^2},
\end{aligned}
\end{equation}
\begin{equation}\label{5.023}
\begin{aligned}
\iint h^\epsilon  \partial_x v_s \epsilon g^\epsilon &\leq  L^2 \epsilon \Vert v_{sx} \Vert_{L^\infty}\Vert h^\epsilon_x\Vert_{L^2}\Vert g^\epsilon_x\Vert_{L^2}= \epsilon L^2 \Vert v_{sx} \Vert_{L^\infty}\Vert g^\epsilon_y \Vert_{L^2}\Vert g^\epsilon_x\Vert_{L^2},
\end{aligned}
\end{equation}
\begin{equation}\label{5.024}
\begin{aligned}
-\iint g_s  \partial_y v^\epsilon \epsilon g^\epsilon &\leq  \epsilon \Vert g_{s} \Vert_{L^\infty}\Vert v^\epsilon_y \Vert_{L^2}\Vert g^\epsilon\Vert_{L^2}\leq \epsilon L \Vert g_{s} \Vert_{L^\infty}\Vert v^\epsilon_y \Vert_{L^2}\Vert g^\epsilon_x\Vert_{L^2},
\end{aligned}
\end{equation}
\begin{equation}\label{5.025}
\begin{aligned}
-\iint \epsilon v_{sy}|g^\epsilon|^2 &\leq  \epsilon \Vert v_{sy} \Vert_{L^\infty}\Vert g^\epsilon\Vert_{L^2}^2\leq \epsilon L^2\Vert v_{sy} \Vert_{L^\infty}\Vert g^\epsilon_x\Vert_{L^2}^2.
\end{aligned}
\end{equation}

Collecting all estimates above together, using the estimates of $u_s,v_s,h_s,g_s$ and the facts that $L<<1$ and choosing $\epsilon<L$, the desired conclusion (\ref{5.03}) is obtained.
\end{proof}

\section{Proof of Lemma \ref{positivity}}\label{ap3}
It is left to give the proof of Lemma \ref{positivity}.
\begin{proof}[Proof of Lemma \ref{positivity}]
Recall that $u_s$ has a strictly positivity lower bound. We begin with (\ref{5.02})$_1\times \partial_y \left(\frac{v^\epsilon}{u_s}\right)-\epsilon$(\ref{5.02})$_2\times \partial_x \left(\frac{v^\epsilon}{u_s}\right)+$(\ref{5.02})$_3\times \partial_y \left(\frac{g^\epsilon}{u_s}\right)-\epsilon$(\ref{5.02})$_4\times \partial_x \left(\frac{g^\epsilon}{u_s}\right)$ to get that
\begin{equation}\label{5.027}
\begin{aligned}
&\iint \partial_y\left(\frac{v^\epsilon}{u_s}\right)\left(u_s\partial_x u^\epsilon+u^\epsilon\partial_x u_s+v_s\partial_y u^\epsilon+v^\epsilon\partial_y u_s+\partial_x p^\epsilon-\nu \Delta_\epsilon u^\epsilon\right)\\
&-\iint \partial_y\left(\frac{v^\epsilon}{u_s}\right)(h_s\partial_x h^\epsilon+h^\epsilon\partial_x h_s+g_s\partial_y h^\epsilon+g^\epsilon\partial_y h_s)\\
&-\iint \epsilon \partial_x\left(\frac{v^\epsilon}{u_s}\right)\left(u_s\partial_x v^\epsilon+u^\epsilon\partial_x v_s+v_s\partial_y v^\epsilon+v^\epsilon\partial_y v_s+\frac{\partial_y p^\epsilon}{\epsilon}-\nu \Delta_\epsilon v^\epsilon\right)\\
&+\iint \epsilon \partial_x\left(\frac{v^\epsilon}{u_s}\right)(h_s\partial_x g^\epsilon+h^\epsilon\partial_x g_s+g_s\partial_y g^\epsilon+g^\epsilon\partial_y g_s)\\
&+\iint \partial_y \left(\frac{g^\epsilon}{u_s}\right)\left(u_s\partial_x h^\epsilon+u^\epsilon\partial_x h_s+v_s\partial_y h^\epsilon+v^\epsilon\partial_y h_s-\kappa \Delta_\epsilon h^\epsilon\right)\\
&-\iint \partial_y \left(\frac{g^\epsilon}{u_s}\right)(h_s\partial_x u^\epsilon+h^\epsilon\partial_x u_s+g_s\partial_y u^\epsilon+g^\epsilon\partial_y u_s)\\
&-\epsilon \iint \partial_x \left(\frac{g^\epsilon}{u_s}\right)\left(u_s\partial_x g^\epsilon+u^\epsilon\partial_x g_s+v_s\partial_y g^\epsilon+v^\epsilon\partial_y g_s-\kappa \Delta_\epsilon g^\epsilon\right)\\
&+\epsilon \iint \partial_x \left(\frac{g^\epsilon}{u_s}\right)(h_s\partial_x v^\epsilon+h^\epsilon\partial_x v_s+g_s\partial_y v^\epsilon+g^\epsilon\partial_y v_s)\\
=&\iint \partial_y \left(\frac{v^\epsilon}{u_s}\right)f_1-\epsilon\iint \partial_x \left(\frac{v^\epsilon}{u_s}\right)f_2\\
&+\iint \partial_y \left(\frac{g^\epsilon}{u_s}\right)f_3-\epsilon\iint \partial_x \left(\frac{g^\epsilon}{u_s}\right)f_4=:F.
\end{aligned}
\end{equation}

Following the similar arguments as in \cite{YGuo2}, we derive that
$$F \lesssim (\Vert f_1\Vert_{L^2}+\sqrt{\epsilon}\Vert f_2\Vert_{L^2})\Vert \nabla_\epsilon v^\epsilon\Vert_{L^2}+(\Vert f_3\Vert_{L^2}+\sqrt{\epsilon}\Vert f_4\Vert_{L^2})\Vert \nabla_\epsilon g^\epsilon\Vert_{L^2}$$
and
\begin{equation}\nonumber
\begin{aligned}
&\iint \partial_y\left(\frac{v^\epsilon}{u_s}\right)\left(u_s\partial_x u^\epsilon+u^\epsilon\partial_x u_s+v_s\partial_y u^\epsilon+v^\epsilon\partial_y u_s+\partial_x p^\epsilon-\nu \Delta_\epsilon u^\epsilon\right)\\
&-\iint \epsilon \partial_x\left(\frac{v^\epsilon}{u_s}\right)\left(u_s\partial_x v^\epsilon+u^\epsilon\partial_x v_s+v_s\partial_y v^\epsilon+v^\epsilon\partial_y v_s+\frac{\partial_y p^\epsilon}{\epsilon}-\nu \Delta_\epsilon v^\epsilon\right)\\
&+\iint \partial_y \left(\frac{g^\epsilon}{u_s}\right)\left(u_s\partial_x h^\epsilon-h^\epsilon\partial_x u_s+v_s\partial_y h^\epsilon-g^\epsilon\partial_y u_s-\kappa \Delta_\epsilon h^\epsilon\right)\\
&-\epsilon \iint \partial_x \left(\frac{g^\epsilon}{u_s}\right)\left(u_s\partial_x g^\epsilon-h^\epsilon\partial_x v_s+v_s\partial_y g^\epsilon-g^\epsilon\partial_y v_s-\kappa \Delta_\epsilon g^\epsilon\right)\\
\lesssim &-\iint |\nabla_\epsilon v^\epsilon |^2-\iint |\nabla_\epsilon g^\epsilon |^2+\Vert \nabla_\epsilon (u^\epsilon,h^\epsilon)\Vert_{L^2}^2+L\Vert \nabla_\epsilon(v^\epsilon,g^\epsilon)\Vert_{L^2}^2\\
&-\int_{x=0} \frac{\epsilon^2(\nu|v^\epsilon_x|^2+\kappa|g^\epsilon_x|^2)}{2u_s}-\epsilon\int_{x=L}\frac{|v^\epsilon_y|^2}{4u_s}
+\iint \frac{|g^\epsilon|^2|\partial_y u_s|^2}{u_s^2}\\
&+\Vert \nabla_\epsilon (u^\epsilon,h^\epsilon)\Vert_{L^2}\Vert \nabla_\epsilon (v^\epsilon,g^\epsilon)\Vert_{L^2}.
\end{aligned}
\end{equation}
It remains to estimate the other terms in LHS of (\ref{5.027}) .

Here we only give the estimates of the terms which are new and different from those in \cite{YGuo2}, and other terms can be estimated by the similar arguments, see also \cite{YGuo2}. Note that $u_s$ has a strictly positivity lower bound, then the new and different terms are estimated as follows:
\begin{equation}\nonumber
\begin{aligned}
\iint \partial_y\left(\frac{v^\epsilon}{u_s}\right)h_s\partial_x h^\epsilon=&\iint \frac{\partial_y v^\epsilon}{u_s}h_s\partial_x h^\epsilon-\iint \frac{\partial_y u_s v^\epsilon}{u_s^2}h_s\partial_x h^\epsilon\\
\lesssim &\left\Vert \frac{h_s}{u_s}\right\Vert_{L^\infty}\Vert \partial_x h^\epsilon\Vert_{L^2}\Vert \partial_yv^\epsilon\Vert_{L^2}\\
&+\Vert y \partial_y u_s\Vert_{L^\infty}\Vert \partial_y v^\epsilon\Vert_{L^2}\Vert \partial_x h^\epsilon\Vert_{L^2}\\
\lesssim &\left\Vert \frac{h_s}{u_s}\right\Vert_{L^\infty}\Vert \partial_y g^\epsilon\Vert_{L^2}\Vert \partial_yv^\epsilon\Vert_{L^2}\\
&+\Vert y \partial_y u_s\Vert_{L^\infty}\Vert \partial_y v^\epsilon\Vert_{L^2}\Vert \partial_y g^\epsilon\Vert_{L^2},
\end{aligned}
\end{equation}

\begin{equation}\nonumber
\begin{aligned}
\iint \partial_y\left(\frac{v^\epsilon}{u_s}\right)g^\epsilon\partial_y h_s=&\iint \frac{\partial_y v^\epsilon}{u_s}g^\epsilon \partial_y h_s-\iint \frac{\partial_y u_s v^\epsilon}{u_s^2}g^\epsilon\partial_y h_s\\
\leq &C\Vert y \partial_y h_s\Vert_{L^\infty}\Vert \partial_y g^\epsilon\Vert_{L^2}\Vert \partial_yv^\epsilon\Vert_{L^2}\\
&+C\Vert y\partial_y u_s\Vert_{L^\infty}\Vert y\partial_y h_s\Vert_{L^\infty}\Vert \partial_y g^\epsilon\Vert_{L^2}\Vert \partial_yv^\epsilon\Vert_{L^2},
\end{aligned}
\end{equation}
\begin{equation}\nonumber
\begin{aligned}
-\epsilon\iint \partial_x\left(\frac{v^\epsilon}{u_s}\right)h_s \partial_x g^\epsilon=&-\epsilon\iint \frac{\partial_x v^\epsilon}{u_s}h_s\partial_x g^\epsilon+\epsilon\iint \frac{\partial_x u_s v^\epsilon}{u_s^2}h_s \partial_x g^\epsilon\\
\leq &C\epsilon\left\Vert \frac{h_s}{u_s}\right\Vert_{L^\infty}\Vert \partial_x v^\epsilon\Vert_{L^2}\Vert \partial_x g^\epsilon\Vert_{L^2}\\
&+C\epsilon L |h_s|\Vert \partial_x u_s\Vert_{L^\infty}\Vert \partial_x v^\epsilon\Vert_{L^2}\Vert \partial_x g^\epsilon\Vert_{L^2},
\end{aligned}
\end{equation}
\begin{equation}\nonumber
\begin{aligned}
\iint \partial_y\left(\frac{g^\epsilon}{u_s}\right)h_s \partial_x u^\epsilon=&\iint \frac{\partial_y g^\epsilon}{u_s}h_s\partial_x u^\epsilon-\iint \frac{\partial_y u_s g^\epsilon}{u_s^2}h_s \partial_x u^\epsilon\\
\leq &\left\Vert \frac{h_s}{u_s}\right\Vert_{L^\infty}\Vert \partial_y g^\epsilon \Vert_{L^2}\Vert \partial_x u^\epsilon \Vert_{L^2}\\
&+C|h_s|\Vert y \partial_y u_s\Vert_{L^\infty}\Vert \partial_y g^\epsilon \Vert_{L^2}\Vert \partial_x u^\epsilon \Vert_{L^2}\\
\leq &\left\Vert \frac{h_s}{u_s}\right\Vert_{L^\infty}\Vert \partial_y g^\epsilon \Vert_{L^2}\Vert \partial_y v^\epsilon \Vert_{L^2}\\
&+C|h_s|\Vert y \partial_y u_s\Vert_{L^\infty}\Vert \partial_y g^\epsilon \Vert_{L^2}\Vert \partial_y v^\epsilon \Vert_{L^2},
\end{aligned}
\end{equation}
\begin{equation}\nonumber
\begin{aligned}
\epsilon \iint \partial_x\left(\frac{g^\epsilon}{u_s}\right)h_s \partial_x v^\epsilon=&\epsilon\iint \frac{\partial_x g^\epsilon}{u_s}h_s\partial_x v^\epsilon-\epsilon\iint \frac{\partial_x u_s g^\epsilon}{u_s^2}h_s \partial_x v^\epsilon\\
\leq &C\epsilon \left\Vert \frac{h_s}{u_s}\right\Vert_{L^\infty}\Vert \partial_x v^\epsilon\Vert_{L^2}\Vert \partial_x g^\epsilon\Vert_{L^2}\\
&+C\epsilon L|h_s|\Vert \partial_x u_s\Vert_{L^\infty}\Vert \partial_x v^\epsilon\Vert_{L^2}\Vert \partial_x g^\epsilon\Vert_{L^2},
\end{aligned}
\end{equation}
\begin{equation}\nonumber
\begin{aligned}
\iint\partial_y&\left(\frac{g^\epsilon}{u_s}\right)v^\epsilon\partial_yh_s=\iint\frac{u_s\partial_y g^\epsilon-g^\epsilon \partial_y u_s}{u_s^2}v^\epsilon \partial_y h_s\\
\leq &C(\Vert u_s\Vert_{L^\infty}+\Vert y\partial_y u_s\Vert_{L^\infty})\Vert y\partial_y h_s\Vert_{L^\infty}\Vert \partial_y g^\epsilon\Vert_{L^2}\Vert \partial_y v^\epsilon\Vert_{L^2},
\end{aligned}
\end{equation}
\begin{equation}\nonumber
\begin{aligned}
\iint \frac{|g^\epsilon|^2|\partial_y u_s|^2}{u_s^2}\lesssim \Vert y\partial_y u_s\Vert_{L^\infty}^2\Vert \partial_y g^\epsilon\Vert_{L^2}^2.
\end{aligned}
\end{equation}
Combining the above estimates together, using the smallness of the $\left\Vert \frac{h_s}{u_s}\right\Vert_{L^\infty}$ and $\Vert y\partial_y (u_s,h_s)\Vert_{L^\infty}$, then the desired estimate (\ref{5.026}) follows.
\end{proof}

\smallskip
{\bf Acknowledgment.}

The authors thank to Professor Yan Guo for valuable discussions and suggestions. S. Ding  was supported by the National Natural Science Foundation of China (No.11371152, No.11571117, No.11871005 and No.11771155) and Guangdong Provincial Natural Science Foundation (No.2017A030313003). F. Xie was partially supported by National Natural Science Foundation of China (Grant No.11571231, 11831003). Part of this work was conducted during Z. Lin visited the School of Mathematical Science at Shanghai Jiao Tong University. Z. Lin would like to thank the school of mathematical sciences for the hospitality.

\bigskip


\begin{thebibliography}{99}

\bibitem{Alexander} R. Alexander, Y. Wang, C. Xu, T. Yang, Well posedness of the Prandtl equation in
Sobolev spaces, J. Amer. Math. Soc. 28(2014) 745-784.



\bibitem{Gerard1}D. G\'{e}rard-Varet, E. Dormy, On the ill-posedness of the Prandtl equations, J. Amer.
Math. Soc. 23(2010) 591-609.

\bibitem{G-P}
D. G\'erard-Varet,  M. Prestipino, Formal Derivation and Stability Analysis of Boundary Layer Models in MHD. Z. Angew. Math. Phys. (2017) 68:76.

\bibitem{Gerard2}D. G\'{e}rard-Varet, Y. Maekawa, Sobolev stability of Prandtl expansions for the steady
Navier-Stokes equations, Arch. Ration. Mech. Anal. 233(3) (2019)  1319-1382.

\bibitem{GMM}
D. G\'erard-Varet, Y. Maekawa, N. Masmoudi, Gevrey stability of Prandtl expansions for 2D Navier-Stokes flows.  Duke Math. J. 167(13) (2018) 2531-2631.



\bibitem{GT}
D. Gilbarg, N.S. Trudinger, Elliptic partial differential equations of second order. Second edition. Grundlehren der Mathematischen Wissenschaften [Fundamental Principles of Mathematical Sciences], 224. Springer-Verlag, Berlin, 1983. xiii+513 pp. ISBN: 3-540-13025-X

\bibitem{Grinier}E. Grenier, On the nonlinear instability of Euler and Prandtl equations, Commun.Pure
Appl.Math. 53(9) (2000) 1067-1091.

\bibitem{YGuo1}Y. Guo, S. Iyer, Validity of Steady Prandtl Layer Expansions, preprint (2018) arXiv:1805.05891.

\bibitem{YGuo2}Y. Guo, T. Nguyen, Prandtl boundary layer expansions of steady Navier-Stokes
os over a moving plate, Ann. PDE. 3(10)(2017) 1-58.


\bibitem{Iyer1}S. Iyer, Steady Prandtl boundary layer expansions over a rotating disk, Arch. Ration.
Mech. Anal. 224(2)(2017) 421-469.

\bibitem{Iyer2}
S. Iyer, Global Steady Prandtl Expansion over a Moving Boundary I, Peking Math J. 2(2)(2019) 155-238.

\bibitem{Iyer3}
S. Iyer, Global Steady Prandtl Expansion over a Moving Boundary II,  Peking Math J. (2019) https://doi.org/10.1007/s42543-019-00014-1.

\bibitem{Iyer4}
S. Iyer, Global Steady Prandtl Expansion over a Moving Boundary III, Peking Math J. (2019) https://doi.org/10.1007/s42543-019-00015-0.

\bibitem{Li}
Q. Li, S. Ding, Symmetrical Prandtl boundary layer expansions of steady Navier-Stokes equations on bounded domain, J. Diff. Equn (2019), https://doi.org/10.1016/j.jde.2019.09.030.


\bibitem{CLiu}C. Liu, Y. Wang, T. Yang, A well-posedness theory for the Prandtl equations in three
space variables, Adv. Math. 308 (2017) 1074-1126.

\bibitem{CLiu1}C. Liu, Y. Wang, T. Yang, On the ill-posedness of the Prandtl equations in three space
dimensions, Arch. Ration. Mech. Anal. 220 (2016) 83-108.

\bibitem{CLiu2} C. Liu, F. Xie, T. Yang, MHD boundary layers theory in Sobolev spaces without monotonicity.
I. well-posedness theory, Comm. Pure and Appl. Math. 72(1) (2019) 63-121.

\bibitem{CLiu3} C. Liu, F. Xie, T. Yang, Justification of Prandtl ansatz for MHD boundary layer,  SIAM J. Math. Anal. 51(3) (2019) 2748-2791.


\bibitem{Maekawa}Y. Maekawa, On the inviscid limit problem of the vorticity equations for viscous incompressible
flows in the half-plane, Commun. Pure Appl. Math. 67(2014) 1045-1128.

\bibitem{Masmoudi} N. Masmoudi, T. K. Wong, Local-in-time existence and uniqueness of solutions to the
Prandtl equations by energy methods, Commun.Pure Appl.Math. 68(2015) 1683-1741.


\bibitem{Oleinik}O. A. Oleinik, The Prandtl system of equations in boundary layer theory, Dokl. Akad.
Nauk SSR (1963) 585-586.

\bibitem{Oleinik1}O. A. Oleinik, V. N. Samokhin, Mathematical models in boundary layers theory, Chapman
and Hall/CRC, 1999.

\bibitem{Sammartino1}
M. Sammartino, R.E. Caflisch, Zero viscosity limit for analytic solutions of the Navier-Stokes equation on a half space, I. Existence for Euler and Prandtl
equations, Comm. Math. Phys. 192 (1998) 433-461.

\bibitem{Sammartino2}
M. Sammartino, R.E. Caflisch, Zero viscosity limit for analytic solutions of the Navier-Stokes equation on a half space, II. Construction of the Navier-tokes solution, Comm. Math. Phys. 192 (1998) 463-491.



\bibitem{Prandtl} L. Prandtl, \"{U}ber
fl\"{u}ssigkeitsbewegungen bei sehr kleiner reibung, Verhaldlg III Int.
Math. Kong, (1905) 484-491.



\bibitem{Wang} C. Wang, Y. Wang, Z. Zhang, Zero-viscosity limit of the Navier-Stokes equations in the
analytic setting, Arch. Ration. Mech. Anal. 224(2017) 555-595.

\bibitem{JWang}
J. Wang, S. Ma, On the steady Prandtl type equations with magnetic effects arising from 2D incompressible MHD equations in a half plane, J.  Mathematical Physics 59(12) (2018): 121508.


\bibitem{X-Y2}
F. Xie, T. Yang,  Lifespan of Solutions to MHD Boundary Layer Equations with Analytic Perturbation of General Shear Flow. Acta Math. Appl. Sin. Engl. Ser. 5(1) (2019) 209-229.



\bibitem{Xin2}Z. Xin, L. Zhang, On the global existence of solutions to the Prandtl system, Adv. Math.
181(2004) 88-133.

\end{thebibliography}
\end{document}